\documentclass[12pt]{amsart}
\usepackage{amssymb}
\usepackage{mathrsfs}
\usepackage{graphicx}
\usepackage{color}
\renewcommand\a{\alpha}
\renewcommand\b{\beta}
\newcommand\g{\gamma}
\renewcommand\d{\delta}
\newcommand\la{\lambda}

\newcommand\s{\sigma}

\newcommand\vf{\varphi}

\renewcommand\t{\tau}

\newcommand\Om{\Omega}

\newcommand\D{\Delta}

\newcommand\vL{\varLambda}

\newcommand\vG{\varGamma}
\newcommand\ve{\varepsilon}

\newcommand\ka{\kappa }

\newcommand{\QQ}{\mathbb Q}

\newcommand{\ZZ}{\mathbb Z}
\newcommand{\NN}{\mathbb N}
\newcommand{\CC}{\mathbb C}

\newcommand\Bp{\mathbf p}

\newcommand\Bi{\mathbf{i}}

\newcommand\CA{\mathcal{A}}
\newcommand\CB{\mathcal{B}}

\newcommand\CI{\mathcal{I}}

\newcommand\CS{\mathcal{S}}

\newcommand\CO{\mathcal{O}}
\newcommand\CK{\mathcal{K}}

\newcommand\CF{\mathcal{F}}

\newcommand\CT{ \mathcal{T}}
\newcommand\CZ{ \mathcal{Z}}

\newcommand\FS{\mathfrak S}

\newcommand\Fg{\mathfrak g}

\newcommand\Ft{\mathfrak t}

\newcommand\Fgl{\mathfrak{gl}}

\newcommand\SD{\mathscr{D}}

\newcommand\wh{\widehat}
\newcommand\wt{\widetilde}

\newcommand\ol{\overline}

\newcommand{\lan}{\langle}
\newcommand{\ran}{\rangle}

\newcommand{\ra}{\rightarrow }

\newcommand\Ker{\operatorname{Ker}}
\newcommand\Hom{\operatorname{Hom}}
\newcommand\End{\operatorname{End}}

\newcommand\id{\operatorname{id}}

\newcommand\Rad{\operatorname{Rad}}

\newcommand{\rad}{\operatorname{rad}}
\newcommand{\corank}{\operatorname{corank}}

\newcommand{\isom}{\,\raise2pt\hbox{$\underrightarrow{\sim}$}\,}
\newcounter{ichi}
\newcommand{\roi}{\roman{ichi}}
\newcounter{ni}
\newcommand{\roii}{\roman{ni}}
\newcounter{san}
\newcommand{\roiii}{\roman{san}}
\newcounter{yon}
\newcommand{\roiv}{\roman{yon}}
\newcounter{go}
\newcommand{\rov}{\roman{go}}
\newcounter{roku}
\newcounter{nana}
\newcommand{\Sc}{\mathscr{S}}

\newcommand{\He}{\mathscr{H}}

\newtheorem{thm}{Theorem}[section]
\newtheorem{lem}[thm]{Lemma}
\newtheorem{cor}[thm]{Corollary}
\newtheorem{prop}[thm]{Proposition}

\def \para{\refstepcounter{thm} \par\medskip\noindent
                \textbf{\thethm .} }

\def \remark{\refstepcounter{thm} \par\medskip\noindent
                \textbf{Remark \thethm .} }

\def \remarks{\refstepcounter{thm} \par\medskip\noindent
                \textbf{Remarks \thethm .} }

\allowdisplaybreaks[4]
\numberwithin{equation}{thm}

\begin{document}
\setlength{\baselineskip}{4.9mm}
\setlength{\abovedisplayskip}{4.5mm}
\setlength{\belowdisplayskip}{4.5mm}
\renewcommand{\theenumi}{\roman{enumi}}
\renewcommand{\labelenumi}{(\theenumi)}
\renewcommand{\thefootnote}{\fnsymbol{footnote}}
\renewcommand{\thefootnote}{\fnsymbol{footnote}}
\parindent=20pt
\newcommand{\dis}{\displaystyle}

\medskip
\begin{center}
{\large \bf Presenting cyclotomic $q$-Schur algebras} 
\\
\vspace{1cm}
Kentaro Wada\footnote{This research was  supported  by JSPS Research Fellowships for Young Scientists} 
\address{Graduate School of Mathematics Nagoya University, 
	Furocho, Chikusaku, Nagoya, Japan 464-8602}
\email{kentaro-wada@math.nagoya-u.ac.jp} 

\end{center}
\title{}
\maketitle
\markboth{Kentaro Wada}{Presenting cyclotomic $q$-Schur algebras}

\begin{abstract}
We give a presentation of cyclotomic $q$-Schur algebras 
by generators and defining relations. 
As an application, 
we give an algorithm 
for computing  decomposition numbers of cyclotomic $q$-Schur algebras. 
\end{abstract}

\setcounter{section}{-1}
\section{Introduction}
Let 
$\He_{n,r}$ 
be an Ariki-Koike algebra 
associated to 
a complex reflection group  
$\FS_n \ltimes (\ZZ / r \ZZ)^n$.  
A cyclotomic $q$-Schur algebra
$\Sc_{n,r}$ 
associated to 
$\He_{n,r}$, 
introduced in \cite{DJM98}, 
is defined as an endomorphism algebra of 
a certain $\He_{n,r}$-module. 
In this paper, 
we give a presentation of cyclotomic $q$-Schur algebras 
by generators and defining relations. 

In the case where 
$r=1$, 
$\He_{n,1}$ 
is  the Iwahori-Hecke algebra of the symmetric group 
$\FS_n$, 
and 
$\Sc_{n,1}$ is the $q$-Schur algbera of type $A$. 
In this case, 
$\Sc_{n,1}$ 
is realized as a quotient algebra of 
the quantum group $U_q=U_q(\Fgl_m)$ 
via the Schur-Weyl duality between $\He_{n,1}$ and $U_q$ in \cite{J}. 
We remark that the Schur-Weyl duality holds not only over $\QQ(q)$ but also over $\ZZ[q,q^{-1}]$ (see \cite{Du}).  
By using the surjection from $U_q$ to $\Sc_{n,1}$, 
Doty and Giaquinto  gave a presentation of $\Sc_{n,1}$ 
by generators and defining relations in \cite{DG}. 
They also gave a presentation of $\Sc_{n,1}$ 
in the way compatible with Lusztig's modified form of $U_q$. 
After that, 
Doty realized in \cite{Do} 
the generalized $q$-Schur algebra (in the sense of Donkin) 
as a quotient algebra of a quantum group (also Lusztig's modified form) 
associated to any Cartan matrix of finite type.  

In the case where 
$r >1$, 
a Schur-Weyl duality 
between 
$\He_{n,r}$ and $U_q(\Fg)$ over $\CK=\QQ(q,\g_1,\cdots,\g_r)$ 
was obtained by Sakamoto and Shoji in \cite{SakS}, 
where $\Fg=\Fgl_{m_1}\oplus \cdots \oplus \Fgl_{m_r}$ is a 
Levi subalgebra of 
a parabolic subalgebra of 
$\Fgl_m$. 
However, 
this Schur-Weyl duality does not hold over $\ZZ[q,q^{-1},\g_1,\cdots,\g_r]$. 
In fact, 
Sakamoto-Shoji's Schur-Weyl duality 
should be understood as a Schur-Weyl duality 
between modified Aliki-Koike algebra $\He_{n,r}^0$ 
introduced in \cite{S1}, and $U_q(\Fg)$  
rather than the duality 
between 
$\He_{n,r}$ and $U_q(\Fg)$. 
The image of $U_q(\Fg)$ in the Schur-Weyl duality 
is isomorphic to 
the modified cyclotomic $q$-Schur algebra 
$\ol{\Sc}_{n,r}^0$ 
associated to $\He_{n,r}^0$ 
introduced in \cite{SawS}. 
$\He_{n,r}^0$ and $\ol{\Sc}_{n,r}^0$ 
are defined over any integral domain $R$ 
with parameters satisfying certain conditions. 
In particular, 
we have $\He_{n,r} \cong \He_{n,r}^0$ over $\CK$ 
though 
$\ol{\Sc}_{n,r}^0 \not\cong \Sc_{n,r}$.
(Note that $\He_{n,r} \not\cong \He_{n,r}^0$ over $R$ in general.) 
Some relations between $\Sc_{n,r}$ and $\ol{\Sc}^0_{n,r}$ were studied in \cite{SawS} and \cite{Saw}.  
They showed that 
$\ol{\Sc}^0_{n,r}$ turns out to be a subquotient algebra of $\Sc_{n,r}$, 
and 
$\ol{\Sc}^0_{n,r} \cong \bigoplus_{(n_1,\cdots,n_r) \atop n_1+\cdots+n_r=n} 
	\Sc_{n_1,1} \otimes \cdots \otimes \Sc_{n_r,1}$, 
where 
each component $\Sc_{n_k,1}$ 
is a $q$-Schur algebra of type A which is a quotient algebra of 
the corresponding Levi component $U_q(\Fgl_{m_k})$ of $U_q(\Fgl_m)$. 

In \cite{SW}, 
we have  generalized the results in \cite{SawS} and \cite{Saw} as follows.  
Let 
$\Bp=(r_1,\cdots,r_g) \in \ZZ_{>0}^g$ 
be such that 
$r_1+ \cdots + r_g=r$. 
We define a subquotient algebra 
$\ol{\Sc}_{n,r}^\Bp$ of $\Sc_{n,r}$ 
with respect  to $\Bp$ 
by using a cellular basis of $\Sc_{n,r}$ given in \cite{DJM98}. 
Then we have 
$\ol{\Sc}_{n,r}^\Bp \cong \bigoplus_{(n_1,\cdots,n_g) \atop n_1+\cdots+n_g=n} 
	\Sc_{n_1,r_1} \otimes \cdots \otimes \Sc_{n_g,r_g}$.
The case of $\Bp=(1,\cdots,1)$ is the one discussed in \cite{SawS}, 
and 
$\ol{\Sc}_{n,r}^{(r)}$ (the case of $\Bp=(r)$) is just $\Sc_{n,r}$. 
These structures suggest us that 
$\ol{\Sc}_{n,r}^{\Bp}$ 
is a quotient algebra of 
a certain algebra $\wt{U}_q(\Fg^\Bp)$ 
with respect to the Levi subalgebra 
$\Fg^\Bp=\Fgl_{m_1+\cdots+m_{r_1}} \oplus \cdots \oplus \Fgl_{m_{r_1+\cdots +r_{g-1}+1} + \cdots +m_r}$ 
of $\Fgl_m$. 
In particular, 
$\Sc_{n,r}$ should be  a quotient algebra of 
a certain algebra $\wt{U}_q(\Fgl_m)$. 
(Note that $\wt{U}_q(\Fgl_m)$ (also $\wt{U}_q(\Fg^\Bp)$) is not a quantum group.)  
This is a motivation in this paper.

On the other hand, 
in \cite{DR},
Du and Rui  
defined (upper and lower) Borel subalgebras $\Sc_{n,r}^{\geq 0}$ and $\Sc_{n,r}^{\leq 0}$ of $\Sc_{n,r}$, 
and 
they showed that 
$\Sc_{n,r}=\Sc_{n,r}^{\leq 0} \cdot \Sc_{n,r}^{\geq 0}$. 
Moreover, 
they showed that 
the Borel subalgebra $\Sc_{n,r}^{\geq 0}$ (resp. $\Sc_{n,r}^{\leq 0}$) 
is isomorphic to the Borel subalgebra $\Sc_{m,1}^{\geq 0}$ (resp. $\Sc_{m,1}^{\leq 0}$) 
of a $q$-Schur algebra  $\Sc_{m,1}$ of type A with an appropriate rank. 
In  fact, 
the Borel subalgebra $\Sc_{m,1}^{\geq 0}$ (resp. $\Sc_{m,1}^{\leq 0}$) of $\Sc_{m,1}$ 
is a quotient algebra of an upper (resp. lower) Borel subalgebra of $U_q(\Fgl_m)$. 
These structures imply that 
$\Sc_{n,r}$ is presented by  generators of $U_q(\Fgl_m)$  
with certain defining relations which are different from the defining relations of $U_q(\Fgl_m)$. 
This is a main idea to find  presentations of $\Sc_{n,r}$ by generators and relations. 

This paper is organized as follows.
In \S 1, 
we introduce a certain  algebra 
$\wt{U}_q=\wt{U}_q(\Fgl_m)$ 
associated to 
the Cartan data of $\Fgl_m$. 
A quantum group $U_q(\Fgl_m)$ turns out to be a quotient algebra of $\wt{U}_q$. 
We also prepare several notions for representations of $\wt{U}_q$ 
similar to  the case of quantum groups, 
e.g. weight modules, highest weight modules and  Verma modules. 
In \S 2, 
we define a (various) finite dimensional quotient algebra $\CS_q$ of $\wt{U}_q$. 
This construction of $\CS_q$ was inspired by the construction of generalized $q$-Schur algebra in \cite{Do}. 
In fact, 
both of 
a $q$-Schur algebra $\Sc_{n,1}$ of type A  
and 
a cyclotomic $q$-Schr algbera $\Sc_{n,r}$ 
are 
examples of these finite dimentional quotient algebras of $\wt{U}_q$.  
We also give a method to study representations of $\CS_q$ 
analogous to 
the theory of cellular algebras in \cite{GL96}. 
In some cases, 
$\CS_q$ turns out to be a quasi-hereditary cellular algebra. 
In \S3, 
we develop an argument of specialization of $\CS_q$ to an arbitrary ring and parameters 
by taking divided powers.  
We remark that 
the arguments in \S1-\S3 
can be applied to any Cartan matrix of finite type. 
(See Remarks \ref{remarks-cartan} (\roii).)

After reviews for known results on 
$q$-Schur algebras and cyclotomic $q$-Schur algebras in \S 4 and \S 5, 
we define a surjective homomorphism $\wt{\rho}$ from $\wt{U}_q$ to $\Sc_{n,r}$ in \S6. 
By using the surjection $\wt{\rho}$ 
combined with the results in \S1-\S3, 
we give two presentations of $\Sc_{n,r}$ in \S7 
(Theorem \ref{main-thm}).

Finally, 
we give an algorithm to compute the decomposition numbers of cyclotomic $q$-Schur algebras 
in \S 8. 
\\

\textbf{ Acknowledgments :}
The part of this work was done during the author's stay in the University of Sydney.  
He is grateful to Professors G. Lehrer and A. Mathas for their hospitality. 
In particular, 
the author thanks A. Mathas for many helpful advices and discussions for this work. 
He is also grateful to Professor T. Shoji for several advices, 
in particular for having taught 
the content of Proposition \ref{prop-rho}. 
Thanks are also due to Professors S. Ariki and H. Miyachi 
for many advices and encouragements. 

\tableofcontents

\section{Algebra $\wt{U}_q$}
\para 
Let 
$P=\bigoplus_{i=1}^m \ZZ \ve_i$ 
be a weight lattice of $\Fgl_m$, 
and 
$P^{\vee}=\bigoplus_{i=1}^m \ZZ h_i$  
be the dual weight lattice 
with the natural pairing 
$\lan \, , \, \ran : P \times P^{\vee} \ra \ZZ$ 
such that  
$\lan \ve_i, h_j \ran = \d_{ij}$.
Set $\a_i=\ve_i - \ve_{i+1}$ for $i=1,\cdots,m-1$, 
then 
$\Pi=\{\a_i\,|\, 1\leq i \leq m-1\}$ 
is a set of simple roots, 
and 
$Q=\bigoplus_{i=1}^{m-1} \ZZ\, \a_i$ 
is a root lattice of $\Fgl_m$. 
Put 
$Q^+ = \bigoplus_{i=1}^{m-1} \ZZ_{\geq 0}\, \a_i$. 
We define a partial order 
\lq\lq \,$ \geq$ "
on $P$ 
by 
$\la \geq \mu $ if $\la - \mu \in Q^+$.

\para 
A quantum group 
$U_q=U_q(\Fgl_m)$ 
is the associative algebra over 
$\QQ(q)$, 
where $q$ is an indeterminate,  
with $1$ 
generated by 
$e_i, f_i$ $(1 \leq i \leq m-1)$ 
and 
$K^{\pm}_i$ $(1 \leq i \leq m)$ 
with the following defining relations 
(we denote $K_i^+$ by $K_i$ simply) : 
\begin{align}
&K_iK_j=K_jK_i, \quad K_i K_i^-=K_i^-K_i=1  \label{gl-1} \\
&K_i e_j K_i^- = q^{\lan \a_j , h_i \ran} e_j \label{gl-2}\\
&K_i f_j K_i^- = q^{- \lan \a_j , h_i \ran} f_j  \label{gl-3}\\
&e_i f_j - f_j e_i = \d_{ij}  \frac{K_i K_{i+1}^- - K_i^- K_{i+1}}{q- q^{-1}} \label{gl-4}\\
&e_{i \pm 1}e_i^2 - (q+q^{-1}) e_i e_{i\pm1} e_i + e_i^2 e_{i \pm 1}=0 \label{gl-5}\\
& e_i e_j= e_j e_i \qquad (|i-j| \geq 2) \notag \\
&f_{i \pm 1}f_i^2 - (q+q^{-1}) f_i f_{i \pm 1} f_i + f^2_i f_{i \pm 1}=0  \label{gl-6}\\*
&f_i f_j= f_j f_i \qquad ( |i-j|\geq 2)  \notag 
\end{align}

Let 
$U_q^{+}$ (resp. $U_q^{-}$) 
be the subalgebra of 
$U_q$ 
generated by 
$e_i$ 
(resp. $f_i$) 
for $i=1,\cdots,m-1$,
and 
$U_q^0$ 
be the subalgebra of 
$U_q$ 
generated by 
$K_i^{\pm}$ 
for 
$i=1,\cdots,m$. 
It is well known that 
$U_q$ has the triangular decomposition 
\[ U_q \cong U_q^- \otimes U_q^0 \otimes U_q^+  \text{ as vector spaces}.\]
Let 
$\CB^+$ (resp. $\CB^-$) 
be the subalgebra of 
$U_q$ 
generated by 
$e_i$ (resp. $f_i$) for $ 1 \leq i \leq m-1 $ and $K^{\pm}_i$ for $1 \leq i \leq m$. 
We call 
$\CB^{\pm}$ 
a Borel subalgebra of $U_q$. 
The following lemma is well known. 

\begin{lem}\label{lem-B}\

\begin{enumerate}
\item 
$U_q^+$ (resp. $U_q^-$)  
is isomorphic to 
the algebra defined by 
generators 
$e_i$ (resp. $f_i $) $(1 \leq i \leq m-1)$ 
with a defining relation 
$(\ref{gl-5})$ (resp. $(\ref{gl-6})$). 

\item 
$U_q^0$ is isomorphic to 
$\QQ(q) [K_1^{\pm},\cdots,K_m^{\pm}]$. 

\item 
$\CB^+$ 
is isomorphic to 
the algebra defined by 
generators 
$e_i $ $(1 \leq i \leq m-1)$ and $K^{\pm}_i$ ($1 \leq i \leq m)$   
with defining relations 
$(\ref{gl-1}),(\ref{gl-2})$ and $(\ref{gl-5})$ 
\item 
$\CB^-$ 
is isomorphic to the algebra defined by 
generators 
$f_i $ $(1\leq i \leq m-1)$ and $K^{\pm}_i$ ($1 \leq i \leq m)$ 
with defining relation 
$(\ref{gl-1}), (\ref{gl-3})$ and $(\ref{gl-6})$. 

\end{enumerate}
\end{lem}

\para \label{divided}
Put 
$\CZ=\ZZ[q,q^{-1}]$. 
We define the $\CZ$-form of $U_q$ as follows. 
For any integer $k\in \ZZ$, 
put 
\[ [k]=\frac{q^{k}-q^{-k}}{q-q^{-1}}.\] 
For any positive integer $t \in \ZZ_{> 0}$, 
put $[t]!=[t][t-1]\cdots[1]$ 
and set  
$[0]!=1$. 
For any integer $k$ and any positive integer $t$, 
put 
\[ \left[ \begin{matrix} k \\ t \end{matrix} \right] 
	= \frac{[k][k-1] \cdots [k-t+1]}{[t][t-1] \cdots [1]} 
	= \frac{[k]!}{[t]![k-t]!}. \] 
For 
$k\in \ZZ_{\geq 0}$ and $i=1,\cdots,m-1$, 
put 
\[ e_i^{(k)}=\frac{e_i^k}{[k]!}, \quad 
	f_i^{(k)}=\frac{f_i^k}{[k]!}.
\]
For 
$t \in \ZZ_{\geq 0}$, $c \in \ZZ$ and $i=1,\cdots,m$, 
put 
\[
	\left[\begin{matrix} K_i ; c \\ t \end{matrix} \right] 
	=\prod_{s=1}^t \frac{K_iq^{c-s+1} - K_i^{-1}q^{-c+s-1}}{q^{s}-q^{-s}}.
\]
Let $_\CZ U_q$ be the $\CZ$-subalgebra 
of $U_q$ 
generated by all 
$e_i^{(k)},f_i^{(k)},K_i^{\pm}$ and $\left[ \begin{matrix} K_i ;0 \\ t \end{matrix} \right]$.
We also define the 
$\CZ$-subalgebra $_\CZ \CB^+$ (resp. $_\CZ \CB^-$) 
of $U_q$ 
generated by 
all $e_i^{(k)}$ (resp. $f_i^{(k)}$), 
$K_i^{\pm}$ and $\left[ \begin{matrix} K_i ; 0 \\ t \end{matrix} \right]$. 

\para 
Let 
$\CA=\CZ[\g_1,\cdots,\g_r]$ 
be the polynomial ring over $\CZ$ 
with indeterminate elements  
$\g_1,\cdots,\g_r$, 
where 
$r$ is an arbitrary non-negative integer 
(put $\CA=\CZ$ when $r=0$), 
and let 
$\CK=\QQ(q,\g_1,\cdots,\g_r)$ 
be the quotient field of $\CA$. 
We define the associative algebra 
$\wt{U}_q=\wt{U}_q(\Fgl_m)$ 
over 
$\CK$ with the unit element $1$ 
by the following generators and defining relations: 
\begin{description}
\item[generators]
$e_i, f_i$ ($1\leq i \leq m-1$), $K_i^{\pm}$ ($1\leq i \leq m$), $\t_i$ ($1\leq i \leq m-1$). 
\item[defining relations]
\begin{align}
&K_iK_j=K_jK_i, \quad K_i K_i^-=K_i^-K_i=1,  \label{U-2} \\
&K_i e_j K_i^- = q^{\lan \a_j , h_i \ran} e_j, \label{U-3}\\
&K_i f_j K_i^- = q^{- \lan \a_j , h_i \ran} f_j,  \label{U-4}\\
&K_i \t_j K_i^- = \t_j,  \label{U-5}\\
&e_i f_j - f_j e_i = \d_{ij}  \t_i \label{U-6}\\
&e_{i \pm 1}e_i^2 - (q+q^{-1}) e_i e_{i \pm 1} e_i + e_i^2 e_{i \pm 1}=0, \label{U-7}\\
& e_i e_j= e_j e_i \qquad (|i-j| \geq 2), \notag \\
&f_{i \pm 1}f_i^2 - (q+q^{-1}) f_i f_{i \pm 1} f_i + f^2_i f_{i \pm 1}=0,  \label{U-8}\\
&f_i f_j= f_j f_i \qquad ( |i-j|\geq 2).  \notag 
\end{align}
\end{description}

Set 
$\deg e_i= \a_i$, 
$\deg f_i= -\a_i$, 
$\deg K_i^{\pm} = 0$ 
and 
$\deg \t_i=0$.
Since 
all the defining relations of $\wt{U}_q$ 
are homogeneous under this degree,  
$\wt{U}_q$ is a 
$Q$-graded algebra, 
and 
$\wt{U}_q$ has the following root space decomposition 
\[\wt{U}_q = \bigoplus_{\a \in Q} \big( \wt{U}_q\big)_\a \,\,, \]
where 
$\big(\wt{U}_q \big)_\a = 
\big\{ u \in \wt{U}_q \bigm| K_i u K_i^- = q^{\lan \a, h_i \ran} u \text{ for }  1 \leq i \leq m \big\}$. 
For $u \in \wt{U}_q$, 
we denote by $\deg(u)=\a$ 
if $u \in (\wt{U}_q)_\a$. 

The following proposition  is clear from definitions. 
\begin{prop} \label{prop-wtU-Ugl}
Let $\wt{I}$ be the two-sided ideal of $\wt{U}_q$ generated by 
\[\t_i - \frac{K_i K_{i+1}^- - K_i^- K_{i+1}}{q- q^{-1}} 
\quad \text{ for }\,\, i=1,\cdots, m-1.\] 
Then  we have the following isomorphism of algebras.  
\[ \wt{U}_q/\wt{I} \cong \CK \otimes_{\QQ(q)} U_q.\]
\end{prop}

\remark 
We note that 
the parameters $\g_1,\cdots,\g_r$ 
do not appear in the definition of $\wt{U}_q$. 
However, 
we will use these parameters 
later when we consider some representations of $\wt{U}_q$ 
or some quotient algebras of $\wt{U}_q$. 
\para 
Let 
$\wt{U}_q^{+}$ (resp. $\wt{U}_q^{-}$) 
be the subalgebra of 
$\wt{U}_q$ 
generated by 
$e_i$ 
(resp. $f_i$) 
for $i=1,\cdots,m-1$,
and 
let 
$\wt{U}_q^0$ 
be the subalgebra of 
$\wt{U}_q$ 
generated by 
$K_i^{\pm}$ 
for 
$i=1,\cdots,m$. 
We also define a Borel subalgebra of
$\wt{U}_q$ as follows.  
Let 
$\wt{\CB}^+$ (resp. $\wt{\CB}^-$) 
be the subalgebra of 
$\wt{U}_q$ 
generated by 
$\wt{U}_q^+$ 
(resp. $\wt{U}_q^-$) 
and 
$\wt{U}_q^0$. 
Lemma \ref{lem-B} and Proposition  \ref{prop-wtU-Ugl} 
imply the following corollary. 

\begin{cor} \label{cor-wtB-B}
There exist the following isomorphisms of algebras. 
\[ \wt{U}_q^{\pm} \cong \CK \otimes_{\QQ(q)} U_q^{\pm}, \quad 
	\wt{U}_q^0 \cong \CK \otimes_{\QQ(q)} U_q^0, \quad 
\wt{\CB}^{\pm} \cong \CK \otimes_{\QQ(q)} \CB^{\pm} .\] 
\end{cor}
\begin{proof}
We only show  an isomorphism for Borel subalgebras. 
Other isomorphisms can be shown in a similar way. 
By Lemma \ref{lem-B}, 
we have a surjective homomorphism of algebras  
$\CK \otimes_{\QQ(q)} \CB^{\pm} \ra \wt{\CB}^{\pm}$. 
On the other hand, 
by restricting the surjection 
$\wt{U}_q \ra \CK \otimes_{\QQ(q)} U_q$ 
in Proposition \ref{prop-wtU-Ugl} to 
$\wt{\CB}^{\pm}$, 
we have a surjection 
$\wt{\CB}^{\pm} \ra \CK \otimes_{\QQ(q)} \CB^{\pm}$. 
Thus, we have 
$\wt{\CB}^+ \cong \CK \otimes_{\QQ(q)} \CB^+$. 
\end{proof}

%

%


\para \label{def O}
For  
$\eta=(\eta_1,\cdots,\eta_{m-1})$ 
such that 
$\eta_i \in \wt{U}_q^- \wt{U}_q^0 \wt{U}_q^+$ 
with  
$\deg (\eta_i)=0$,  
let 
$\wh{\CO}^{\eta}$ 
be the category 
consisting of 
$\wt{U}_q$-modules 
satisfying the following conditions (a) and (b): 
\begin{description}
\item[(a)] 
$M \in \wh{\CO}^{\eta}$ 
has the weight space decomposition 
\[ M= \bigoplus_{\mu \in P} M_\mu, \]
where 
$M_\mu =\big\{ v \in M \bigm| K_i \cdot v =q^{\lan \mu, h_i \ran} v  
\text{ for } 1 \leq  i \leq m \big\}$.  

\item[(b)] 
For 
$M \in \wh{\CO}^{\eta}$ and $i=1,\cdots,m-1$, 
it holds that 
$(\t_i - \eta_i) \cdot M =0.$

\end{description}
Let $\wh{\CO}^{\eta}_{\text{tri}}$ 
be the full subcategory of $\wh{\CO}^{\eta}$ 
satisfying the following additional condition: 
\begin{description}
\item[(c)] 
For each $u \in \wt{U}_q$, 
there exists an element 
$x \in \wt{U}_q^- \wt{U}_q^0 \wt{U}_q^+$ 
such that 
\[ (u-x) \cdot M=0 \quad \text{ for any }M \in \wh{\CO}^{\eta}_{\text{tri}}.\]
\end{description}
By this definitions, 
in $\wh{\CO}^{\eta}_{\text{tri}}$, 
the action of $\wt{U}_q$ has a triangular decomposition. 

Finally, 
let $\CO^\eta$ be the full subcategory of $\wh{\CO}^{\eta}$ 
satisfying the following additional conditions: 
\begin{description}
\item[(d)] 
For any $M \in \CO^\eta$, 
the dimension of $M$ is finite. 

\item[(e)] 
For any $M \in \CO^\eta$, 
we have 
\[ M_\mu = 0 \qquad \text{unless } \mu \in P_{\geq 0},\]
where 
$P_{\geq 0} = \bigoplus_{i=1}^m \ZZ_{\geq 0} \, \ve_i$. 
\end{description}
As is seen later, 
$\CO^\eta$ is a full subcategory of $\wh{\CO}^\eta_{\text{tri}}$. 
Moreover, 
we will construct all simple objects of $\CO^\eta$ 
through some quotient algebras of $\wt{U}_q$ 
(Theorem \ref{thm into Oeta}). 

\remarks

(\roi)  
If 
$\eta_i \in \wt{U}^0_q$ 
for any $i=1,\cdots,m-1$,  
we have 
$\wh{\CO}^{\eta}=\wh{\CO}^{\eta}_{\text{tri}}$. 
 
(\roii) 
Let 
$\wt{I}^{\eta}$ 
be the two-sided ideal of $\wt{U}_q$ generated by 
$(\t_i - \eta_i)$, 
and 
put 
$\wt{U}_q^{\eta} =\wt{U}_q/\wt{I}^{\eta}$. 
Then, 
we can regard a $\wt{U}_q^{\eta}$-module 
as a $\wt{U}_q$-module 
through the natural surjection. 
Clearly, 
any $\wt{U}_q^\eta$-module 
equipped with the weight space decomposition 
is contained in $\wh{\CO}^\eta$. 
On the other hand, 
a $\wt{U}_q$-module $M$ contained in $\wh{\CO}^{\eta}$ 
is regarded as a $\wt{U}_q^\eta$-module 
since we have  that 
$\wt{I}^\eta \cdot M=0$  by the condition (b).  
Thus, 
the category 
$\wh{\CO}^\eta$ 
coincides with the category 
consisting of $\wt{U}_q^\eta$-modules 
which have weight space decompositions.  

(\roiii) 
When $\CK=\QQ(q)$ and $\eta_i=(K_i K_{i+1}^- - K_i^- K_{i+1})/(q- q^{-1}) $ 
for any $i=1,\cdots, m-1$, 
$\wh{\CO}^\eta$  
coincides with 
the category of $U_q$-modules 
having weight space decompositions.


\para \label{def-highest}
Next, 
we introduce a notion of 
highest weight modules. 
Let 
$\eta$ 
be as in 
\ref{def O}. 
We call 
$\wt{U}_q$-module 
$M_\la^{\eta}$ 
a highest weight module of highest weight $\la \in P$ 
associated to $\eta$ 
if 
there exists an element $v_\la \in M_\la^{\eta}$ 
satisfying the following conditions: 
\begin{align}
& u \cdot v_\la=0  &&\text{ for any } u \in \wt{U}_q \text{ such that }\label{hi-1} \\
	&&& \quad \deg(u)=\sum_{i=1}^{m-1}d_i \a_i \text{ with } d_i>0 \text{ for some } i,  \notag \\
& K_i \cdot v_\la =q^{\lan \la,h_i \ran} v_\la  &&\text{ for }i=1,\cdots,m, \label{hi-2}\\
& \wt{U}_q \cdot v_\la = M_\la^{\eta}, \label{hi-3}\\
& (\t_i - \eta_i)  \cdot M_\la^{\eta} =0  &&\text{ for } i=1,\cdots,m-1, \label{hi-4}
\end{align}
We call the above element $v_\la$ 
a highest weight vector of 
$M_\la^{\eta}$. 

\remarks 

(\roi)  
Note that,  
since 
we take 
$\eta_i \in \wt{U}_q^- \wt{U}_q^0 \wt{U}_q^+$ 
such that  
$\deg(\eta_i)=0$, 
$(\ref{hi-1})$, $(\ref{hi-2})$ and $(\ref{hi-4})$ 
imply that 
$\t_i \cdot v_\la \in \CK \cdot v_\la$. 

(\roii) 
A highest weight module 
$M_\la^{\eta}$ 
is contained in 
$\wh{\CO}^{\eta}$. 

(\roiii) 
If a highest weight module $M_\la^{\eta}$ is contained in $\wh{\CO}^{\eta}_{\text{tri}}$, 
we can replace $(\ref{hi-1})$ with  
\begin{align}
e_i \cdot v_\la =0 \quad \text{ for } i=1,\cdots,m-1. 
\end{align}

(\roiv)
For a $\wt{U}_q^{\eta}$-module $M$, 
if there exists an element $v_\la \in M$ for some $\la \in P$ 
satisfying the conditions 
(\ref{hi-1})-(\ref{hi-3}), 
$M$ is a highest weight module of highest weight $\la \in P$ 
associated to $\eta$. 
In particular, 
if $\eta_i=(K_i K_{i+1}^- - K_i^- K_{i+1})/(q- q^{-1}) $ 
for any $i=1,\cdots, m-1$ (namely, $\wt{U}^\eta_q \cong U_q$),  
the definition of a highest weight module in \ref{def-highest} 
coincides with 
the usual definition of a highest weight module of $U_q(\Fgl_m)$.

\begin{lem}
If 
a highest weight module 
$M_\la^{\eta}$ 
is contained in $\wh{\CO}^{\eta}_{\text{tri}}$, 
we have the followings. 
\begin{enumerate}
\item 
The dimension of the weight space $(M_\la^{\eta})_\la$ 
with the highest weight $\la$ 
is equal to $1$. 
\item 
$M_\la^{\eta}$ 
has the unique maximal submodule. 
\end{enumerate}
\end{lem}
\begin{proof}
(\roi) is clear from definitions. 
By (\roi) and $(\ref{hi-3})$, 
a proper $\wt{U}_q$-submodule of 
$M_\la^{\eta}$ 
does not have a weight $\la$. 
Thus, 
the sum of all proper $\wt{U}_q$-submodules of 
$M_\la^{\eta}$ 
does not have the weight $\la$, 
and 
this is the unique maximal submodule of $M_\la^{\eta}$. 
\end{proof}

\remark 
When a highest weight module $M_\la^{\eta}$ 
with 
a highest weight vector $v_\la$  
is \textbf{not} contained in $\wh{\CO}_{\text{tri}}^{\eta}$, 
it may occur that 
$u \cdot v_\la \not\in \CK v_\la$ 
and 
$u \cdot v_\la$ has the weight $\la$ 
for some $u \in \wt{U}_q$ such that 
$\deg (u) =0$.


\para 
Let 
$J_\la^{\eta}$ 
be the left ideal of $\wt{U}_q$ generated by 
\begin{align*}
&u \in \wt{U}_q && \text{ such that } \deg(u)=\sum_{i=1}^{m-1}d_i \a_i \text{ with } d_i>0 \text{ for some } i,  \\
& K_i-q^{\lan \la,h_i \ran}1 &&\text{ for }i=1,\cdots,m, \\
&(\t_i- \eta_i)\cdot u && \text{ for } i=1,\cdots,m-1 \text{ and }u \in \wt{U}_q, \\
\end{align*}
Put 
$V_\la^{\eta}=\wt{U}_q/ J_{\la}^{\eta}$, 
then 
one sees that 
$V_\la^{\eta}$ 
is a highest weight module of a highest weight $\la$ 
associated to $\eta$ 
with a highest weight vector 
$1 + J_\la^{\eta}$.   
We call $V_\la^{\eta}$ 
a Verma module of $\wt{U}_q$. 
We have the following lemma.

\begin{lem}
Any highest weight module $M_\la^{\eta}$ 
of a highest weight $\la$ 
associated to $\eta$ 
is a homomorphic image of 
$V_\la^{\eta}$. 
\end{lem}
\begin{proof}
Let 
$M_\la^{\eta}$ 
be a highest weight module 
of a highest weight $\la$ 
associated to $\eta$ 
with a highest weight vector $v_\la$. 
We regard $\wt{U}_q$ as a $\wt{U}_q$-module 
by left multiplications. 
Then, 
we have a natural surjective homomorphism of $\wt{U}_q$-modules 
$\wt{U}_q \ra M_\la^{\eta}$ 
such that 
$1 \mapsto v_\la$. 
Moreover, 
one can check that 
$J_\la^{\eta}$ 
is included in 
the kernel of this homomorphism. 
Thus, 
this homomorphism induces  
the surjective homomorphism 
from 
$V_\la^{\eta}$ 
to 
$M_\la^{\eta}$. 
\end{proof}

\para 
Finally, 
we consider an $\CA$-form of $\wt{U}_q$ as follows. 
We use the same notations in \ref{divided}. 
Let 
$_\CA \wt{U}_q$ 
be the $\CA$-subalgebra 
of $\wt{U}_q$ 
generated by all 
$e_i^{(k)},f_i^{(k)},K_i^{\pm},  \t_i$ and $\left[ \begin{matrix} K_i; 0 \\ t \end{matrix} \right]$. 
We also define the   
$\CA$-subalgebra $_\CA \wt{\CB}^+$ (resp. $_\CA \wt{\CB}^-$) 
of $\wt{U}_q$ 
generated by 
all $e_i^{(k)}$ (resp. $f_i^{(k)}$), 
$K_i^{\pm}$ and $\left[ \begin{matrix} K_i ; 0 \\ t \end{matrix} \right]$. 
Then, an isomorphism 
$_\CA \wt{\CB}^{\pm} \cong\, \CA \otimes_\CZ \,_\CZ\CB^{\pm}$ 
follows from 
Corollary \ref{cor-wtB-B}.



\section{Algebra $\CS_q$}
\label{Sq} 

Recall that 
$P=\bigoplus_{i=1}^m \ZZ \ve_i$ 
is the weight lattice of $\Fgl_m$. 
We can identify $P$ with 
a set of $m$-tuple of integers 
$\ZZ^m $ 
by the correspondence 
\[P \ni \la = \sum_{i=1}^m \la_i \ve_i \mapsto (\la_1, \cdots, \la_m) \in \ZZ^m.\] 
Under this identification, 
we use the notation $\la=(\la_1,\cdots,\la_m)$ 
for $\la \in P$. 
Let $\vL$ be a finite subset of 
$P_{\geq 0} = \bigoplus_{i=1}^m \ZZ_{\geq 0} \, \ve_i$. 
In this section, 
we consider a certain quotient algebra $S_q= \, S_q(\vL)$ of 
$ \wt{U}_q$ 
with respect to $\vL$. 

\para 
We define the associative algebra 
$ \wt{\CS}_q=\,  \wt{\CS}_q(\vL)$ 
over $\CK$ with 1 
by following generators and defining relations: 
\begin{description}
\item[generators]
$E_i, F_i$ ($1\leq i \leq m-1$), $1_\la$ ($\la \in \vL$), $\t_i^\la $ ($1\leq i \leq m-1, \, \la \in \vL$). 
\item[defining relations]
\begin{align}
&1_\la 1_\mu = \d_{\la \mu} 1_\la, \quad \sum_{\la \in \vL} 1_\la =1, \label{Sq-1} \\
& \t_i^\la 1_\mu = 1_\mu \t_i^\la= \d_{\la\mu} \t_i^\la, \label{Sq-2} \\
&E_i 1_\la= 
	\begin{cases} 
		1_{\la + \a_i} E_i & \text{if } \la+\a_i \in \vL \\ 
		0 & \text{otherwise} 
	\end{cases}, \label{Sq-3}\\
&F_i 1_\la= 
	\begin{cases} 
		1_{\la - \a_i} F_i & \text{if } \la-\a_i \in \vL \\ 
		0 & \text{otherwise} 
	\end{cases}, \label{Sq-4}\\
&1_\la E_i= 
	\begin{cases} 
		E_i 1_{\la - \a_i} & \text{if } \la-\a_i \in \vL \\ 
		0 & \text{otherwise} 
	\end{cases}, \label{Sq-5}\\
&1_\la F_i= 
	\begin{cases} 
		F_i 1_{\la + \a_i} & \text{if } \la+\a_i \in \vL \\ 
		0 & \text{otherwise} 
	\end{cases}, \label{Sq-6}\\
&E_i F_j - F_j E_i = \d_{ij}  
	\Big( 
		\sum_{\la \in \vL} \t_i^\la 
	\Big), \label{Sq-7}\\ 
&E_{i \pm 1}E_i^2 - (q+q^{-1}) E_i E_{i \pm 1} E_i + E_i^2 E_{i \pm 1}=0, \label{Sq-8}\\
& E_i E_j= E_j E_i \qquad (|i-j| \geq 2), \notag \\
&F_{i \pm 1}F_i^2 - (q+q^{-1}) F_i F_{i \pm 1} F_i + F^2_i F_{i \pm 1}=0,  \label{Sq-9}\\
&F_i F_j= F_j F_i \qquad ( |i-j|\geq 2).  \notag 
\end{align}
\end{description}

We can prove the following proposition 
in a similar way as in 
\cite[Proposition 3.4]{Do}. 

\begin{prop} \label{prop-sur-wtUq-wtSq}
There exists a surjective homomorphism of algebras 
\[ \wt{\Psi } : \,  \wt{U}_q \ra \,  \wt{\CS}_q \]
such that 
$\wt{\Psi}(e_i)=E_i,\, 
\wt{\Psi} (f_i)=F_i, \, 
\wt{\Psi} (K_i^{\pm})= \sum_{\la \in \vL}q^{\pm \la_i} 1_\la, \, 
\wt{\Psi} (\t_i)= \sum_{\la \in \vL} \t_i^\la$.
\end{prop} 

\begin{proof}
In order to show that 
$\wt{\Psi}$ is well-defined, 
we should check the defining relations of $\wt{U}_q$ in the images of $\wt{\Psi}$, 
and we see them in direct calculations. 
Note that 
$\t_i^\la= \big(\sum_{\mu \in \vL} \t_i^\mu \big) 1_\la= \wt{\Psi}(\t_i) 1_\la$ 
by (\ref{Sq-2}). 
Thus, 
in order to prove 
that $\wt{\Psi}$ is surjective, 
it is enough to show that 
$1_\la$ ($\la \in \vL$) 
is generated by the image of $K_i$ ($i=1,\cdots,m$). 
This will be proven in Lemma \ref{lem-1-lamda}.

\end{proof}

We define a partial order \lq\lq \, $ \succeq$ " on $P_{\geq 0}$ by 
$\la \succ \mu$  
if $\la \not= \mu$ and $\la_i \geq \mu_i$ for any $i=1,\cdots,m$. 
For $\la=(\la_1,\cdots, \la_m) \in \vL$, 
put 
\begin{align}
\label{def Kla}
 K_\la = 
\left[ \begin{matrix} K_1;0 \\ \la_1 \end{matrix} \right] 
\left[ \begin{matrix} K_2;0 \\ \la_2 \end{matrix} \right] 
\cdots 
\left[ \begin{matrix} K_m;0 \\ \la_m \end{matrix} \right]. 
\end{align}
Then we have the following lemma.

\begin{lem} \label{lem-1-lamda}\
\begin{enumerate} 
\item 
$\wt{\Psi} \left( \left[ \begin{matrix} K_i;0 \\ t \end{matrix} \right] \right) $ 
$(1 \leq i \leq m,\, t \in \ZZ_{\geq 0})$ 
is written as a linear combination of $\big\{1_\la \bigm| \la \in \vL \big\}$ 
with $\CZ$-coefficients. 

\item 
For $\la \in \vL$, 
we have 
\[ 1_\la = \wt{\Psi}(K_\la) + 
	\sum_{\mu \in \vL \atop \mu \succ \la} r_\mu \wt{\Psi}(K_\mu) \qquad (r_\mu \in \CZ).\] 

\end{enumerate}
\end{lem}

\begin{proof}
In this proof, 
we denote $\wt{\Psi}(K_i^{\pm})$ by $K_i^{\pm}$ simply. 
Thus, we have 
$K_i^{\pm}=\sum_{\la \in \vL} q^{\pm \la_i} 1_\la$. 
For $1 \leq i \leq m, t \in \ZZ_{\geq 0}$ and $\la \in \vL$, 
we have 
\begin{align}
\left[ \begin{matrix} K_i ; 0 \\ t \end{matrix} \right] 1_\la 
&= 
\prod_{s=1}^t \frac{K_i q^{-s+1} - K_i^- q^{s-1}}{q^s - q^{-s}} 1_\la 
\label{K-1-lamda} \\
&= 
\prod_{s=1}^t \frac{q^{\la_i - s+1} - q^{-(\la_i -s +1)}}{q^s -q^{-s}} 1_\la \notag \\
&= 
\prod_{s=1}^t \frac{[\la_i - s +1]}{[s]} 1_\la \notag \\
&= 
\frac{[\la_i][\la_i-1]\cdots[\la_i-t+1]}{[1][2]\cdots[t]} 1_\la \notag \\
&= 
\begin{cases}
\left[ \begin{matrix} \la_i \\ t \end{matrix} \right] 1_\la  & \text{ if } t \leq \la_i \\[5mm]
0 & \text{ if } t> \la_i.
\end{cases}\notag
\end{align}
Since $1=\sum_{\la \in \vL} 1_\la$ and $\left[ \begin{matrix} \la_i \\ t \end{matrix} \right] \in \CZ$, 
we have (\roi). 
By the definition of $K_\la$ and (\ref{K-1-lamda}), 
we have 
\begin{align} 
K_\la=K_\la(\sum_{\mu \in \vL} 1_\mu)= 1_\la + \sum_{\mu \in \vL \atop \mu \succ \la} 
\left( \prod_{i=1}^m \left[ \begin{matrix} \mu_i \\ \la_i \end{matrix} \right] 1_\mu \right). 
\label{K expand 1-lamda}
\end{align}
Since $\vL$ is a finite set, 
there exists a maximal element $\la \in \vL$ with respect to the order \lq\lq \,$ \succeq$ ". 
Thus, we have $1_\la = K_\la$ when $\la$ is a maximal element of $\vL$ by (\ref{K expand 1-lamda}).  
By induction on $\vL$ together with (\ref{K expand 1-lamda}), 
we have (\roii). 
\end{proof}

\remark 
For $\la=(\la_1,\cdots,\la_m) \in P_{\geq 0}$, set $|\la|=\sum_{i=1}^m \la_i$. 
If $\vL=\{ \la \in P_{\geq 0} \,|\, |\la|=n \}$ for some $n \in \ZZ_{>0}$, 
we have 
$\mu \not\succ \la$ for any $\la,\mu \in \vL$ 
since $|\mu|>|\la|$ if $\mu \succ \la$. 
Thus, we have  
$1_\la = \wt{\Psi}(K_\la) $ for any $\la \in \vL$ 
by Lemma \ref{lem-1-lamda}. 

%
%

\para \label{def Sq}
Let 
$ \wt{\CS}_q^+$ (resp. $ \wt{\CS}_q^-$) 
be the subalgebra of $ \wt{\CS}_q$ 
generated by 
$E_i $ (resp. $F_i$) 
for $ 1 \leq i \leq m-1$, 
and let 
$ \wt{\CS}_q^0 $ 
be the subalgebra of $ \wt{\CS}_q$ 
generated by 
$1_\la$ for $\la \in \vL$. 
By Lemma \ref{lem-1-lamda}, 
it is clear that $ \wt{\CS}_q^0$  
(resp. $ \wt{\CS}_q^{\pm}$) 
coincides with 
the image of $ \wt{U}_q^0$ 
(resp. $ \wt{U}_q^{\pm}$) 
under the surjection $\wt{\Psi}$ in Proposition \ref{prop-sur-wtUq-wtSq}. 

We consider the $Q$-grading on $ \wt{\CS}_q$ arising from the grading on $ \wt{U}_q$, 
namely 
we set 
$\deg E_i =\a_i, \, 
\deg F_i = -\a_i, \,
\deg 1_\la =0, \,
\deg \t_i^\la=0$. 

For each $\la \in \vL$ and $i=1,\cdots,m-1$, 
we take an element $\eta_i^\la $ of $ \wt{\CS}_q^- \wt{\CS}_q^+ \cdot 1_\la$ 
such that
$\deg (\eta_i^\la) =0$. 
By the condition $\deg (\eta_i^\la)=0$ 
together with (\ref{Sq-3})-(\ref{Sq-6}), 
we have 
$\eta_i^\la \in 1_\la \cdot  \wt{\CS}_q^- \wt{\CS}_q^+ \cdot 1_\la$. 
Moreover, 
again by (\ref{Sq-3})-(\ref{Sq-6}), 
we have 
$\eta_i^\la \in \wt{\CS}_q^- \wt{\CS}_q^0 \wt{\CS}_q^+$. 
Put 
$\eta_{\vL}=\{\eta_i^\la \,|\, 1 \leq i \leq m-1,\, \la \in \vL \}$. 
Let $\wt{\CI}^{\eta_{\vL}}$ be the two-sided ideal of $ \wt{\CS}_q$ 
generated by all 
$\t_i^\la - \eta_i^\la$ ($1 \leq i \leq m-1,\, \la \in \vL$). 
We define the quotient algebra $ \CS_q$ of $ \wt{\CS}_q$ by 
\[ \CS_q = \CS_q^{\eta_\vL} =  \wt{\CS}_q/ \wt{\CI}^{\eta_\vL} .\] 
Let 
$ \CS_q^0$ (resp. $\CS_q^{\pm}$), 
be the image of 
$\wt{\CS}_q^0$ (resp. $\wt{\CS}_q^{\pm}$) 
under the natural surjection $\wt{\CS}_q \ra \CS_q$.
Under the map $\wt{\CS}_q \ra  \CS_q$, 
we denote the image of $E_i$ (resp. $F_i$, $1_\la$) 
by the same symbol 
$E_i$ (resp. $F_i$, $1_\la$) again, 
and 
the image of $\t_i^\la$ by $\eta_i^\la$. 
We denote the composition of 
$\wt{\Psi}$ and the natural surjection 
$\wt{\CS}_q \ra \CS_q$ 
by 
$\Psi : \wt{U}_q \ra \CS_q$. 
Thus, 
we have 
$\Psi(e_i)=E_i$, 
$\Psi(f_i)=F_i$, 
$\Psi(K^{^\pm}_i)=\sum_{\la \in \vL} q^{\pm \la_i} 1_{\la}$ 
and 
$\Psi(\t_i)=\sum_{\la \in \vL} \eta^\la_i$. 
\begin{prop} 
\label{tri decom}
$\CS_q$ has a triangular decomposition 
\[ \CS_q = \CS_q^- \CS_q^0  \CS_q^+ .\] 
Moreover, 
the dimension of 
$ \CS_q$ is finite. 
\end{prop} 

\begin{proof}
First, we show the following claim. 

\textbf{(Claim A)} 
For $1 \leq i, j_1,\cdots, j_l \leq m-1$, we have 
\[ E_i F_{j_1} \cdots F_{j_l}= \sum_{k=1}^{m-1} a_k E_k + b,\]
where 
$a_k \in \CS_q$ and $b \in  \CS_q^- \CS_q^0$.

We prove this claim by induction on $l$. 
When $l=1$, we have 
\[ E_i F_{j_1}=
\begin{cases} 
	F_{j_1} E_i + \sum_{\la \in \vL} \eta_i^\la & \text{ if } i=j_1 \\
	F_{j_1} E_i & \text{ otherwise }.
\end{cases}
\]
Since $\eta_i^\la \in \CS_q^- \CS_q^0 \CS_q^+$, 
we obtain the claim. 
When $l \geq 2$, 
we have 
\[ E_i F_{j_1}\cdots F_{j_l}=
\begin{cases} 
	F_{j_1} E_i F_{j_2} \cdots F_{j_l} 
	+ \Big(\sum_{\la \in \vL} \eta_i^\la\Big) F_{j_2} \cdots F_{j_l} 
		& \text{ if } i=j_1 \\
	F_{j_1} E_i F_{j_2} \cdots F_{j_l}& \text{ otherwise }.
\end{cases}
\]
Note that $\eta_i^\la \in \CS_q^- \CS_q^0 \CS_q^+$ and $\deg(\eta_i^\la)=0$. 
Applying the induction hypothesis to the right hand side of this formula, 
we obtain the claim. 
  
For any $u \in \CS_q$, we have 
$ u = u \cdot 1 =\sum_{\la \in \vL} u \cdot 1_\la.$
Thus, 
in order to prove the first assertion of the proposition, 
we should show that  
\begin{align*} \label{u 1-lamda}
u \cdot 1_\la \in  \CS_q^- \CS_q^+  \cdot 1_\la 
\quad 
\text{ for any } 
u \in \CS_q 
\text{ and }
\la \in \vL. 
\tag{\textbf{Claim B}}
\end{align*} 
This claim 
implies that $u \in \CS_q^- \CS_q^0  \CS_q^+$ 
for any $u \in  \CS_q$
by the relation (\ref{Sq-3}). 
Hence, we show (\ref{u 1-lamda}) by the backword induction on $\vL$ with respect to the order \lq\lq $\geq$". 
By (\textbf{Claim A}) combined with the relations (\ref{Sq-1}) and (\ref{Sq-3})-(\ref{Sq-6}), 
for any $u \in \CS_q $ and $\la \in \vL$, 
we have 
\begin{align}
\label{u1u1u1}
u \cdot 1_\la = \sum_{k=1}^{m-1} a_k E_k 1_\la + b \cdot 1_\la \qquad 
(a_k \in  \CS_q, \,\, b \in \CS_q^-). 
\end{align}
Clearly, 
$b \cdot 1_\la \in \CS_q^-  \CS_q^+  \cdot 1_\la$. 
On the other hand, 
we have $a_k E_k 1_\la= a_k 1_{\la+\a_k} E_k$ by (\ref{Sq-3}), 
where we set $1_{\la + \a_k}=0$ if $\la + \a_k \not\in \vL$.  

First, 
we assume that $\la$ is a maximal element of $\vL$. 
Then, for any $k=1,\cdots, m-1$, 
we have $\la + \a_k \not\in \vL$ 
since $\la + \a_k \geq \la$ in P and $\la$ is maximal in $\vL$. 
Thus, we have 
$1_{\la + \a_k}=0$ 
for $k=1,\cdots,m-1$.
In this case, 
we have 
$u \cdot 1_\la= b \cdot 1_\la \in  \CS_q^-  \CS_q^+  \cdot 1_\la$. 

Next, 
we assume that $\la$ is not maximal in $\vL$, 
and that $\la + \a_k \in \vL$. 
In this case, 
by the induction hypothesis, 
we have $a_k 1_{\la+\a_k} \in \CS_q^-  \CS_q^+  \cdot 1_{\la+\a_k}$. 
Thus we have 
$a_k 1_{\la +\a_k} E_k = a_k E_k 1_{\la} \in \CS_q^-  \CS_q^+  \cdot 1_\la$.
Combined with (\ref{u1u1u1}), 
we obtain (\ref{u 1-lamda}), 
thus the first assertion of the proposition is proven. 

Recall that  
$\CS_q^0$ is the subalgebra of $\CS_q$ generated by $\{1_\la\,|\, \la \in \vL\}$,  
and  
$\{1_\la \not=0 \, |\, \la \in \vL\}$ 
is a set of pairwise orthogonal idempotents. 
Thus,  
$\{1_\la \not=0 \, | \, \la \in \vL\}$ 
gives an $\CK$-basis of $\CS_q^0$.   

On the other hand, 
a set 
$\{ E_{i_1} E_{i_2} \cdots E_{i_l} 
\,|\,1 \leq i_1, \cdots, i_l \leq m-1,  \, l \geq 0 \}$ 
gives a spanning set of $\CS_q^+$ over $\CK$. 
Since 
\begin{align*} 
E_{i_1} \cdots E_{i_l}
&= \sum_{\la \in \vL} \left( E_{i_1} \cdots E_{i_l}  1_\la \right) \\
&= \sum_{\la \in \vL} \left( 1_{\la+ \a_{i_1}+\cdots + \a_{i_l} }E_{i_1} \cdots E_{i_l}  \right), 
\end{align*}
we have 
$E_{i_1} \cdots E_{i_l}=0$ 
if the integer $l$ is sufficient large. 
This implies that 
$\CS_q^+$ is finitely generated over $\CK$. 
Similarly, 
we see that 
$\CS_q^-$ is finitely generated over $\CK$. 
Combined with the triangular decomposition, 
we conclude that $\CS_q$ is finite dimensional. 
\end{proof} 

The following result was proved in the proof of the above proposition. 

\begin{cor}
\label{basis Sq0}
$\{1_\la \not=0 \,|\, \la \in \vL\}$ 
gives a $\CK$-basis of $\CS^0_q$. 
\end{cor}

\para 
For each $\la \in \vL$, 
we define the following subspaces of $\CS_q$ ; 
\begin{align*}
\CS_q(\geq \la) 
	= \big\{ x 1_\mu y \bigm| x \in \CS_q^-,\, y \in \CS_q^+,\, \mu \in \vL \text{ such that }\mu \geq \la \big\}, \\
\CS_q(>\la) 
	= \big\{ x 1_\mu y \bigm| x \in \CS_q^-,\, y \in \CS_q^+,\, \mu \in \vL \text{ such that }\mu > \la \big\}. 
\end{align*}
By using the triangular decomposition 
and the defining relations of $\CS_q$, 
one can easily check the following lemma.

\begin{lem}\label{lem ideal}
For $\la \in \vL$, 
both of 
$\CS_q(\geq \la)$ and $\CS_q(> \la) $ are  two-sided ideals of $\CS_q$. 
\end{lem}

\para 
Thanks to Lemma \ref{lem ideal}, 
for $\la \in \vL$, 
$\CS_q(\geq \la)/\CS_q(>\la)$ 
turns out to be an $(\CS_q,\CS_q)$-bimodule 
by multiplications.  
In general, 
it happens that $\CS_q(\geq \la)=\CS_q(> \la)$. 
So, 
we take a subset 
$\vL^+=\{\la \in \vL \,|\, \CS_q(\geq \la)\not=\CS_q(> \la) \}$ of $\vL$. 
It is clear that 
\begin{align}
\label{vL+}
\la \in \vL^+ \text{ if and only if } 1_\la \not\in \CS_q(>\la).
\end{align}

For $\la \in \vL^+$, 
we define a subspace 
$\D(\la)$ 
of 
$\CS_q (\geq \la)/ \CS_q(>\la)$ 
by 
\[
\D(\la)=\CS_q^- \cdot 1_{\la} + S_q(>\la).
\]
Note that $E_k 1_{\la} = 1_{\la+ \a_k} E_k \in \CS_q(>\la)$ 
for $k=1,\cdots, m-1$, 
together with the triangular decomposition, 
$\D(\la)$ 
turns out to be a left $\CS_q$-submodule of 
$\CS_q (\geq \la)/ \CS_q(>\nobreak\la)$.
Similarly, 
we can define a right $\CS_q$-submodule 
$\D^\sharp (\la)$ of $\CS_q (\geq \la)/ \CS_q(>\la)$ 
by 
\[\D^\sharp (\la) = 1_{\la}\cdot  \CS_q^+ + \CS_q(>\la).\] 
For 
$x \in \CS_q^-, y\in \CS_q^+$, 
we denote 
the coset of $\CS_q (\geq \la) / \CS_q(>\la)$ 
containing 
$x 1_{\la} y$ 
by 
$\ol{x 1_{\la} y}$. 
Then, 
we denote an element of $\D(\la)$ 
\big(resp. $\D^\sharp (\la)$\big) 
by 
$\ol {x 1_{\la}}$ $(x \in \CS_q^-)$ 
\big(resp. $\ol {1_{\la} y}$ $(y \in \CS_q^+)$\big). 
It is clear that 
$\D(\la)=\CS_q \cdot \ol{1_\la}$ and 
$\D^\sharp (\la) =\ol{1_\la} \cdot \CS_q$. 
We can check the following lemma immediately from the definitions.

\begin{lem}\label{lem surjection}
For $\la \in \vL^+$, 
there exists a surjective homomorphism of $(\CS_q,\CS_q)$-bimodules 
\[\D(\la) \otimes_{\CK} \D^\sharp (\la)
	\ra \CS_{q}(\geq \la) / \CS_q(>\la)\]
such that 
$\ol{x 1_{\la}} \otimes \ol{1_{\la} y} \mapsto \ol{x 1_{\la} y}$ \,
for\, $x \in \CS_q^-, \,y\in \CS_q^+$. @  
\end{lem}

\para
As will be seen later, 
if the surjection in Lemma \ref{lem surjection} gives an isomorphism 
for any $\la \in \vL^+$ 
and 
$\CS_q$ has a certain involution $\iota$,  
$\CS_q$ turns out to be a quasi-hereditary cellular algebra, 
and $\D(\la)$ $(\la \in \vL^+)$ is a left cell (standard) module of $\CS_q$.  
In such a case, 
we can apply a general theory of (quasi-hereditary) cellular algebras. 
However, in general, 
we do not know whether   
$\D(\la) \otimes_{\CK} \D^\sharp (\la)$ 
is isomorphic to 
$ \CS_{q}(\geq \la) / \CS_q(>\la)$ 
or not  
(In fact, 
it happens that $\D(\la) \otimes_{\CK} \D^\sharp (\la)$ 
 is not isomorphic to 
$ \CS_{q}(\geq \la) / \CS_q(>\la)$. 
See Appendix \ref{example eta=0}
),  
and 
do not know whether  
$\CS_q$ has such an involution. 
Nevertheless, 
we develop a certain representation theory of $\CS_q$ 
which is almost similar to the theory of standardly based algebras in the sens of 
\cite{DR98}, 
and also similar to the theory of cellular algebras (see e.g. \cite{GL96}, \cite[ch.2] {M-book}).

\para 
For $y \in \CS_q^+ , \, x \in \CS_q^-$ and $\la \in \vL^+$,  
we have 
$1_{\la} y x  1_{\la} = 1_{\la} 1_{\la + \a} y x$ 
if  $\deg ( y x )=\a$.  
Thus, 
we have 
$1_\la yx 1_\la=0$ 
if 
$\deg (yx) =\a \not=0$. 
On the other hand, 
if $\deg \big( y x \big)=0$, 
we can write 
\begin{align}
\label{xy constant}
1_{\la} y x 1_{\la}
	= r_0 1_{\la} 
	+ \sum_{Y \in \CS_q^+, X \in \CS_q^- \atop \deg(Y)=- \deg(X)\not=0} 
	r_{XY} 1_{\la} XY 1_{\la} 
	\quad (r_0, \, r_{XY} \in \CK) 
\end{align}
by investigating the degrees through the triangular decomposition. 
These imply, 
for $y \in \CS_q^+$, $x\in \CS_q^-$ and $\la \in \vL^+$, 
that 
we have 
\[ 1_{\la} y x 1_{\la} 
	\equiv r_{yx} 1_{\la}  \mod \CS_q(>\la) \qquad (r_{yx} \in \CK).   
\]

By using this formula, 
for $\la \in \vL^+$,
we can define a bilinear form 
$\lan \, , \, \ran : \D^\sharp(\la) \times \D(\la) \ra \CK$ 
such that 
\begin{align}
\lan \ol{1_{\la}y} \,,\, \ol{x 1_{\la}} \ran 1_{\la} 
\equiv 
1_{\la} y x  1_{\la} 
\mod \CS_q(>\la) 
\quad
\text{ for } y \in \CS_q^+, \, x \in \CS_q^-. 
\end{align}
For $\a \in Q^+$, 
put 
\[\Upsilon _\a=\big\{(i_1,i_2,\cdots,i_k) \bigm| 1 \leq i_1,i_2,\cdots, i_k \leq m-1 \,\, 
	\text{ such that }
	\a_{i_1}+\a_{i_2}+\cdots+\a_{i_k}=\a \big\}.\] 
From the definition, 
for $(i_1,\cdots, i_k) \in \Upsilon_\a, \,(j_1,\cdots, j_l) \in \Upsilon _\b$ ($\a,\b \in Q^+$), 
we have 
\begin{align}\label{bilinear zero}
\lan \ol{1_{\la} E_{i_1} \cdots E_{i_k}} \,,\, \ol{F_{j_1}\cdots F_{j_l} 1_\la} \ran =0 
\quad \text{ if } \a \not= \b .
\end{align}
We have the following lemma.

\begin{lem} \label{lem property bilinear}
For $\la \in \vL^+$, we have the following formulas.
\begin{enumerate}
\item 
$\lan \ol{y} \cdot u \, , \, \ol{x} \ran = \lan \ol{y} \, , \, u \cdot \ol{x} \ran $\,
for 
$\ol{x} \in \D(\la), \, \ol{y} \in \D^\sharp (\la), \, u \in \CS_q$.

\item 
$(F_{i_1}\cdots F_{i_k} 1_{\la} E_{j_1}\cdots E_{j_l}) \cdot \ol{x} 
= \lan \ol{1_{\la}E_{j_1} \cdots E_{j_l}}\,,\,\ol{x} \ran 
\ol{F_{i_1}\cdots F_{i_k} 1_{\la}}$  \\
for $\ol{x} \in \D(\la)$ and 
$F_{i_1}\cdots F_{i_k} 1_{\la} E_{j_1}\cdots E_{j_l} \in \CS_q(\geq \la)$. 
\end{enumerate}
\end{lem}

\begin{proof}

(\roi) For $x\in \CS_q^-$, $y \in \CS_q^+$ and $u \in \CS_q$, we have 
\begin{align*}
\lan \ol{ 1_{\la} y} \cdot u \, , \, \ol{x 1_{\la}} \ran 1_{\la} 
&\equiv  
1_{\la} yux 1_{\la} \\
&\equiv 
\lan \ol{ 1_{\la} y } \, , \, u \cdot \ol{x 1_{\la}} \ran 1_{\la}
\mod \CS_q(>\la).
\end{align*} 

(\roii) 
For $x \in \CS_q^-$ and 
$F_{i_1}\cdots F_{i_k} 1_{\la} E_{j_1}\cdots E_{j_l} \in \CS_q(\geq\la)$, 
we have 
\begin{align*}
(F_{i_1}\cdots F_{i_k} 1_{\la} E_{j_1}\cdots E_{j_l}) \cdot \ol{x 1_{\la}} 
&= 
\ol{F_{i_1}\cdots F_{i_k} (1_{\la} E_{j_1}\cdots E_{j_l} x 1_{\la})} \\*
&=
\ol{ F_{i_1}\cdots F_{i_k} 
	\lan \ol{1_{\la} E_{j_1} \cdots E_{j_l}} \,,\, \ol{x 1_{\la}} \ran 
	1_{\la}}\\*
&=
\lan \ol{1_{\la} E_{j_1} \cdots E_{j_l}} \,,\, \ol{x 1_{\la}} \ran 
\ol{ F_{i_1}\cdots F_{i_k} 1_{\la}}.
\end{align*}
\end{proof}

\para 
For $\la \in \vL^+$, 
let   
\begin{align*}
&\rad \D(\la)= \big\{ \ol{x} \in \D(\la) \bigm| \lan \ol{y} , \ol{x} \ran =0 
\text{ for any }\ol{y} \in \D^\sharp(\la) \big\}, \\
&\rad \D^\sharp (\la)= \big\{ \ol{y} \in \D^\sharp (\la) \bigm| \lan \ol{y} , \ol{x} \ran =0 
\text{ for any }\ol{x} \in \D (\la) \big\}.
\end{align*}
By Lemma \ref{lem property bilinear} (\roi), 
$\rad \D(\la)$ (resp. $\rad \D^\sharp (\la)$) 
is a left (resp. right) $\CS_q$-submodule of $\D(\la)$ (resp. $\D^\sharp (\la)$). 
Put 
$L(\la)= \D(\la) / \rad \D(\la)$ 
and 
$L^\sharp (\la)= \D^\sharp (\la) / \rad \D^\sharp (\la)$ 
We have the following theorem. 
This theorem is proven in a similar way as in the general theory of standardly based algebras 
or cellular algebras  
(see \cite{DR98}, \cite{GL96}, \cite[Ch.2]{M-book} ). 

\begin{thm}\ \label{thm standard simple Sq}
\begin{enumerate}
\item 
For $\la \in \vL^+$, 
$\rad \D(\la)$ (resp. $\rad \D^\sharp (\la)$)  
is the unique proper maximal $\CS_q$-submodule of 
$ \D(\la)$  (resp. $\D^\sharp (\la)$). 
Thus, 
$L(\la)$ (resp. $L^\sharp (\la)$)
is a  left (resp. right) absolutely simple $\CS_q$-module. 

\item 
For $\la, \mu \in \vL^+$,  
if $L(\mu)$ (resp. $L^\sharp (\mu)$)
is a composition factor of 
$\D(\la)$ (resp. $\D^\sharp (\la)$), 
we have 
$\la \geq \mu$. 
Thus, 
$L(\la) \cong L(\mu)$ (resp. $L^\sharp (\la) \cong L^\sharp (\mu)$) 
if and only if 
$\la=\mu$. 
Moreover, 
the multiplicity of $L(\la)$ (resp. $L^\sharp (\la)$) in $\D(\la)$ (resp. $\D^\sharp (\la)$) is equal to one. 

\item 
$\big\{ L(\la) \bigm| \la \in \vL^+ \big\}$ 
(resp. $\big\{ L^\sharp (\la) \bigm| \la \in \vL^+ \big\}$ ) 
gives a complete set of 
non-isomorphic left (resp. right) simple $\CS_q$-modules. 

\item 
$\CS_{q}$ is semisimple if and only if 
$\D(\la) \cong L(\la)$ and $\D^\sharp (\la) \cong L^\sharp(\la)$ for any $\la \in \vL^+$.   
\end{enumerate}
\end{thm}

\begin{proof}
We prove the assertions only for left $\CS_q$-modules. 
The proof is similar for right $\CS_q$-modules. 
(\roi) 
It is clear that 
$\lan \ol{1_{\la}}, \ol{1_{\la}} \ran =1 $. 
Thus, 
we have $\D(\la) \supsetneqq \rad \D(\la)  $. 
For $\ol{x} \in \D(\la) \setminus \rad \D(\la)$, 
there exists an element 
$\ol{y} \in \D^\sharp(\la)$ 
such that 
$\lan \ol{y}, \ol{x} \ran \not=0$. 
Since $\lan \, , \,\ran$ is a bilinear form over a field $\CK$, 
we can suppose that $\lan \ol{y}, \ol{x} \ran =1$. 
Let 
\[
\ol{y}=\sum_{(j_1,\cdots, j_l)\in \Upsilon _\a \atop  \a \in Q^+} 
	r_{(j_1,\cdots,j_l)} \ol{1_{\la} E_{j_1} \cdots E_{j_l}}.
\] 
For $\ol{t}=\ol{F_{i_1}\cdots F_{i_k} 1_{\la}} \in \D(\la)$, 
put 
\[y_{\ol{t}}=F_{i_1}\cdots F_{i_k} 1_{\la} 
	\Big(\sum_{(j_1,\cdots, j_l)\in \Upsilon_\a \atop  \a \in Q^+} 
	r_{(j_1,\cdots, j_l)} E_{j_l}\cdots E_{j_1}\Big) \in \CS_q.
\] 
Then, we have 
\begin{align*}
y_{\ol{t}}\cdot \ol{x} 
&=
\sum r_{(j_1,\cdots,j_l)} (F_{i_1}\cdots F_{i_k} 1_{\la} E_{j_l}\cdots E_{j_1})\cdot \ol{x} \\
&= 
\sum r_{(j_1,\cdots,j_l)} \lan \ol{1_\la E_{j_1}\cdots E_{j_l} } \, , \,\ol{x} \ran 
	\ol{F_{i_1} \cdots F_{i_k} 1_{\la}} 
	&&(\because \text{Lemma \ref{lem property bilinear} (\roii)}) \\
&=
\lan \ol{y}, \ol{x} \ran \ol{F_{i_1} \cdots F_{i_k} 1_{\la}} \\
&=
\ol{F_{i_1} \cdots F_{i_k} 1_{\la}}.
\end{align*} 
This implies that 
$\D(\la)$ 
is generated by $\ol{x}$ as an $\CS_q$-module. 
Since this fact holds for any $\ol{x} \in \D(\la) \setminus \rad \D(\la)$, 
$\rad \D(\la)$ is the proper unique maximal submodule of $\D(\la)$. 

(\roii) 
For $\la \in \vL^+$, 
we have 
$1_\la \cdot L(\la) \not= 0$ 
since 
$\ol{1_\la} \not\in \rad \D(\la)$. 
On the other hand, 
one sees easily that 
$1_\mu \cdot \D(\la)=0$ 
for any $\mu \in \vL$ such that $\mu \not\leq \la$. 
Thus, 
if $L(\mu)$ is a composition factor of $\D(\la)$, 
we have $1_\mu \cdot \D(\la) \not=0$  
and $\mu \leq \la$.  
Moreover, 
one sees that 
$1_\la \cdot \rad \D(\la) = 0$ (note that $\ol{1_\la} \not\in \rad \D(\la)$). 
This implies that 
$L(\la)$ does not appear in $\rad \D(\la)$ as a composition factor. 
Thus we have (\roii). 

(\roiii) 
Let 
$\{\la_{\lan 1 \ran},\la_{\lan 2 \ran},\cdots, \la_{\lan z \ran}\}$ be 
such that $i<j$ if $\la_{\lan i \ran}>\la_{\lan j \ran}$. 
Put $\CS_q(\la_{\lan i \ran})=\sum_{j \leq i} \CS_q^- 1_{\la_{\lan j \ran}} \CS_q^+$, 
then 
$\CS_q(\la_{\lan i \ran})$ 
turns out to be a two-sided ideal of $\CS_q$. 
Thus, 
we have the following filtration of two-sided ideals. 
\begin{align}
\CS_q=\CS_q(\la_{\lan z \ran}) \supset \CS_q(\la_{\lan z-1 \ran}) 
\supset \cdots \supset \CS_q(\la_{\lan 1 \ran}) \supset \CS_q(\la_{\lan 0 \ran})=0.
\end{align}
One sees easily that 
$\CS_q(\la_{\lan i \ran})/\CS_q(\la_{\lan i-1 \ran}) 
\cong \CS_q(\geq \la_{\lan i \ran})/ \CS_q (> \la_{\lan i \ran})$ 
as $(\CS_q,\CS_q)$-bimodules 
for $\la_{\lan i \ran} \in \vL$. 
Moreover, 
one can check that 
\[\CS_q(\la_{\lan i \ran}) \not= \CS_q(\la_{\lan i-1 \ran})  
\, \text{ if and only if } \, 
1_{\la_{\lan i \ran}} \not\in \CS_q(>\la_{\lan i \ran }) 
\, \text{ if and only if }  \,
\la_{\lan i \ran} \in \vL^+. 
\]
Let
$\vL^+=\{\la_{\lan c_1 \ran},\cdots,\la_{ \lan c_{z'} \ran} \}$ 
such that 
$i<j$ if $c_i < c_j$. 
Then, 
we have the following filtration of two-sided ideals.  
\begin{align} \label{filtration Sq}
\CS_q=\CS_q(\la_{\lan c_{z'} \ran}) \supsetneqq \CS_q(\la_{\lan c_{z'-1} \ran}) 
\supsetneqq \cdots \supsetneqq \CS_q(\la_{\lan c_1 \ran}) \supsetneqq \CS_q(\la_{\lan c_0 \ran})=0  
\end{align}
such that  
$\CS_q(\la_{\lan c_{i} \ran})/\CS_q(\la_{\lan c_{i-1} \ran}) 
	\cong \CS_q(\geq \la_{\lan c_{i} \ran})/ \CS_q(> \la_{\lan c_{i} \ran})$ 
as $(\CS_q,\CS_q)$-bimodules. 

By the filtration of $\CS_q$ in (\ref{filtration Sq}) 
and the surjective homomorphism of $(\CS_q,\CS_q)$-bimodules 
$\D(\la) \otimes_{\CK} \D^\sharp (\la) 
	\ra \CS_q(\geq \la)/ \CS_q(> \la)$ 
for $\la\in \vL^+$ 
in Lemma \ref{lem surjection}, 
any composition factor of 
$\CS_q$ is a composition factor of 
$\D(\la)$ for some $\la \in \vL^+$. 
Thus, it is enough to show that any composition factor of 
$\D(\la)$ $(\la \in \vL^+)$ 
is isomorphic to $L(\mu)$ for some $\mu \in \vL^+$. 
We prove it by using the induction on $\vL^+$. 

Let $\la \in \vL^+$ be a minimal element with respect to the order \lq\lq $\geq$ ".  
We take $\ol{x}=\sum r_{(i_1,\cdots,i_k)} \ol{F_{i_1}\cdots F_{i_k} 1_{\la}} 
\in \rad\D(\la)$. 
Put 
$x= \sum r_{(i_1,\cdots,i_k)} F_{i_1}\cdots F_{i_k} 1_{\la} \in \CS_q(\geq\la)$. 
For $\mu \in \vL^+$ such that $\la \not= \mu $,  
we have 
$\CS_q(\geq \mu)\cdot x \in \CS_q(\geq \la) \cap \CS_q(\geq \mu) 
\subset \CS_q(>\la)$  
since 
both of $\CS_q(\geq \la)$ and $\CS_q(\geq \mu )$ 
are two-sided ideals of $\CS_q$ and $\la$ is a minimal element of $\vL^+$. 
This implies that 
$\CS_q(\geq \mu) \cdot \ol{x} =0 $ 
for any 
$\mu \in \vL^+$ 
such that 
$\mu \not= \la$. 
On the other hand, 
for any 
$F_{y_1}\cdots F_{y_b} 1_{\la} E_{x_1}\cdots E_{x_a} \in \CS_q(\geq \la)$, 
we have 
\begin{align*}
(F_{y_1}\cdots F_{y_b} 1_{\la} E_{x_1}\cdots E_{x_a}) 
\cdot \ol{x} 
= \lan \ol{1_{\la} E_{x_1}\cdots E_{x_a}} \, , \, \ol{x} \ran 
	\ol{F_{y_1}\cdots F_{y_b} 1_\la}
=0, 
\end{align*}
where 
the first equation follows Lemma \ref{lem property bilinear} (\roii), 
and 
the second equation follows $\ol{x} \in \rad \D(\la)$. 
This implies that 
$\CS_q(\geq \la) \cdot \ol{x}=0$. 
Together with the above arguments, 
we have $\CS_q \cdot \ol{x}=0$. 
In particular, 
we have 
$\ol{x}=1 \cdot \ol{x}=0$. 
This means that $\rad \D(\la) =0$, 
and we have $\D(\la)=L(\la)$. 

Next, we suppose that $\la \in \vL^+$ is not minimal. 
Put 
\[\CS_q(\not< \la) =\sum_{\mu \in \vL \atop \mu \not< \la} \CS_q^- 1_\mu \CS_q^+ 
\,\, \text{ and }\,\, 
\CS_q(\not\leqq \la) =\sum_{\mu \in \vL \atop \mu \not\leqq \la} \CS_q^- 1_\mu \CS_q^+.\]
One sees that 
$\CS_q(\not< \la)$ and $\CS_q(\not\leqq \la)$ 
are two-sided ideals of $\CS_q$. 
It is clear that 
$\CS_{q}(\not\leqq \la) \cdot \D(\la) =0$. 
Moreover, we see that 
$\CS_q(\geq \la) \cdot \rad \D(\la)=0$ 
in a similar way as in the above arguments. 
Thus, 
we have 
$\CS_{q}(\not< \la) \cdot \rad \D(\la) =0$. 
This implies that 
the action of $\CS_q$ on $\rad \D(\la)$ 
induces the action of 
$\CS_q/\CS_q(\not< \la)$ on $\rad \D(\la)$. 
Thus, any composition factor of $\rad \D(\la)$ 
is a composition factor of $\CS_q/\CS_q(\not< \la)$. 
Moreover, 
we can take a total order of $\vL$ such that 
$\CS_q(\not< \la )=\CS_q(\la_{\lan k \ran})$ for some $k$ 
and that $\la_{\lan j \ran}<\la$ for any $j=k+1,\cdots,z$.
Thus, by Lemma \ref{lem surjection}, 
any composition factor of $\CS_q/\CS_q(\not< \la )$ 
is a composition factor of $\D(\mu)$ 
for some 
$\mu \in \vL^+$ 
such that $\mu<\la$. 
By the induction hypothesis, 
we see that any composition factor of $\D(\mu)$ such that $\mu < \la$ 
is isomorphic to 
$L(\nu)$ for some $\nu \in \vL^+$. 
It follows that 
any composition factor of $\rad \D(\la)$ 
is isomorphic to 
$L(\nu)$ for some $\nu \in \vL^+$. 
Since $\D(\la)/ \rad \D(\la)=L(\la)$, 
we obtain (\roiii). 

(\roiv) 
Suppose that $\CS_q$ is semisimple, 
then $L(\la)$ and $L(\mu)$ ($\la \not= \mu \in \vL^+$) 
belong to different blocks of $\CS_q$. 
On the other hand, 
$\D(\la)$ is indecomposable since $\D(\la)$ has the unique top. 
Thus, all the composition factors of $\D(\la)$ belong to the same block. 
This means that $\D(\la)$ has only $L(\la)$ as composition factors, 
and 
we have  
$\D(\la) = L(\la)$ for any $\la \in \vL^+$ 
by (\roii).  
We have 
$\D^\sharp (\la) = L^\sharp (\la)$ for any $\la \in \vL^+$ in a similar way. 

Next we suppose that 
$\D(\la) \cong L(\la)$ and $\D^\sharp (\la) \cong L^\sharp(\la)$ for any $\la \in \vL^+$. 
Then, 
the surjective homomorphism of $(\CS_q, \CS_q)$-bimodules 
$\D(\la) \otimes_{\CK} \D^\sharp (\la) \ra \CS_q(\geq \la) / \CS_q(>\la)$ in Lemma \ref{lem surjection} 
must be isomorphic.  
Thus, the filtration \eqref{filtration Sq} implies that 
\[
\dim_{\CK} \CS_q = \sum_{\la \in \vL^+} (\dim_{\CK} \D(\la))^2. 
\]
($\dim_\CK L(\la) = \dim_\CK L^\sharp (\la)$ will be prove in Lemma \ref{iso L dual}.)
This implies that 
$\CS_q$ is semisimple. 
\end{proof}

\para  
\label{def Borel Sq +-}
Let 
$\CS_q^{\geq 0}$ 
(resp. $\CS_q^{\leq 0}$) 
be the subalgebra of 
$\CS_q$ 
generated by 
$\CS_q^{+}$ (resp. $\CS_q^-$) 
and $\CS_q^0$. 
Thus, 
$\CS_q^{\geq 0}$ (resp. $\CS_q^{\leq 0}$) 
is generated by $E_i$ (resp. $F_i$) for $i=1,\cdots,m-1$  and 
$1_{\la}$ for $\la \in \vL$. 
For $\la \in \vL$ such that $1_{\la} \not=0$ in $\CS_q$, 
let $\theta_\la = \CK v_{\la}$ be the one dimensional vector space with a basis $v_{\la}$. 
We define a left action of $\CS_q^{\geq 0}$ on $\theta_{\la}$ by 
\[1_\mu \cdot v_{\la}=\d_{\la \mu} v_\la, \quad E_i \cdot v_\la =0 \quad 
\text{for } \mu \in \vL \text{ and } i=1,\cdots,m-1.\]
One can check that this action is well-defined for $\la \in \vL$ such that $1_{\la}\not=0$. 
Similarly, 
we define a right action of $\CS_q^{\leq 0}$ on $\theta_{\la}$ by 
\[ v_{\la}\cdot 1_{\mu} =\d_{\la \mu} v_\la, \quad  v_\la \cdot F_i=0 \quad 
\text{for } \mu \in \vL \text{ and } i=1,\cdots,m-1.\]
We have the following theorem. 
( 
A similar theorem for cyclotomic $q$-Schur algebras 
has been obtained by \cite{DR}. 
The proof given here is similar to the proof given in \cite{DR}. 
)

\begin{thm}\ \label{thm standard induced}
\begin{enumerate}
\item 
$\{1_\la \,|\, \la \in \vL \text{ such that } 1_\la \not=0 \}$ 
is the complete set of primitive idempotents in $\CS_q^{\geq 0}$ and $\CS_q^{\leq 0}$.  

\item 
$\{ \theta_\la \,|\, \la \in \vL \text{ such that } 1_\la \not=0 \}$ 
is a complete set of non-isomorphic simple left $\CS_q^{\geq 0}$-modules,  
and of non-isomorphic simple right $\CS_q^{\leq 0}$-modules.

\item 
For $\la \in \vL$ such that $1_{\la} \not=0$, 
we have the following isomorphism of left $\CS_q$-modules. 
\[
\CS_q \otimes_{\CS_q^{\geq 0}} \theta_\la \cong 
\begin{cases} 
\D(\la) &\text{ if } \la \in \vL^+, \\
0 & \text{otherwise}.
\end{cases}
\]

\item 
For $\la \in \vL$ such that $1_{\la} \not=0$, 
we have the following isomorphism of right $\CS_q$-modules. 
\[
\theta_\la  \otimes_{\CS_q^{\leq 0}} \CS_q \cong
\begin{cases} 
\D^\sharp (\la) &\text{ if } \la \in \vL^+, \\
0 & \text{otherwise}.
\end{cases}
\]
\end{enumerate}
\end{thm}

\begin{proof}
We show the theorem only for $\CS_q^{\geq 0}$. 
The proof is similar for $\CS_q^{\leq 0}$. 
Note that 
\[
1_\la E_{i_1}\cdots E_{i_k} 1_\la = 1_\la 1_{\la + \a_{i_1}+\cdots +\a_{i_k}} E_{i_1} \cdots E_{i_k}=0\] 
for $1\leq i_1, \cdots, i_k\leq m-1,\, k\geq 1$. 
Thus, 
for $\la \in \vL$ such that $1_{\la} \not=0$, 
we have 
$1_\la \CS_q^{\geq 0} 1_\la = \CK 1_\la$. 
This implies that 
$1_\la$ is a primitive idempotent of $\CS_q^{\geq 0}$ 
since 
$1_\la \CS_q^{\geq 0} 1_\la \cong \End_{\CS_q^{\geq 0}}(\CS_q^{\geq 0} 1_\la)$,  
and 
$\dim_{\CK} \End_{\CS_q^{\geq 0}}(\CS_q^{\geq 0} 1_\la) \geq 2$ 
if $1_\la$ is not primitive. 
Moreover, 
we have $1=\sum_{\la \in \vL } 1_\la$, 
and so  
$\{1_\la \,|\, \la \in \vL \text{ such that } 1_\la \not=0\}$ 
is the complete set of primitive idempotents in $\CS_q^{\geq 0}$. 
Thus, 
for $\la \in \vL$ such that $1_\la \not=0$, 
$\Theta_\la=\CS_q^{\geq 0} 1_\la$ 
is a principal indecomposable $\CS_q^{\geq 0}$-module. 
By investigating  the degrees, 
$\CS_q^{\geq 0} \cdot (x 1_\la)$ 
is a proper $\CS_q^{\geq 0}$-submodule of $\Theta_\la$ 
for any $x \in \CS_q^{+}$ such that $x \not=1$.  
This implies that 
$\Theta_\la/ \Rad \Theta_\la \cong \theta_\la$.  
Now, we proved (\roi) and (\roii).

Next, we prove (\roiii). 
If $\la \not\in \vL^+$, 
we can write 
$1_\la = \sum_{x \in \CS_q^-, y \in \CS_q^+, \mu>\la } r_{x,y,\mu}\, x 1_\mu y$ in $\CS_q$.  
Thus, we have 
\begin{align*}
1\otimes \theta_\la 
= 
\sum_{\nu \in \vL} 1_\nu \otimes \theta_\la 
= 
1_\la \otimes \theta_\la 
=\sum_{x \in \CS_q^-, y \in \CS_q^+, \mu>\la } r_{x,y,\mu}\, x 1_\mu y \otimes \theta_\la 
=0.
\end{align*}
This implies that 
$\CS_q \otimes_{\CS_q^{\geq 0}} \theta_\la =\CS_q \cdot (1\otimes \theta_\la)=0$. 
Hence, 
we suppose that $\la \in \vL^+$. 
Note that $\D(\la)$ is generated by an element $\ol{1_\la}$, 
and that $\CS_q \otimes_{\CS_q^{\geq 0}} \theta_\la$ 
is generated by $1\otimes v_\la$ 
as $\CS_q$-modules. 
We define a map 
$f_\la : \D(\la) \ra \CS_q \otimes_{\CS_q^{\geq 0}} \theta_\la$ 
by 
$\ol{u \cdot 1_\la} \mapsto  u \otimes v_\la$ 
for $u \in \CS_q$. 
One can check that $f_\la$ gives a well-defined $\CS_q$-homomorphism.   
On the other hand, 
we define the map 
$\wt{g}_\la: \CS_q \times \theta_\la \ra \D(\la)$ 
by 
$(u, rv_\la) \mapsto r \ol{u\cdot 1_\la}$ 
for $u \in \CS_q, r\in \CK$.  
One can check that 
$\wt{g}_\la$ 
gives a well-defined $\CS_q^{\geq 0}$-balanced map. 
Thus, 
$\wt{g}_\la$ 
induces 
an $\CS_q$-homomorphism 
$g_\la : \CS_q\otimes_{\CS_q^{\geq 0}} \theta_\la \ra \D(\la)$ 
such that $u \otimes v_\la \mapsto \ol{u\cdot 1_\la}$. 
Thus, (\roiii) is proved. 
\end{proof}

\para
For given $\eta_\vL=\{\eta_i^\la \,|\, 1 \leq i \leq m-1, \, \la \in \vL\}$, 
where 
$\eta_i^\la \in \wt{\CS}_q^- \wt{\CS}_q^+ 1_\la$ 
such that 
$\deg(\eta_i^\la)=0$, 
we take 
$\eta_i \in \wt{U}_q^- \wt{U}_q^0 \wt{U}_q^+$ $(1 \leq i \leq m-1)$ 
such that 
$\wt{\Psi}(\eta_i)=\sum_{\la \in \vL} \eta_i^\la$, 
and put $\eta=(\eta_1,\cdots,\eta_{m-1})$.  

On the other hand, 
for given $\eta=(\eta_1,\cdots,\eta_{m-1})$, 
where 
$\eta_i \in \wt{U}_q^- \wt{U}_q^0 \wt{U}_q^+$ 
such that 
$\deg (\eta_i)=0$, 
and for given 
$\vL\subset P$,  
set  
$\eta_i^\la = \wt{\Psi}(\eta_i) 1_\la$ 
$(1 \leq i \leq m-1,\, \la \in \vL)$, 
and put 
$\eta_\vL=\{\eta_i^\la \,|\, 1 \leq i \leq m-1,\,\la \in \vL \}$. 

Under this correspondence, 
we have the following theorem. 

\begin{thm}\
\label{thm into Oeta}
\begin{enumerate}
\item 
Let $\CS_q^{\eta_\vL}$-mod be the category 
of finite dimensional left $\CA_q^{\eta_\vL}$-modules.  
Then 
$\CS_q^{\eta_\vL}$-mod is a full subcategory of $\CO^\eta$. 
In particular, 
when we regard a $\CS_q^{\eta_\vL}$-module as a $\wt{U}_q$-module 
through the surjection $\Psi: \wt{U}_q \ra \CS_q^{\eta_\vL}$, 
$\D(\la)$ $(\la \in \vL^+)$ is a highest weight module, 
and 
$L(\la)$ $(\la \in \vL^+)$ is a simple highest weight module 
with a highest weight $\la$ associated to $\eta$. 

\item 
For each $M \in \CO^\eta$, 
if the set of weight $\la$ such that $M_\la \not=0$ 
is contained in $\vL$, 
then  
we have 
$M \in \CS_q^{\eta_\vL}$-mod, 
where we regard the $\CS_q^{\eta_\vL}$-mod as a full subcategory of $\CO^\eta$ by $(\roi)$. 
In particular,  
any simple object of $\CO^\eta$ 
is obtained as in Theorem \ref{thm standard simple Sq}
through the quotient algebra $\CS_q^{\eta_\vL}$ for a suitable $\vL \subset P_{\geq 0}$, 
where the choice of $\vL$ depends on the simple object of $\CO^\eta$.  

\item 
$\CO^\eta$ 
is a full subcategory of 
$\wh{\CO}^\eta_{\text{tri}}$. 
\end{enumerate}
\end{thm}

\begin{proof}
(\roi) is clear through the surjection 
$\Phi : \wt{U}_q \ra \CS_q^{\eta_\vL}$,  
and by 
the definitions of $\D(\la)$ and $L(\la)$. 

We prove (\roii). 
For $M \in \CO^\eta$, 
put  
$\vL_M=\{\la \in P_{\geq 0}  \,|\, M_{\la} \not= 0\}$. 
(Note that $M_\la =0 $ unless $\la \in P_{\geq 0}$ by the condition (e) in the definition of $\CO^\eta$.) 
Since the dimension of $M$ is finite, 
$\vL_M$ is a finite set. 
We take a finite subset $\vL$ of $P_{\geq 0}$ such that $\vL_M \subset \vL$.  
Then, we can define an action of $\CS_q^{\eta_\vL}$ on $M$ as follows; 
\begin{align*}
&E_i \cdot m = e_i \cdot m &&\text{ for } 1 \leq i \leq m-1, \, m \in M, \\
&F_i \cdot m = f_i \cdot m &&\text{ for } 1 \leq i \leq m-1, \, m \in M, \\ 
&1_\la \cdot m = \d_{\la\mu} m &&\text{ for } \la \in \vL, \, m \in M_\mu.
\end{align*}
One can check that this action is well-defined 
by using the defining relations of $\wt{U}_q$ and 
the definition of $\CO^\eta$. 
We denote this $\CS_q^{\eta_\vL}$-module by $M^{\vL}$.  
When we regard $M^{\vL}$ as a $\wt{U}_q$-module through the surjection $\Psi$, 
$M^{\vL}$ coincides with $M$. 
This implies that 
$M \in \CS_q^{\eta_\vL}$-mod. 
Now, the last assertion of (\roii) is clear. 

Since $\CS_q^{\eta_\vL}$ has the triangular decomposition 
compatible with that of $\wt{U}_q$,  
(\roiii) follows from (\roii). 
\end{proof}

\para \label{def iota}
We define an algebra anti-automorphism 
$\iota : \wt{\CS}_q \ra \wt{\CS}_q$ 
by 
$\iota(E_i)=F_i$, 
$\iota(F_i)=E_i$, 
$\iota(1_\la)=1_\la$ 
and 
$\iota(\t_i^\la)=\t_i^\la$ 
for $i=1,\cdots,m-1$ and $\la \in \vL$. 
We can easily check that 
$\iota$ is well-defined. 
We consider the following conditions; 
\begin{align*} \label{condition to be cellular}
&\iota(\eta_i^\la)=\eta_i^\la  \text{ for any } i=1,\cdots, m-1 \text{ and } \la \in \vL. 
\tag{\textbf{C-1}} \\
&\D(\la)\otimes_{\CK}  \D^\sharp (\la) \cong \CS_q(\geq \la)/ \CS_q(> \la) 
 \text{ as } (\CS_q,\CS_q)\text{-bimodules for any } \la \in \vL^+. 
\tag{\textbf{C-2}}
\end{align*}
Thanks to the condition (C-1), 
$\iota$ induces a well-defined algebra anti-automorphism on $\CS_q$. 
In view of the Lemma \ref{lem surjection}, 
the condition (C-2) is equivalent to the following condition; 
\begin{align*}
\sum_{x \in \CS_q^-, y \in \CS_q^+} \hspace{-1em}r_{xy} x 1_\la y \in \CS_q(> \la) 
\Rightarrow 
\sum_{x \in \CS_q^-, y \in \CS_q^+}\hspace{-1em} r_{xy} \ol{x 1_\la} \otimes  \ol{ 1_\la y} =0  
	\in  \D(\la)\otimes_{\CK} \D^\sharp (\la). 
\tag{\textbf{C'-2}}
\end{align*}

It is clear that 
\begin{align*}
&u \in \CS_q(\geq \la) \text{ if and only if } \iota(u) \in \CS_q(\geq \la), \\
&u \in \CS_q(> \la) \text{ if and only if } \iota(u) \in \CS_q(> \la).
\end{align*}
This implies that 
$\D(\la) \ni \ol{x} \mapsto \ol{\iota(x)} \in \D^\sharp (\la)$ 
gives an isomorphism of $\CK$-vector spaces. 
We consider the filtration of $\CS_q$ in (\ref{filtration Sq}). 
Recall that
\[\CS_q(\la_{\lan c_i \ran})/ \CS_q(\la_{\lan c_{i-1} \ran}) 
\cong \CS_q(\geq \la_{\lan c_i \ran})/ \CS_q(> \la_{\lan c_{i} \ran})
\quad \text{ as } (\CS_q,\CS_q)\text{-bimodules}.
\] 
Under the condition (C-1) and (C-2),  
we have the following commutative diagram; 
\[
\begin{matrix}
\CS_q(\la_{\lan c_i \ran})/ \CS_q(\la_{\lan c_{i-1} \ran}) 
	& \cong 
	&\hspace{-2em} \D(\la_{\lan c_i \ran}) \otimes_{\CK} \D^\sharp (\la_{\lan c_i \ran}) \\
\scalebox{1}[2]{$\downarrow$} \,\iota 
	& 
	& \hspace{5em}\scalebox{1}[2]{$\downarrow$}\, 
	  \scalebox{0.8}{$\ol {x}\otimes \ol{y} \mapsto \ol{\iota(y)} \otimes \ol{\iota(x)}$} \\[2mm]
\CS_q(\la_{\lan c_i \ran})/ \CS_q(\la_{\lan c_{i-1} \ran}) 
	& \cong 
	&\hspace{-2em} \D(\la_{\lan c_i \ran}) \otimes_{\CK} \D^\sharp (\la_{\lan c_i \ran}).
\end{matrix}
\] 
This implies that 
$\CS_q(\la_{\lan c_i \ran})/ \CS_q(\la_{\lan c_{i-1} \ran})$ 
is a cell ideal of 
$\CS_q/ \CS_q(\la_{\lan c_{i-1} \ran})$ in the sense of \cite{KX98}. 
Thus, 
$\CS_q$ turns out to be a cellular algebra (see \cite[Definition 3.2]{KX98}), 
and 
$\D(\la)$ ($\la \in \vL^+$) 
gives a cell module of $\CS_q$. 
Moreover, 
we already know that 
$\{L(\la) \,|\,\la \in \vL^+\}$ 
gives a complete set of non-isomorphic simple $\CS_q$-modules. 
Thus, we have the following theorem. 

\begin{thm} \label{thm Sq cellular}
If $\CS_q$ satisfies the conditions (C-1) and (C-2), 
$\CS_q$ is a quasi-hereditary cellular algebra. 
\end{thm}



\section{Specialization to an arbitrary ring} 
\label{specialization}
In this section, 
we define an $\CA$-form $_\CA \CS_q$ of $\CS_q$, 
and we consider a specialization $_R \CS_q$ of $_\CA \CS_q$ to an arbitrary ring $R$. 
We will assume some conditions on the choice of 
$\{ \eta_i^{\la} \,|\, 1 \leq i \leq m-1,\,\la \in \vL\}$ 
so that, 
in the case where $R$ is a field, 
we obtain the properties of 
$_R \CS_q$ 
which are similar to those obtained in the previous section, 
and are compatible with the case where $R=\CK$.

\para 
Put $E_i^{(k)}=E_i^k/[k]!$, $F_i^{(k)}=F_i^k/[k]!$. 
Let $_\CA \CS_q$ be the $\CA$-subalgebra of $\CS_q$ 
generated by $E_i^{(k)}$, $F_i^{(k)}$ $(1 \leq i \leq m-1,\, k\geq 1)$ 
and $1_\la$ $(\la \in \vL)$. 
Note that, 
by Lemma \ref{lem-1-lamda}, 
we have $\Psi(\,_\CA \wt{U}_q)=\,_\CA \CS_q$. 

Let $_\CA \CS_q^+$ (resp. $_\CA \CS_q^-$) 
be the $\CA$-subalgebra of $_\CA \CS_q$ generated by 
$E_i^{(k)}$ (resp. $F_i^{(k)}$) for 
$1 \leq i \leq m-1$, $k \geq 0$, 
and $_\CA \CS_q^0$ be the $\CA$-subalgebra of $_\CA \CS_q$ 
generated by $1_\la$ for $\la \in \vL$. 
As we have seen in section \ref{Sq}, 
$\CS_q$ has the triangular decomposition $\CS_q=\CS_q^- \CS_q^0 \CS_q^+$ over $\CK$. 
However, 
it may happen that  
such relations break over $\CA$. 
Hence, 
the triangular decomposition will hold over $\CA$ 
so that  
we consider the following condition  
\begin{align*}
E_i^{(k)} F_i^{(l)} \in \,_\CA \CS_q^- \, _\CA \CS_q^0 \, _\CA \CS_q^+ 
\quad \text{ for } 1 \leq i \leq m-1, \,\, k,l \geq 1. 
\tag{\textbf{A-1}} \label{A-1}
\end{align*}
Under this assumption, 
we can prove the following proposition 
by replacing $E_i$, $F_j$ ($1 \leq i,j \leq m-1$) 
with divided powers $E_i^{(k)}$, $F_j^{(l)}$ ($1 \leq i,j \leq m-1,\,\, k,l \geq 1$) 
in the proof of Proposition \ref{tri decom}.

\begin{prop}\label{prop tri A}
Suppose that (A-1) holds. 
Then  
$_\CA \CS_q$ has a triangular decomposition 
\[ _\CA \CS_q = \, _\CA \CS_q^- \, _\CA \CS_q^0 \,  _\CA \CS_q^+ .\] 
Moreover, 
$ _\CA \CS_q$ is finitely generated over $\CA$. 
\end{prop}

In the rest of  this section, 
we always assume the condition (A-1).

\para 
Let $R$ be an arbitrary ring, 
and we take $\xi_0, \xi_1,\cdots,\xi_{r} \in R$, 
where $\xi_0$ is invertible in $R$. 
We regard $R$ as an $\CA$-module 
by the homomorphism of rings $\pi : \CA \ra R$ such that 
$q \mapsto \xi_0, \g_i \mapsto \xi_i$ ($1 \leq i \leq r$). 
Then, we obtain the specialized algebra 
$R \otimes_{\CA} \,_\CA \CS_q$ 
of $_\CA \CS_q$ 
through the homomorphism $\pi$. 
We denote it by $_R \CS_q$, 
and 
denote 
$1 \otimes x \in R \otimes_{\CA} \, _\CA \CS_q$ 
simply 
by 
$x$ 
if it does not cause any confusion. 
Let $_R \CS_q^+$ (resp. $_R \CS_q^-$) 
be the $R$-subalgebra of $_R \CS_q$ generated by 
$1 \otimes E_i^{(k)}$ (resp. $1 \otimes F_i^{(k)}$) for 
$1 \leq i \leq m-1$, $k \geq 0$, 
and $_R \CS_q^0$ be the $R$-subalgebra of $_R \CS_q$ 
generated by $1 \otimes 1_\la$ for $\la \in \vL$. 
By Proposition \ref{prop tri A}, 
we have the triangular decomposition 
\[
_R \CS_q = \, _R \CS_q^- \, _R \CS_q^0 \,  _R \CS_q^+.
\] 
Thanks to the triangular decomposition, 
we have the following results which are similar to the case over $\CK$.  
For $\la \in \vL$, 
let 
\begin{align*}
_R \CS_q(\geq \la) 
	= \big\{ x 1_\mu y \bigm| x \in \,_R \CS_q^-,\, y \in \,_R \CS_q^+,\, \mu \in \vL 
		\text{ such that }\mu \geq \la \big\}, \\
_R \CS_q(>\la) 
	= \big\{ x 1_\mu y \bigm| x \in \,_R \CS_q^-,\, y \in \,_R \CS_q^+,\, \mu \in \vL 
	\text{ such that }\mu > \la \big\}. \\
\end{align*}
Then, 
$_R \CS_q(\geq \la)$ and $_R \CS_q(> \la)$ are two-sided ideals of $_R \CS_q$. 
Put 
\[_R \vL^+=\{\la \in \vL \,|\, _R \CS_q(\geq \la) \not= \,_R \CS_q(>\la) \} 
	=\{ \la \in \vL \,|\, 1_{\la} \not \in \, _R \CS_q(>\la)\}.\] 
For $\la \in \,_R \vL^+$, 
we define a left (resp. right) $_R \CS_q$-submodule 
$_R \D(\la)$ (resp. $_R \D^\sharp(\la))$) 
of $_R \CS_q(\geq \la)/\,_R \CS_q(> \la)$ by 
\[_R \D(\la)=\,_R \CS_q^- \cdot 1_{\la} + \,_R \CS_q(>\la), \quad 
	\,_R \D^\sharp(\la)=1_{\la} \cdot \,_R \CS_q^+  + \,_R \CS_q(>\la).
\]
Let 
$_R \CS_q^{\geq 0}$ 
(resp. $_R \CS_q^{\leq 0}$) 
be the subalgebra of 
$_R \CS_q$ 
generated by 
$_R \CS_q^{+}$ (resp. $_R \CS_q^-$) 
and $_R \CS_q^0$. 
For $\la \in \vL$ such that $1_{\la} \not=0$ in $_R\CS_q$, 
let $\theta_\la = R v_{\la}$ be the free $R$-module with a basis $v_{\la}$. 
We define the left action of $_R \CS_q^{\geq 0}$ on $\theta_{\la}$ by 
\[1_\mu \cdot v_{\la}=\d_{\la \mu} v_\la, \quad E_i^{(k)} \cdot v_\la =0 \quad 
\text{for } \mu \in \vL,\, i=1,\cdots,m-1 \text{ and } k\geq 1.\]
Similarly, 
we define a right action of $_R \CS_q^{\leq 0}$ on $\theta_{\la}$ by 
\[ v_{\la}\cdot 1_{\mu} =\d_{\la \mu} v_\la, \quad  v_\la \cdot F_i^{(k)}=0 \quad 
\text{for } \mu \in \vL, \,  i=1,\cdots,m-1 \text{ and } k\geq 1.\]

We have the following theorem which is shown in a similar way as in the proof of Theorem \ref{thm standard induced}.

\begin{thm}\ \label{thm standard induced borel R}
\begin{enumerate}
\item 
$\{1_\la \,|\, \la \in \vL \text{ such that } 1_\la \not=0 \}$ 
is the complete set of primitive idempotents in $_R\CS_q^{\geq 0}$ and $_R\CS_q^{\leq 0}$.  

\item 
$\{ \theta_\la \,|\, \la \in \vL \text{ such that } 1_\la \not=0 \}$ 
is a complete set of non-isomorphic simple left $_R \CS_q^{\geq 0}$-modules,  
and of non-isomorphic simple right $_R \CS_q^{\leq 0}$-modules.

\item 
For $\la \in \vL$ such that $1_{\la} \not=0$, 
we have the following isomorphism of left (resp. right) $_R\CS_q$-modules. 
\begin{align*}
_R \CS_q \otimes_{_R \CS_q^{\geq 0}} \theta_\la \cong 
\begin{cases} 
_R \D(\la) &\text{ if } \la \in \,_R \vL^+, \\
0 & \text{otherwise}, 
\end{cases}
\\
\theta_\la  \otimes_{_R \CS_q^{\leq 0}} \,_R \CS_q \cong
\begin{cases} 
_R \D^\sharp (\la) &\text{ if } \la \in \,_R\vL^+, \\
0 & \text{otherwise}.
\end{cases}
\end{align*}
\end{enumerate}
\end{thm}

\para 
For $\la \in \,_R \vL^+$, 
we can define a bilinear form 
$\lan \,,\, \ran : \,_R \D^\sharp (\la) \times \,_R \D(\la) \ra R$ such that 
\[ \lan \ol{1_\la y} \,,\,\ol{x 1_\la} \ran 1_\la 
\equiv 1_\la yx 1_\la \mod \,_R \CS_q(>\la) 
\quad \text{ for } x \in \, _R\CS_q^-, y \in \, _R \CS_q^+ .\]
Put 
$\rad \,_R \D(\la)=\{ \ol{x} \in \, _R \D(\la) \,|\, 
	\lan \ol{y} , \ol{x} \ran=0 \text{ for any } \ol{y} \in \,_R \D^\sharp (\la) \}$, 
and 
put 
$_R L(\la)=\,_R \D(\la) / \rad \, _R\D(\la)$. 
Similarly, 
put 
$\rad \,_R \D^\sharp (\la)=\{ \ol{y} \in \, _R \D^\sharp(\la) \,|\, 
	\lan \ol{y} , \ol{x} \ran=0 \text{ for any } \ol{x} \in \,_R \D (\la) \}$, 
and 
put 
$_R L^\sharp (\la)=\,_R \D^\sharp (\la) / \rad \, _R\D^\sharp (\la)$. 
Then, 
one can prove the following theorem by replacing $E_i$, $F_j$ ($1 \leq i,j \leq m-1$) 
with divided powers $E_i^{(k)}$, $F_j^{(l)}$ ($1 \leq i,j \leq m-1,\,\, k,l \geq 1$) 
in the proof of Theorem \ref{thm standard simple Sq}.

\begin{thm} \label{thm standard simple Sq R}
Suppose that $R$ is a field. 
Then 
we have the followimgs.
\begin{enumerate}
\item 
For $\la \in \,_R \vL^+$, 
$\rad \,_R \D(\la)$, 
(resp. $\rad \,_R \D^\sharp(\la)$)  
is a unique proper maximal submodule of 
$ _R \D(\la)$ 
(resp. $_R \D^\sharp (\la)$). 
Thus,  
$_R L(\la)$ 
(resp. $_R L^\sharp(\la)$) 
is an absolutely  simple left (resp. right) $_R \CS_q$-module. 

\item 
For $\la, \mu \in \, _R\vL^+$,  
if $_R L(\mu)$ (resp. $_R L^\sharp(\mu)$) 
is a composition factor of 
$_R \D(\la)$ (resp. $_R \D^\sharp (\la)$),  
we have 
$\la \geq \mu$. 
Thus, 
$_R L(\la) \cong \, _R L(\mu)$ 
if and only if 
$\la=\mu$. 
Moreover, 
the multiplicity of $_R L(\la)$ (resp. $_R L^\sharp(\la)$) 
in $_R \D(\la)$ (resp. $_R \D^\sharp(\la)$) 
is equal to one. 

\item 
$\{ \,_R L(\la) \,|\, \la \in\,  _R\vL^+ \}$ 
(resp. $\{ \,_R L^\sharp (\la) \,|\, \la \in\,  _R\vL^+ \}$ ) 
gives a complete set of 
non-isomorphic left (resp. right) simple $_R \CS_q$-modules. 

\item 
$_R \CS_{q}$ is semisimple if and only if 
$_R \D(\la) \cong \,_R L(\la)$ and $_R \D^\sharp (\la) \cong \,_R L^\sharp(\la)$ for any $\la \in \vL^+$.   

\end{enumerate}
\end{thm}

\para 
Throughout the rest of this section, 
we assume that $R$ is a field. 
Since 
$\rad\, _R \D^\sharp (\la) \times \rad \,_R \D (\la)$ 
is included in the kernel of the bilinear form 
$\lan \,,\, \ran: \,_R \D^\sharp(\la) \times \,_R \D(\la) \ra R$,  
$\lan \, , \,\ran$ 
induces 
a bilinear form on $_R L^\sharp(\la) \times \,_R L(\la)$. 
Clearly, 
this bilinear form is non-degenerate on $_R L^\sharp(\la) \times \,_R L(\la)$.
We regard 
$\Hom_R(\,_R L^\sharp(\la) , R)$ 
as an left $_R \CS_q$-module 
by the standard way.  
Thanks to the associativity of the bilinear form $\lan \,,\, \ran$ (Lemma \ref{lem property bilinear} (\roi)), 
the $R$-homomorphism 
$G:\,_R L(\la) \ra  \Hom_R(\,_R L^\sharp (\la) , R) $ 
given by 
$\ol{x} \mapsto \lan - , \ol{x} \ran$ 
turns out to be 
an $_R \CS_q$-homomorphism. 
Since 
$\lan \,,\, \ran$ 
is  non-degenerate on $_R L^\sharp(\la) \times \,_R L(\la)$, 
the homomorphism  $G$ is not a $0$-map. 
Hence,  
$G$ is an isomorphism of left $_R \CS_q$-modules 
since 
both of $_R L(\la)$ and $\Hom_R(\,_R L^\sharp (\la) , R)$ are simple. 
Thus,  we have the following lemma 
(a similar argument holds for $_R L^\sharp(\la)$). 

\begin{lem} \label{iso L dual}
Suppose that $R$ is a field. 
For $\la \in \,_R \vL^+$,  
we have the following isomorphisms. 
\begin{enumerate}
\item 
$_R L(\la) \cong \Hom_R (\,_R L^\sharp(\la),R)$ as left $_R \CS_q$-modules.
\item 
$_R L^\sharp (\la) \cong \Hom_R (\,_R L(\la),R)$ as right $_R \CS_q$-modules. 
\end{enumerate}
In particular, 
we have 
$\dim_R \,_R L(\la) = \dim_R \, _R L^\sharp(\la)$. 
\end{lem}

\para 
For $\la \in \,_R \vL^+$, 
let $_R P(\la)$ be the projective cover of $_R L(\la)$. 
For $\la, \mu \in \,_R \vL^+$, 
we denote  
the multiplicity of 
$_R L(\mu)$ 
in the composition series of 
$_R P(\la)$ 
by 
$[\,_R P(\la) : \, _R L(\mu)]$. 
Similarly, 
we denote  
the multiplicity of 
$_R L(\mu)$ (resp. $_R L^\sharp (\mu)$)  
in the composition series of 
$_R \D(\la)$ (resp. $_R \D^\sharp (\la)$)  
by 
$[\,_R \D(\la) : \, _R L(\mu)]$ 
(resp. $[\, _R \D^\sharp (\la) : \,_R L^\sharp(\la) ]$). 
We have the following relation concerning with these multiplicities.

\begin{lem} \label{rel mul Pla Lmu Dla}
Suppose that $R$ is a field. 
For $\la, \mu \in \,_R \vL^+$, 
we have 
\[ [\,_R P(\la) : \,_R L(\mu)] \leqq \sum_{\nu \in \,_R \vL^+}  
	[\,_R \D(\nu) : \, _R L(\mu)][\, _R \D^\sharp (\nu) : \,_R L^\sharp(\la) ]
.\]
\end{lem} 
\begin{proof}
In the proof, 
we omit the subscript $R$ 
as we always consider the objects over $R$. 
Let 
$ \vL^+ =\{\la_{\lan 1 \ran} , \cdots, \la_{\lan z \ran}\}$ 
be such that 
$i<j$ if $\la_{\lan i \ran} > \la_{\lan j \ran}$. 
Then we have the following filtrations of two-sided ideals, 
\begin{align} \label{fil R Sq}
\CS_q= \CS_q(\la_{\lan z \ran}) \supsetneqq  \CS_q(\la_{\lan z -1 \ran}) \supsetneqq \cdots \supsetneqq 
	 \CS_q(\la_{\lan 1 \ran}) \supsetneqq  \CS_q(\la_{\lan 0 \ran})=0 
\end{align}
such that 
$ \CS_q(\la_{\lan i \ran})/ \CS_q(\la_{\lan i-1 \ran}) 
	\cong  \CS_q(\geq \la_{\lan i \ran}) / \CS_q(> \la_{\lan i \ran})$ 
as 
$ \CS_q$-bimodules.  
Since $P(\la)$ is a left projective $ \CS_q$-module, 
the filtration (\ref{fil R Sq}) 
gives the following filtration of left $ \CS_q$-modules. 
\begin{align*}
 P(\la) = M_z \supset M_{z-1} \supset \cdots \supset M_{1} \supset M_0=0 
\end{align*}
such that 
$M_{i}/M_{i-1} 
	\cong \big( \CS_q(\geq \la_{\lan i \ran}) /\CS_q(> \la_{\lan i \ran})\big) 
		\otimes_{ \CS_q} P(\la)$. 
This implies that 
\begin{align} \label{Pla Lmu eq}
[ P(\la) : L(\mu)]= 
	\sum_{\nu \in  \vL^+} 
		\big[\big( \CS_q(\geq \nu) /  \CS_q(> \nu) \big) \otimes_{\CS_q}  P(\la) : L(\mu)\big].
\end{align}
Since 
there exists a surjection 
$\D(\nu) \otimes_{R}  \D^\sharp(\nu) \ra  \CS_q(\geq \nu) /  \CS_q(> \nu)$ 
of $ \CS_q$-bimodules, 
(\ref{Pla Lmu eq}) implies that 
\begin{align*}
[ P(\la) :  L(\mu)] \leqq \sum_{\nu \in \vL^+} 
	\big[  \D(\nu) \otimes_{R}  \D^\sharp(\nu) \otimes_{ \CS_q}  P(\la) :  L(\mu)\big].
\end{align*}
Thus, 
we should prove that 
\begin{align*}
\big[  \D(\nu) \otimes_{R} \D^\sharp(\nu) \otimes_{ \CS_q} P(\la) :  L(\mu)\big] 
= 
[\D(\nu) :  L(\mu)][\D^\sharp (\nu) :  L^\sharp(\la) ].
\end{align*} 
Since 
\[\big[  \D(\nu) \otimes_{R}  \D^\sharp(\nu) \otimes_{ \CS_q}  P(\la) : L(\mu)\big] 
= [ \D(\nu) : L(\mu)] \cdot \dim_{R}  \big(\D^\sharp(\nu) \otimes_{\CS_q}  P(\la)\big),
\]
it is enough to show that 
$\dim_{R}  \big(\D^\sharp(\nu) \otimes_{_R \CS_q}  P(\la)\big) = 
[\D^\sharp (\nu) : L^\sharp (\la)]$. 
By a standard theory of finite dimensional algebras over a field, 
we have 
\begin{align*}
\dim_{R}  \big(\D^\sharp(\nu) \otimes_{ \CS_q}  P(\la)\big) 
&=
\dim_{R} \big( \Hom_R \big( \big(\D^\sharp(\nu) \otimes_{\CS_q}  P(\la) , R \big) \big) \\ 
&=
\dim_{R}  \big( \Hom_{\CS_q} \big(  P (\la ) , \Hom_R ( \D^\sharp(\nu), R ) \big) \big) \\
&=
[ \Hom_R ( \D^\sharp(\nu), R ) : L(\la) ] \\
&=
[ \Hom_R ( \D^\sharp(\nu), R ) : \Hom_R ( L^\sharp (\la), R) ] \quad(\text{ Lemma \ref{iso L dual} } ) \\
&=
[\D^\sharp (\nu) : L^\sharp (\la)]. 
\end{align*}
Now the lemma is proven. 
\end{proof} 
 
\para 
For $\la \in \, _R \vL^+$, 
$_R \D(\la)$
is an indecomposable $_R \CS_q$-module 
since 
$_R \D(\la)$ has the unique top.  
Thus, 
all the composition factors of 
$_R \D(\la)$ 
belong to 
the same block of $_R \CS_q$. 
 
For $\la,\mu \in \,_R \vL^+$, 
we denote by  
$\la \sim \mu$ 
if 
there exists a sequence 
$\la=\la_0, \la_1,\cdots, \la_k  = \mu$ ($\la_i \in \,_R \vL^+$) 
such that 
$_R \D(\la_{i-1})$ and $_R \D(\la_i)$ ($1 \leq i \leq k$)  
have a common composition factor. 
Clearly, 
\lq\lq $\sim$" gives an equivalent relation on $_R \vL^+$, 
and 
$_R \D(\la)$ and $_R \D(\mu)$ 
belong to the same block 
if $\la \sim \mu$. 
If $_R \CS_q$ satisfies the condition (C-1), 
one can prove that 
the converse is also true. 
To prove it, 
we prepare the following lemma.

\begin{lem} \label{lem right left mult}
Suppose that $R$ is a field. 
If $_R \CS_q$ satisfies the condition (C-1), 
we have 
\[ [\, _R \D(\la) : \,_R L(\mu)]=[ \, _R \D^\sharp (\la) : \, _R L^\sharp (\mu)] .\]
\end{lem}
\begin{proof}
Thanks to (C-1), 
we can define an isomorphism of $R$-modules 
$\iota : \,_R \D(\la) \ra \,_R \D^\sharp(\la)$ 
via  
$\ol{x} \mapsto \ol{\iota(x)}$.  
For 
$y \in \,_R \CS_q^+$ and  $x \in \, _R \CS_q^-$, 
we have 
\[ \lan \ol{1_\la y} , \ol{x 1_\la} \ran 1_\la 
	\equiv 1_\la y x 1_\la 
	=1_\la \iota (x) \iota(y) 1_\la 
	\equiv \lan \ol{1_\la \iota(x)} \, ,\, \ol{\iota(y) 1_\la} \ran 1_\la 
\mod \,_R \CS_q(>\la).
\]
Thus, 
we have 
$\lan \ol{y} , \ol{x} \ran = \lan \ol{ \iota (x)} , \ol{ \iota (y)} \ran$ 
for any 
$\ol{x} \in \,_R \D(\la)$ and $\ol{y} \in \D^\sharp(\la)$. 
This implies that 
$\rad \,_R\D^\sharp(\la)=\{\ol{\iota(x)} \,|\, \ol{x} \in \rad \,_R \D(\la)\}$. 
Therefore, 
$\iota : \,_R \D(\la) \ra \,_R \D^\sharp(\la)$ 
induces an $R$-isomorphism  
$\, _R L(\la) \ra \, _R L^\sharp (\la)$. 
Let 
$_R \D(\la)=M_0 \supsetneqq M_1 \supsetneqq \cdots \supsetneqq M_k \supsetneqq 0$ 
be a composition series of $_R \D(\la)$ 
such that 
$M_{i-1}/M_i \cong \,_R L(\mu_i)$. 
By investigating the action of $_R \CS_q$, 
we see that 
$\iota(_R \D(\la))=\iota(M_0) \supsetneqq \iota(M_1) \supsetneqq \cdots \supsetneqq \iota(M_k) \supsetneqq 0$
gives a composition series of $_R \D^\sharp(\la)$ 
such that 
$\iota(M_{i-1})/\iota(M_i) \cong \,_R L^\sharp (\mu_i)$. 
This implies the lemma. 
\end{proof}

We have the following theorem.

\begin{thm} \label{thm block class}
Suppose that $R$ is a field. 
If $_R \CS_q$ satisfies the condition (C-1), 
then 
$\la \sim \mu $ 
if and only if 
$_R \D(\la)$ and $_R \D(\mu)$ 
belong to the same block of $_R \CS_q$ 
for $\la,\mu \in \, _R \vL^+$.

\end{thm}
\begin{proof}
As we have already seen the \lq\lq only if " part, 
we prove the \lq\lq if " part.   
Assume that  
$_R \D(\la)$ and $_R \D(\mu)$ belong to the same block. 
Then 
$_R P(\la)$ and $_R P(\mu)$ belong to the same block. 
Thus, 
there exists a sequence 
$\la=\la_0, \la_1,\cdots,\la_k=\mu$ ($\la_i \in \, _R \vL^+$) 
such that 
$_R P(\la_{i-1})$ and $_R P(\la_{i})$ ($1 \leq i \leq k$) 
have a common composition factor $_R L(\mu_i)$.  
By Lemma \ref{rel mul Pla Lmu Dla}, 
there exists 
$\nu_i, \nu'_i \in \,_R \vL^+$ ($1 \leq i \leq k$) 
such that 
$[\,_R \D (\nu_i) : \, _R L (\mu_i)]\not=0 $, 
$[\,_R \D^\sharp (\nu_i) : \, _R L^\sharp (\la_{i-1})] \not=0 $, 
$[\,_R \D (\nu'_i) : \, _R L (\mu_i)] \not=0 $, 
$[\,_R \D^\sharp (\nu'_i) : \, _R L^\sharp (\la_i)] \not=0 $. 
Combined with Lemma \ref{lem right left mult}, 
we have 
\[ \la_{i-1} \sim \nu_i \sim \mu_i \sim \nu'_i \sim \la_i\]
for each $1 \leq i \leq k$. 
Thus we have $\la \sim \mu$.

\end{proof}

%
%
%
%

%

\para 
Finally,  
we consider the following condition; 
\begin{align*}
&\text{For any }  \la \in \,_\CA \vL^+, \, 
_\CA \D(\la) \text{ is a free } \CA \text{-module, and } \tag{\textbf{A-2}}\\  
&_\CA \D(\la) \otimes_{\CA}  \,_\CA \D^\sharp (\la) \cong \, _\CA \CS_q(\geq \la)/ \, _\CA \CS_q(>\la) 
\quad \text{as } (_\CA \CS_q,\,_\CA \CS_q)\text{-bimodules}.
\end{align*}
We have the following theorem. 

\begin{thm} \label{thm cellular R}
Suppose that 
the conditions (A-1), (A-2) and (C-1) hold. 
Then, for an arbitrary ring $R$ and parameters $\xi_0,\xi_1,\cdots,\xi_r \in R$, 
$_R \CS_q$ is a cellular algebra with respect to the poset $\vL^+$. 
In particular, 
when $R$ is a field, 
$_R \CS_q$ is a quasi-hereditary cellular algebra. 
\end{thm}
\begin{proof}
Thanks to (C-1), 
the map 
$_\CA \D(\la) \ni \ol{x} \mapsto \ol{\iota(x)} \in \,_\CA \D^\sharp (\la)$ 
gives an isomorphism of $\CA$-modules. 
Thus, (A-2) implies that $_\CA \D^\sharp(\la)$ is a free $\CA$-module. 
Now, we can prove 
that $_\CA \CS_q$ is a cellular algebra 
with respect to the poset $_\CA \vL^+$ 
in a similar way as in the case over $\CK$ (Theorem \ref{thm Sq cellular}), 
and $_\CA \D(\la)$ ($ \la \in \,_\CA\vL^+ $)  
is a (left) cell module of $_\CA \CS_q$. 
Thus, for any ring $R$, 
$_R \CS_q$ is a cellular algebra with respect to the poset $_\CA \vL^+$, 
and $R \otimes_{\CA} \,_\CA \D(\la)$  ($ \la \in \,_\CA\vL^+ $)  
is a cell module of $_R \CS_q$. 

From now on, 
we assume that $R$ is a field. 
It is clear that 
$ 1 \otimes 1_\la \in \,_R \CS_q(>\la)$ 
if 
$1_\la \in \,_\CA \CS_q(>\la)$. 
This implies that 
$_R \vL^+ \subset \, _\CA \vL^+$. 
Since 
$R \otimes_{\CA} \,_\CA \D(\la)$ 
has an element $1 \otimes \ol{1_\la}$, 
we have that 
$\rad \big( R \otimes_{\CA} \,_\CA \D(\la) \big) \not= R \otimes_{\CA} \,_\CA \D(\la)$ 
for any $\la \in \,_\CA\vL^+$. 
This implies that 
$_R \CS_q$ is quasi-hereditary, 
and that 
the number of isomorphism classes of simple $_R \CS_q$-modules 
is equal to $_\CA \vL^+$ by the general theory of cellular algebras. 
On the other hand, 
we know that 
the number of isomorphism classes of simple $_R \CS_q$-modules 
is equal to $_R \vL^+$ 
by Theorem \ref{thm standard simple Sq R}. 
Thus, 
we have $_R \vL^+ = \, _\CA \vL^+$. 
In particular, 
we have 
$_\CA \vL^+= \vL^+$ 
when $R= \CK$. 
\end{proof}

\remarks\
\label{remarks-cartan}
(\roi) 
Let $_\CA \wt{\CS}_q^\natural=\, _\CA \wt{\CS}_q^\natural(\vL)$ be the 
$\CA$-subalgebra of $\wt{\CS}_q$ generated by 
$E_i, F_i, 1_\la, \tau_i^\la$ for $1\leq i \leq m-1, \la \in \vL$. 
Clearly, 
$_\CA \wt{\CS}_q^\natural$ 
is isomorphic to the associative algebra over $\CA$ 
defined by generators $E_i, F_i, 1_\la,  \tau_i^\la$ and 
defining relations (\ref{Sq-1})-(\ref{Sq-9}). 
Moreover, 
$_\CA \wt{\CS}_q^\natural$ 
is a homomorphic image of $_\CA \wt{U}_q^\natural$, 
where 
$_\CA \wt{U}_q^\natural$ is the $\CA$-subalgebra of $\wt{U}_q$ 
generated  by all 
$e_i,f_i, \tau_i, K_j^{\pm}, \left[ \begin{matrix} K_j ;0 \\ t \end{matrix} \right]$.  
For $_\CA \wt{\CS}_q^\natural$, 
we can take $\eta_\vL$, 
and we can define the quotient algebra 
$_\CA \CS_q^{\natural}=\, _\CA \CS_q^{\natural \eta_\vL}$ 
as the case of $\CS_q^{\eta_\vL}$ 
(in this case, 
the condition (A-1) 
for $_\CA \CS_q^\natural$  to have the triangular decomposition 
is unnecessary 
since we do not take a divided power).   
For an arbitrary ring $R$ and parameters $\xi_0,\xi_1,\cdots,\xi_r$, 
we take the specialized algebra 
$_R \CS_q^\natural = R \otimes_{\CA} \,_\CA \CS_q^\natural$. 
Then, 
for $_R \CS_q^\natural$, 
one can apply  similar arguments as in the case of $_R \CS_q$. 
In particular, 
similar results to Theorem \ref{thm standard induced borel R}, Theorem \ref{thm standard simple Sq R}, 
Theorem \ref{thm block class} and  
Theorem \ref{thm cellular R}
hold for $_R \CS_q^\natural$. 
However, 
$_R \CS_q^\natural$ 
is different from $_R \CS_q$ in general. 

(\roii) 
For any Cartan matrix of finite type, 
one can define the algebra $\wt{U}_q$ 
and its quotient algebra $\CS_q$ 
associated to a given Cartan matrix in a similar way. 
In this case, 
we should take a weight lattice $P$ 
whose  rank  
is equal to the rank of the root lattice, 
and  
we take a finite subset $\vL $ of $P$  
to define the quotient algebra $\wt{\CS}_q$  
without 
taking a subset of  $P$ such as $P_{\geq 0}$. 
We should use a similar arguments as in the proof of \cite[Lemma 3.2]{Do} 
instead of Lemma \ref{lem-1-lamda}
in order to prove a similar statement as in Proposition \ref{prop-sur-wtUq-wtSq}. 
We also remove the condition (e) from the definition of $\CO^\eta$. 
Then, 
we have all statements in \S \ref{Sq} and \S \ref{specialization} 
corresponding to a given Cartan matrix.  
  

\section{Review on $q$-Schur algebras of type A}

\para 
Let $n,m$ be  positive integers, 
and 
$\vL_{n,1}$ 
be the set of compositions of $n$ with $m$ parts, 
namely 
\[\vL_{n,1}=\big\{ \mu =(\mu_1,\cdots, \mu_m) \in \ZZ^m_{\geq 0} \bigm| \mu_1+\cdots +\mu_m=n \big\}.\] 
We regard $\vL_{n,1}$ as a subset of $P$ 
by the injective map  
from $\vL_{n,1}$ to $ P$ 
given by  
$\mu=(\mu_1,\cdots, \mu_m) \mapsto \sum_{i=1}^m \mu_i \ve_i$.  
Thus, 
for 
$\mu=(\mu_1,\cdots,\mu_m) \in \vL_{n,1}$ and $\a_i$ ($1 \leq i \leq m-1$),  
we have 
\[\mu \pm \a_i=(\mu_1,\cdots,\mu_{i-1},\, \mu_i \pm 1,\, \mu_{i+1} \mp 1 ,\,\mu_{i+2}, \cdots, \mu_m).\]  

For $\mu \in \vL_{n,1}$, 
the diagram of $\mu$ 
is the set $[\mu]=\{(i,j)\in \NN \times \NN \,|\, 1 \leq j \leq \mu_i, \, 1 \leq i \leq m\}$,  
and 
a $\mu$-tableau is a bijection 
$\Ft : [\mu] \ra \{1,2,\cdots,n\}$. 
Let $\Ft^\mu$ be the $\mu$-tableau 
in which the integers 
$1,2,\cdots, n$ are attached in the way from left to right, 
and top to bottom 
in $[\mu]$.
The symmetric group 
$\FS_n$ acts on the set of $\mu$-tableaux from right by permuting the integers attached in $[\mu]$. 
For $\mu,\nu \in \vL_{n,1}$, 
a $\mu$-tableau of type $\nu$ is a map 
$T : [\mu] \ra \{1,\cdots, m\}$ 
such that 
$\nu_i=\sharp \{ x \in [\mu] \,|\, T(x)=i \}$. 
For $\mu,\nu$ and  $\mu$-tableau $\Ft$, 
let $\nu(\Ft)$ be a $\mu$-tableau of type $\nu$ 
obtained by 
replacing each entry $i$ in $\Ft$ 
by $k$ 
if $i$ appear in the $k$-th row  of $\Ft^\nu$.

For $\mu \in \vL_{n,1}$, 
let 
$\FS_\mu$ 
be the Young subgroup of $\FS_n$ 
corresponding to $\mu$, 
and 
$\SD_\mu$ 
be the set of 
distinguished representatives of right $\FS_\mu$-cosets. 
For $\mu,\nu \in \vL_{n,1}$, 
$\SD_{\mu \nu}=\SD_\mu \cap \SD_\nu^{-1}$ 
is the set of distinguished representatives of $\FS_\mu$-$\FS_\nu$  double cosets.  

\para
Let $R$ be an integral domain, 
and $q$ be an invertible element in $R$. 
The Iwahori-Hecke algebra 
$_R \He_{n}$ 
of the symmetric group $\FS_n$  
is the associative algebra over $R$ 
generated by $T_1,\cdots,T_{n-1}$ with the following defining relations; 
\begin{align*}
&(T_i - q)(T_i + q^{-1}) &&(1 \leq i \leq n-1), \\
&T_i T_{i+1} T_i= T_{i+1} T_i T_{i+1} &&(1\leq i \leq n-2), \\
&T_i T_j=T_jT_i &&(|i-j|\geq 2).
\end{align*}
 
For $w \in \FS_n$, 
we denote by $\ell(w)$ the length of $w$, 
and 
by $T_w$ the standard basis of $_R \He_n$ corresponding to $w$. 
We define an anti-automorphism 
$\ast : \,_R \He_n \ni x \mapsto x^\ast \in \,_R \He_{n}$ 
by 
$T_i^\ast = T_i$ for $i=1,\cdots,n-1$. 
Thus, we have 
$T_w^\ast=T_{w^{-1}}$ for $w \in \FS_n$. 
For 
$\mu \in \vL_{n,1}$, 
set 
$x_\mu=\sum_{w \in \FS_\mu} q^{\ell(w)} T_w$,  
and we define the right $_R \He_{n}$-module 
$M^\mu= x_\mu \cdot \,_R \He_{n}$. 
The $q$-Schur algebra $_R\Sc_{n,1}$ associated to $_R \He_n$ 
is defined by 
\[ _R \Sc_{n,1} = \End_{\,_R \He_{n}}\Big( \bigoplus_{\mu \in \vL_{n,1}} M^\mu \Big). \]

The following lemma is well known (see e.g. \cite[4.6]{M-book}). 
\begin{lem} \label{lem-d}
For $\mu,\nu \in \vL_{n,1}$ and $d \in \SD_{\mu\nu}$, 
let $T=\nu (\Ft^\mu \cdot d)$, $S=\mu(\Ft^\nu\cdot d^{-1})$. 
Then we have 
\[ \sum_{y \in \SD_\nu \atop \mu(\Ft^\nu\cdot y)=S } q^{\ell(y)} T_y^\ast x_\nu  
	= \sum_{ w \in \FS_\mu d \FS_\nu} q^{\ell(w)}T_w 
	=\sum_{x \in \SD_\mu \atop \nu(\Ft^\mu \cdot x)=T} q^{\ell(x)} x_{\mu} T_x.
\]
\end{lem}
Thanks to this lemma, 
for $\mu,\nu \in \vL_{n,1}$ and $d \in \SD_{\mu\nu}$, 
we can define an $_R \He_{n}$-module homomorphism 
$\psi_{\mu,\nu}^d : M^\nu \ra M^\mu$ by 
\begin{align*}
\psi_{\mu,\nu}^d (x_\nu \cdot h) 
&= \Big(\sum_{y \in \SD_\nu \atop \mu(\Ft^\nu \cdot y)=S} q^{\ell(y)} T_y^\ast x_\nu \Big) \cdot h \\
&= \Big( \sum_{x \in \SD_\mu \atop \nu(\Ft^\mu\cdot x)=T} q^{\ell(x)} x_\mu T_x \Big) \cdot h 
	\qquad (h\in \,_R \He_{n}).  
\end{align*}
We extend this homomorphism to an element of $_R \Sc_{n,1}$  by 
$\psi_{\mu,\nu}^d(m_\tau)=0$ 
for 
$m_\tau \in M^\tau$ with 
$\tau \in \vL_{n,1}$ 
such that 
$\tau \not= \nu$.  
It is known that 
$\{ \psi_{\mu,\nu}^d \,|\, \mu,\nu \in \vL_{n,1}, d \in \SD_{\mu\nu}\}$ 
gives a free $R$-basis of $_R \Sc_{n,1}$ (see \cite[Theorem 4.7]{M-book}). 

\para 
Next, we define the Borel subalgebras of $_R\Sc_{n,1}$ following \cite{DR}.
Let 
$I(m,n)=\{\Bi=(i_1,\cdots,i_n)\,|\, 1 \leq i_k \leq m \text{ for }1 \leq k \leq n\}$. 
$\FS_n$ acts on $I(m,n)$ from right by 
$\Bi\cdot w =(i_{w(1)},\cdots, i_{w(n)})$ 
for 
$\Bi=(i_1,\cdots, i_n)\in I(m,n)$ 
and 
$w \in \FS_n$.  
We define a partial order \lq\lq $\succeq $" on $I(m,n)$ by 
\[ (i_1,\cdots,i_n) \succeq (j_1,\cdots,j_n) \text{ if and only if } 
i_k \geq j_k \text{ for all } k=1,\cdots, n.\]
For $\la \in \vL_{n,1}$, 
put  
\[ \Bi_\la=(\underbrace{1,\cdots,1}_{\la_1 \text{terms}}, \underbrace{2,\cdots,2}_{\la_2 \text{terms}}, 
\cdots, \underbrace{m,\cdots,m}_{\la_m \text{terms}} ).\]
For $\mu \in \vL_{n,1}$, 
we set 
\begin{align*}
&\Om_1^{\succeq}(\mu) = \big\{(\la,d) \bigm| \la \in \vL_{n,1}, d \in \SD_{\la\mu} 
	\text{ such that } \Bi_\la \cdot d \succeq \Bi_{\mu} \big\}, \\
&\Om_1^{\preceq }(\mu) = \big\{(\la,d) \bigm| \la \in \vL_{n,1}, d \in \SD_{\la\mu} 
	\text{ such that } \Bi_\mu \cdot d \preceq \Bi_{\la} \big\}.
\end{align*}

Let 
$_R \Sc_{n,1}^{\leq 0} $
be the free $R$-submodule of $_R\Sc_{n,1}$ spanned by 
$\{ \psi_{\la,\mu}^d \,|\, (\la,d)\in \Om_1^{\succeq}(\mu),\, \mu \in \vL_{n,1}\}$, 
and 
$_R\Sc_{n,1}^{\geq 0} $
be the free $R$-submodule of $_R\Sc_{n,1}$ spanned by 
$\{ \psi_{\mu,\la}^d \,|\, (\la,d)\in \Om_1^{\preceq}(\mu),\, \mu \in \vL_{n,1}\}$. 
By \cite[Theorem 2.3]{DR}, 
$_R\Sc_{n,1}^{\leq 0} $ (resp. $_R\Sc_{n,1}^{\geq 0}$) 
becomes a subalgebra of $_R\Sc_{n,1}$. 

\para 
We denote $_{\QQ(q)} \Sc_{n,1}$ (resp. $_{\QQ(q)} \Sc_{n,1}^{\leq 0}, \, _{\QQ(q)} \Sc_{n,1}^{\geq 0}$) 
simply 
by $\Sc$ (resp. $\Sc_{n,1}^{\leq 0}, \, \Sc_{n,1}^{\geq 0}$).
The following theorem is known by several authors.

\begin{thm}[\cite{J}, \cite{Du}, \cite{PW},\cite{DR}, \cite{DP}]  \label{Uqgl-S1}\
\label{Uqgl-S1}
\begin{enumerate}
\item
There exists a surjective homomorphism of algebras 
\[\rho : U_q \ra \Sc_{n,1}.\]

\item 
By restricting $\rho$ to $\CB^{\pm}$, we have the surjective homomorphisms 
\[\rho|_{\CB^+} : \CB^+ \ra \Sc_{n,1}^{\geq 0}, \quad \rho|_{\CB^-} : \CB^- \ra \Sc_{n,1}^{\leq 0}. \]

\item
By restricting $\rho$ to $_\CZ U_q$, we have the surjective homomorphism 
\[\rho|_{_\CZ U_q} : \,_\CZ U_q \ra \,_\CZ \Sc_{n,1}.\]

\item 
By restricting $\rho$ to $_\CZ\CB^{\pm}$, we have the surjective homomorphisms 
\[\rho|_{_\CZ\CB^+} : \,_\CZ\CB^+ \ra \,_\CZ\Sc_{n,1}^{\geq 0}, 
	\quad \rho|_{_\CZ\CB^-} : \,_\CZ\CB^- \ra \,_\CZ\Sc_{n,1}^{\leq 0}. \]
\end{enumerate}
\end{thm}

We can describe precisely the image of generators of $U_q$ 
under the homomorphism $\rho$ in Theorem \ref{Uqgl-S1} 
 as follows. 

\begin{prop}[{\cite{S2}}] \label{prop-rho}\

\begin{enumerate}
\item 
For $e_i$ $(1 \leq i \leq m-1)$, we have 
\[ \rho(e_i)=\sum_{\mu \in \vL_{n,1}} q^{-\mu_{i+1}+1} \psi_{\mu+\a_i,\mu}^1\,\,,\]
where if $\mu+\a_i \not\in \vL_{n,1}$, we define $\psi_{\mu+\a_i, \mu}^1=0$. 

\item 
For $f_i$ $(1 \leq i \leq m-1)$, we have 
\[ \rho (f_i) = \sum_{\mu \in \vL_{n,1}} q^{-\mu_i+1} \psi_{\mu-\a_i, \mu}^1\,\,, \]
where if $\mu-\a_i \not\in \vL_1$, we define $\psi_{\mu-\a_i,\mu}^1=0$. 

\item 
For $K_{i}^{\pm}$ $(1 \leq i \leq m)$, we have 
\[ \rho(K_i^{\pm})=\sum_{\mu \in \vL_{n,1}} q^{\pm \mu_i}\psi_{\mu, \mu}^1 \,\,.\]
Clearly, $\psi_{\mu, \mu}^1$ is an identity map on $M^\mu$. 
\end{enumerate}
\end{prop}
\begin{proof}
See Appendix \ref{proof-prop-rho}. 
\end{proof}

\para 
By Theorem \ref{Uqgl-S1}, 
the 
$q$-Schur algebra $\Sc_{n,1}$ 
is a quotient algebra of 
$U_q$. 
Thus, 
$\Sc_{n,1}$ 
is generated by 
the generators of $U_q$. 
In \cite{DG}, 
Doty and Giaquinto 
described the kernel of the surjection $\rho : U_q \ra\Sc_{n,1}$ precisely. 
Moreover, 
they also gave a presentation of the $q$-Schur algebra $_\CZ \Sc_{n,1}$ over $\CZ$.  

\begin{thm}[{\cite[Theorem 3.1, Theorem 3.3]{DG}}] \label{S1-presen}\

\begin{enumerate}
\item 
The $q$-Schur algebra $\Sc_{n,1}$ is 
isomorphic to 
the associative algebra over $\QQ(q)$ 
generated by 
$e_i,f_i$ $(1\leq i \leq m-1)$ and $K^{\pm}_i$ $(1 \leq i \leq m)$ 
with the defining relations $(\ref{gl-1})$-$(\ref{gl-6})$ together with 
the following two relations.  
\begin{align}
& K_1K_2\cdots K_m=q^n, 
\label{add Sn1-1}\\
& (K_i-1)(K_i-q)(K_i-q^2)\cdots(K_i-q^n)=0. 
\label{add Sn1-2}
\end{align}

\item
$_\CZ \Sc_{n,1}$ is the $\CZ$-subalgebra of $\Sc_{n,r}$ 
generated by all 
$e_i^{(k)},f_i^{(k)},K_j^{\pm}$ and $\left[ \begin{matrix} K_j ; 0\\ t \end{matrix} \right]$ 
for $1 \leq i \leq m-1,\, 1 \leq j \leq m,\, k \geq 1,\, t \geq 1$.
\end{enumerate}
\end{thm}

In \cite{DG}, 
they gave an alternative presentation of $\Sc_{n,1}$ by generators and relations as follows. 

\begin{thm}[{\cite[Theorem 3.4]{DG}}] \label{thm-DG-2}\

\begin{enumerate}
\item 
The $q$-Schur algebra $\Sc_{n,1}$ is 
isomorphic to 
an associative algebra over $\QQ(q)$ 
generated by 
$E_i,F_i$ $(1\leq i \leq m-1)$ and $1_\la$ $(\la \in \vL_{n,1})$ with the following defining relations: 
\begin{align*}
&1_\la 1_\mu=\d_{\la\mu}1_\la, \quad \sum_{\la \in \vL_{n,1}}1_\la =1, \\
&E_i 1_\la=
	\begin{cases}
		1_{\la+\a_i}E_i &\text{ if } \la+\a_i \in \vL_{n,1}, \\
		0 &\text{ otherwise }, 
	\end{cases}\\
&F_i 1_\la=
	\begin{cases}
		1_{\la-\a_i}F_i &\text{ if } \la-\a_i \in \vL_{n,1}, \\
		0 &\text{ otherwise }, 
	\end{cases}\\
&1_\la E_i=
	\begin{cases}
		E_i 1_{\la-\a_i} &\text{ if } \la-\a_i \in \vL_{n,1}, \\
		0 &\text{ otherwise }, 
	\end{cases}\\
&1_\la F_i=
	\begin{cases}
		F_i 1_{\la+\a_i} &\text{ if } \la+\a_i \in \vL_{n,1}, \\
		0 &\text{ otherwise }, 
	\end{cases}\\
&E_iF_j-F_jE_i=\d_{ij} \sum_{\la \in \vL_{n,1}}[\la_i - \la_{i+1}] 1_\la \\
&E_{i \pm 1}E_i^2 - (q+q^{-1}) E_i E_{i\pm1} E_i + E_i^2 E_{i \pm 1}=0 \\
&E_i E_j= E_j E_i \qquad (|i-j| \geq 2) \notag \\
&F_{i \pm 1}F_i^2 - (q+q^{-1}) F_i F_{i \pm 1} F_i + F^2_i F_{i \pm 1}=0  \\
&F_i F_j= F_j F_i \qquad ( |i-j|\geq 2)  \notag 
\end{align*} 

\item
$_\CZ \Sc_{n,1}$ is the $\CZ$-subalgebra of $\Sc_{n,1}$ 
generated by all 
$E_i^{(k)},F_i^{(k)}$ $(1 \leq i \leq m-1,\, k \geq 1)$ and $1_\la$ $(\la \in \vL_{n,1})$. 
\end{enumerate}
\end{thm}

\remark 
\label{remark condition r=1}
For $\la \in \vL_{n,1}$ and  $i=1,\cdots,m-1$, 
put 
$\eta_i^\la =[\la_{i} - \la_{i+1}] 1_\la$, 
and 
$\eta_{\vL_{n,1}} =\{\eta_i^\la \,|\, 1 \leq i \leq m-1, \, \la \in \vL_{n,1}\}$. 
It is clear that 
$\Sc_{n,1}$ 
is isomorphic to 
$\CS_{q}^{\eta_{\vL_{n,1}}} $ 
defined in \ref{def Sq}. 
Clearly, 
$\CS_{q}^{\eta_{\vL_{n,1}}} $ 
satisfies the condition (C-1). 
It is known that 
the $q$-Schur algebra $_\CZ \Sc_{n,1}$ over $\CZ$
has a triangular decomposition  which coincides with 
the triangular decomposition of $_\CZ\CS_q$ in Proposition \ref{prop tri A}, 
and that $_\CZ \Sc_{n,1}$ is a cellular algebra. 
Moreover, 
$_\CZ\D(\la)$ for $\la \in \vL_{n,1}^+$ 
coincides with a cell module of $_\CZ \Sc_{n,1}$ 
thanks to Theorem \ref{thm standard induced borel R}. 
In particular, 
$\vL_{n,1}^+$ coincides with the set of partitions of size $n$  
(see \cite{DR} and \cite{M-book} for the results on $q$-Schur algebra $_\CZ \Sc_{n,1}$). 
Thus, 
$\Sc_{n,1} \big(\cong \CS_q^\eta(\vL_{n,1})\big)$ 
satisfies the conditions (A-1), (A-2) and (C-1). 
\\

In \cite{DP},  a presentation of Borel subalgebras $\Sc_{n,1}^{\leq 0}$ and $\Sc_{n,1}^{\geq 0}$ 
was given as follows.

\begin{thm}[{\cite[Theorem 8.1]{DP}}] \label{thm gen rel of Borel} 
The Borel subalgebra $\Sc_{n,1}^{\leq 0}$ (resp. $\Sc_{n,1}^{\geq 0}$) 
is isomorphic to the associative algebra generated by  
$f_i$ (resp. $e_i$) $(1\leq i \leq m-1)$ and $K_i^{\pm}$ $(1 \leq i \leq m)$ with the defining relations 
\eqref{gl-1}, \eqref{gl-3}, \eqref{gl-6}, 
\eqref{add Sn1-1} and \eqref{add Sn1-2} 
(resp. \eqref{gl-1}, \eqref{gl-2}, \eqref{gl-5}, \eqref{add Sn1-1} and \eqref{add Sn1-2}). 
\end{thm}

\remark 
The above presentation of Borel subalgebras is not exactly the same as the one  given in \cite[Theorem 8.1]{DP}. 
However, it is equivalent to the presentation in [loc. cit.] (see \cite[Remarks 4.4]{DP}).


\section{Review on Cyclotomic $q$-Schur algebras} 
\label{cyclotomic q-Schur}

\para 
Let $R$ be an integral domain, 
and 
we take parameters 
$q,Q_1,\cdots, Q_r \in R$, 
where 
$q$ is invertible in $R$. 
The Ariki-Koike algebra $_R \He_{n,r}$ associated to 
$\FS_n \ltimes (\ZZ/r\ZZ)^n$ 
is the associative algebra with 1 over $R$ 
generated by 
$T_0,T_1,\cdots,T_{n-1}$ 
with the following defining relations: 
\begin{align*}
&(T_0-Q_1)(T_0-Q_2)\cdots (T_0-Q_r)=0, \\
&(T_i-q)(T_i+q^{-1})=0 &&(1 \leq i \leq n-1),\\
&T_0 T_1 T_0 T_1=T_1 T_0 T_1 T_0, \\
&T_i T_{i+1} T_i = T_{i+1} T_i T_{i+1} &&(1\leq i \leq n-2),\\
&T_i T_j=T_j T_i &&(|i-j|\geq 2). 
\end{align*}

The subalgebra of 
$_R \He_{n,r}$ 
generated by 
$T_1,\cdots,T_{n-1}$ 
is isomorphic to 
the Iwahori-Hecke algebra 
$_R \He_{n}$. 
We define an algebra anti-automorphism 
$\ast : \,_R \He_{n,r} \ni x \mapsto x^\ast \in \,_R \He_{n,r}$ 
by 
$T_i^\ast = T_i$ for $i=0,\cdots,n-1$. 

\para 
\label{vLnr to vLn1}
Put
\[ \vL_{n,r}=\left\{ \mu=(\mu^{(1)},\cdots,\mu^{(r)}) \Biggm| 
	\begin{matrix}
		\mu^{(k)}=(\mu_1^{(k)},\cdots,\mu_n^{(k)}) \in \ZZ_{\geq 0}^n \\
	    \sum_{k=1}^r \sum_{i=1}^n \mu_i^{(k)}=n 
	\end{matrix}	
	\right\}. 
\]
Thus, 
$\vL_{n,r}$ is a set of $r$-tuples of compositions with $n$ parts 
whose size is equal to $n$.  
Put 
$m=rn$ and 
$p_k=(k-1)n$ for $k=1,\cdots, r$.  
Then, 
there exists a bijection from 
$\vL_{n,r}$ to $\vL_{n,1}$ 
such that  
$\mu \mapsto \ol{\mu}$, 
where 
$\ol{\mu}=(\ol{\mu}_1,\ol{\mu}_2,\cdots,\ol{\mu}_m) \in \vL_{n,1}$ 
obtained by  
$\ol{\mu}_{p_k+i}= \mu^{(k)}_i$. 

\para 
For $i=1,\cdots,n$, 
put 
$L_1=T_0$ 
and  
$L_i=T_{i-1} L_{i-1} T_{i-1}$. 
For $\mu \in \vL_{n,r}$, 
put 
\[u_{\mu}^+=\prod_{k=1}^r \prod_{i=1}^{a_k}(L_i - Q_k), 
\quad m_{\mu}=x_{\ol{\mu}} \, u_{\mu}^+, \quad 
M^\mu = m_\mu \cdot \,_R \He_{n,r}, 
\] 
where $a_k=\sum_{j=1}^{k-1}|\mu^{(j)}|$ 
with 
$a_1=0$.  
Note that 
$(m_\mu)^\ast =m_\mu$, 
and 
we define 
$(M^\mu)^\ast =\,_R \He_{n,r} \cdot m_\mu$. 
The cyclotomic $q$-Schur algebra 
$_R \Sc_{n,r}$ 
associated to 
$_R \He_{n,r}$ 
is defined by 
\[ _R \Sc_{n,r}= \End_{_R \He_{n,r}} \Big( \bigoplus_{\mu \in \vL_{n,r}} M^\mu \Big).\]

The following lemma is well known, 
and one can check them in direct calculations 
by using the defining relations of $_R \He_{n,r}$. 
\begin{lem} \label{lem-L}\

\begin{enumerate}
\item 
$L_i$ and $L_j$ commute with each other for any $1 \leq i,j \leq n$
\item 
$T_i$ and $L_j$ commute with each other if $j\not=i,i+1$. 
\item 
$T_i$ commute with 
both of  
$L_iL_{i+1}$ and $L_i + L_{i+1}$.
\item
For $a \in R$ and $i=1,\cdots,n-1$, 
$T_i$ commutes with 
$\prod_{j=1}^k(L_j-a)$ 
if $k\not=i$.  

\item 
$L_{i+1}T_i=(q-q^{-1})L_{i+1}+ T_i L_i$, 
\quad 
$T_i L_{i+1}=(q-q^{-1})L_{i+1}+ L_i T_i$.

\item 
$L_i T_i=(q^{-1}-q)L_{i+1}+T_i L_{i+1}$, 
\quad 
$T_i L_i=(q^{-1}-q)L_{i+1}+L_{i+1}T_i$. 
\end{enumerate}
\end{lem}

\para \label{def-vf}
For 
$\la,\mu \in \vL_{n,r}$ 
and 
$d \in \SD_{\ol{\la} \ol{\mu}}$ 
such that 
$\Bi_{\ol{\la}} \cdot d \succeq \Bi_{\ol{\mu}}$, 
we define 
$\vf_{\la,\mu}^d \in \,_R \Sc_{n,r}$ by 
\[ \vf_{\la,\mu}^d (m_{\nu}\cdot h)= 
	\d_{\mu\nu} \Big( \sum_{w \in \FS_{\ol{\la}} d \FS_{\ol{\mu}}} q^{\ell(w)} T_w \Big) u_{\mu}^+ \cdot h 
\qquad (\nu \in \vL_{n,r},\, h \in \,_R \He_{n,r}) .
\]
This definition is well-defined by Lemma \ref{lem-d},   
and  
we have 
$\vf_{\la,\mu}^d \in \Hom_{_R \He_{n,r}}(M^\mu, M^\la)$  
by \cite[Lemma 5.6]{DR}. 

For 
$\la,\mu \in \vL_{n,r}$ 
and 
$ d \in \SD_{\ol{\mu}\ol{\la}}$ 
such that 
$\Bi_{\ol{\la}} \succeq \Bi_{\ol{\mu}} \cdot d$, 
we have 
$\Bi_{\ol{\la}} \cdot d^{-1} \succeq \Bi_{\ol{\mu}}$ 
and 
$d^{-1} \in \SD_{\ol{\la}\ol{\mu}}$ 
from definitions immediately.  
Thus, 
we can define 
$\vf_{\la,\mu}^{d^{-1}} \in \Hom_{_R \He_{n,r}}(M^\mu, M^\la)$ 
as above.  
On the other hand, 
by 
\cite[Corollary 5.17]{DJM98}, 
we have 
$\vf_{\la,\mu}^{d^{-1}}(m_\mu) \in (M^\mu)^\ast \cap M^\la$, 
hence 
$\big(\vf_{\la,\mu}^{d^{-1}}(m_\mu)\big)^\ast \in M^\mu \cap (M^\la)^\ast$.
Thus, we define 
$\vf'^{d}_{\mu,\la} \in \,_R \Sc_{n,r}$ 
by 
\[ \vf'^d_{\mu,\la}(m_{\nu} \cdot h)= \d_{\la\nu} 
	\big( \vf_{\la,\mu}^{d^{-1}}(m_{\mu})\big)^\ast \cdot h 
		\qquad (\nu \in \vL_{n,r},\, h \in \,_R \He_{n,r}),\]
and we have 
$\vf'^{d}_{\mu\la} \in \Hom_{_R \He_{n,r}}(M^\la,M^\mu)$. 

Let 
$_R \Sc_{n,r}^{\leq 0}$ 
(resp. $_R \Sc_{n,r}^{\geq 0}$) 
be the free $R$-submodule 
of 
$_R \Sc_{n,r}$ 
spanned by 
$\{\vf_{\la,\mu}^d \,|\, (\ol{\la},d) \in \Om^{\succeq}(\ol{\mu}), \mu \in \vL_{n,r} \}$ 
(resp. $\{\vf'^d_{\mu,\la} \,|\, (\ol{\la},d) \in \Om^{\preceq}(\ol{\mu}), \mu \in \vL_{n,r} \}$ ). 
Then 
$_R \Sc_{n,r}^{\leq 0}$ 
(resp. $_R \Sc_{n,r}^{\geq 0}$) 
is a subalgebra of 
$_R \Sc_{n,r}$, 
and 
$\{\vf_{\la,\mu}^d \,|\, (\ol{\la},d) \in \Om^{\succeq}(\ol{\mu}), \mu \in \vL_{n,r} \}$ 
(resp. $\{\vf'^d_{\mu,\la} \,|\, (\ol{\la},d) \in \Om^{\preceq}(\ol{\mu}), \mu \in \vL_{n,r} \}$ ) 
gives a free $R$-basis of 
$_R \Sc_{n,r}^{\leq 0}$ 
(resp. $_R \Sc_{n,r}^{\geq 0}$)  
by 
\cite[Lemma 5.12, Theorem 5.13]{DR}. 

Moreover, 
in \cite{DR}, 
Du and Rui 
proved the following theorem. 

\begin{thm}[{\cite[Theorem 5.13, 5.16]{DR}}] \label{thm-iso-Borel}\

\begin{enumerate}
\item 
There exists an algebra isomorphism 
$\CF^{\leq 0} : \,_R \Sc_{n,r}^{\leq 0} \ra \,_R \Sc_{n,1}^{\leq 0}$ 
such that 
$\CF^{\leq 0 }(\vf_{\la,\mu}^d)=\psi_{\ol{\la},\ol{\mu}}^d$ 
for 
$\vf_{\la,\mu}^d \in \{\vf_{\la,\mu}^d \,|\, (\ol{\la},d) \in \Om^{\succeq}(\ol{\mu}), \mu \in \vL_{n,r} \}$. 

\item 
There exists an algebra isomorphism 
$\CF^{\geq 0} : \,_R \Sc_{n,r}^{\geq 0} \ra \,_R \Sc_{n,1}^{\geq 0}$ 
such that 
$\CF^{\geq 0}(\vf'^d_{\mu,\la})=\psi_{\ol{\mu},\ol{\la}}^d$ 
for 
$\vf'^d_{\mu,\la} \in \{\vf'^d_{\mu,\la} \,|\, (\ol{\la},d) \in \Om^{\preceq}(\ol{\mu}), \mu \in \vL_{n,r} \}$. 

\item 
$_{R}\Sc_{n,r}$ 
has a triangular decomposition 
\[_R \Sc_{n,r} = \, _R \Sc_{n,r}^{\leq 0} \cdot \,_R \Sc_{n,r}^{\geq 0} 
= \sum_{\la \in \vL_{n,r}} \,_R \Sc_{n,r}^{\leq 0} \cdot \vf_{\la,\la}^1 \cdot \,_R \Sc_{n,r}^{\geq 0}.\]
\end{enumerate}

\end{thm}



\section{A cyclotomic $q$-Schur algebra as a quotient algebra of $\wt{U}_q$ }
\para 
As in the previous section, 
let $n,r$ be positive integers,  
and 
put 
$m=nr$.
Let 
$\vG = \{1,\cdots,n\} \times \{1,\cdots,r\}$, 
and 
$\vG'=\vG\setminus \{(n,r)\}$. 
As a convention, 
we set 
$(n+1,k)=(1,k+1)$ and 
$(0,k+1)=(n,k)$ 
for $k=1,\cdots,r-1$. 
For 
$(i,k)\in \vG$, 
put 
$\ve_{(i,k)} =\ve_{p_k +i}$, 
where $p_k=(k-1)n$.   
Thus, 
we can rewrite the weight lattice $P$ 
by 
$P=\bigoplus_{(i,k)\in \vG} \ZZ \ve_{(i,k)}$, 
and 
we regard 
$\vL_{n,r}$ 
as a subset of $P$ 
by 
the injective map 
from 
$\vL_{n,r}$ to $P$ given by 
$\vL_{n,r} \ni \mu \mapsto \sum_{(i,k) \in \vG} \mu_i^{(k)} \ve_{(i,k)} \in P$. 
For $(i,k) \in \vG$, 
put 
$h_{(i,k)}=h_{p_k+i}$, 
then 
the dual weight lattice $P^{\vee}$ 
can be rewritten 
as 
$P^\vee = \bigoplus_{(i,k) \in \vG} \ZZ h_{(i,k)}$.  
Moreover, 
for 
$(i,k) \in \vG'$, 
put 
$\a_{(i,k)}=\a_{p_k+i}=\ve_{(i,k)} - \ve_{(i+1,k)}$. 
Thus, 
for $\mu \in \vL_{n,r}$, 
$\mu \pm \a_{(i,k)}$ makes sense in $P$.

\para 
For $\mu \in \vL_{n,r}$ and $(i,k) \in \vG'$,  
if $\mu + \a_{(i,k)} \in \vL_{n,r}$ 
then 
we have 
$\Bi_{\ol{\mu}} \succeq \Bi_{\ol{\mu +\a_{(i,k)}}}$ 
from definitions.   
On the other hand, 
if $\mu - \a_{(i,k)} \in \vL_{n,r}$ 
then 
we have 
$\Bi_{\ol{\mu - \a_{(i,k)}}} \succeq \Bi_{\ol{\mu}}$.
Then, 
for $(i,k) \in \vG'$, 
we define an element $\vf^{\pm}_{(i,k)} \in \,_R\Sc_{n,r}$ by 
\begin{align*}
&\vf^+_{(i,k)} =\sum_{\mu \in \vL_{n,r}} q^{-\mu_{i+1}^{(k)}+1} \vf'^{1}_{\mu+\a_{(i,k)}, \, \mu}\,\,,\\
&\vf^-_{(i,k)} = \sum_{\mu \in \vL_{n,r}} q^{-\mu_{i}^{(k)}+1} \vf^1_{\mu-\a_{(i,k)},\,\mu}\,\,, 
\end{align*}
where 
we define 
$\vf'^{1}_{\mu+\a_{(i,k)},\,\mu}=0$ 
(resp. $\vf^1_{\mu-\a_{(i,k)},\,\mu}=0$) 
if 
$\mu+\a_{(i,k)} \not\in \vL_{n,r}$ 
(resp. $\mu-\a_{(i,k)} \not\in \vL_{n,r}$ ).

For $(i,k)\in \vG$, 
we define $\ka^{\pm}_{(i,k)} \in \, _R \Sc_{n,r}$ by 
\[ \ka^{\pm}_{(i,k)}=\sum_{\mu \in \vL_{n,r}} q^{\pm \mu_i^{(k)}} \vf_{\mu,\mu}^1,\]
and write $\ka^{+}_{(i,k)}$ by $\ka_{(i,k)}$ for simplicity. 

For $\mu \in \vL_{n,r}$ and $(i,k) \in \vG$, 
put 
$N=\sum_{l=1}^{k-1} |\mu^{(l)}| + \sum_{j=1}^{i-1} \mu^{(k)}_j$. 
By Lemma \ref{lem-L}, 
one sees that 
$(L_{N+1}+L_{N+2}+\cdots +L_{N+\mu_i^{(k)}})$ 
commutes with $m_{\mu}$. 
Thus, 
we can define $\s_{(i,k)}^\mu \in \,_R \Sc_{n,r}$ by 
\[\s_{(i,k)}^\mu (m_{\nu} \cdot h)= 
	\d_{\mu,\nu} \big(m_{\mu}(L_{N+1}+\cdots +L_{N+\mu_i^{(k)}})\big)\cdot h 
	\quad (\nu \in \vL_{n,r} \,\, h \in \,_R \He_{n,r}), 
	\]
where we define 
$\s_{(i,k)}^\mu=0$ if $\mu_i^{(k)}=0$. 
Moreover, we define 
\[ \s_{(i,k)}=\sum_{\mu \in \vL_{n,r}} \s_{(i,k)}^\mu.\]

\para 
Recall that 
$\CA=\CZ[\g_1,\cdots,\g_r]$ 
is a polynomial ring over $\CZ=\ZZ[q,q^{-1}]$ 
with indeterminate elements 
$\g_1,\cdots,\g_r$, 
and that 
$\CK=\QQ(q,\g_1,\cdots,\g_r)$ 
is the quotient field of $\CA$. 
We denote $_\CK \Sc_{n,r}$ simply by $\Sc_{n,r}$, 
where we set $Q_i=\g_i$ ($1 \leq i \leq r$). 
Now, we can define a surjective homomorphism of $\CK$-algebras 
from $\wt{U}_q$ to $\Sc_{n,r}$ as in the following proposition. 

\begin{prop} \label{prop-main-1}
There exists 
a surjective homomorphism 
$\wt{\rho} : \wt{U}_q \ra \Sc_{n,r}$ 
such that, 
for $(i,k) \in \vG'$,  
\begin{align}
&\wt{\rho}(e_{p_k+i})= \vf_{(i,k)}^+, \label{wtrho-1} \\
&\wt{\rho}(f_{p_k+i})= \vf_{(i,k)}^-, \\
&\wt{\rho}(\t_{p_k+i})=
	\begin{cases}
		\displaystyle 
			-\g_{k+1} \frac{\ka_{(n,k)}\ka_{(1,k+1)}^- -\ka_{(n,k)}^- \ka_{(1,k+1)}}{q-q^{-1}} \\[5mm]
			\hspace{3em}+\ka_{(n,k)}\ka^-_{(1,k+1)}(q^{-1} \s_{(n,k)} - q \s_{(1,k+1)}) & (\text{ if } i=n),
\\[5mm]
		\displaystyle 
			\frac{\ka_{(i,k)}\ka_{(i+1,k)}^- -\ka_{(i,k)}^- \ka_{(i+1,k)}}{q-q^{-1}} & (\text{ otherwise }),
	\end{cases}
\end{align}
and that,  
for $(i,k)\in \vG$, 
\begin{align}
&\wt{\rho}(K_{p_k+i}^{\pm})=\ka_{(i,k)}^{\pm}. \label{wtrho-5}\hspace{17em}
\end{align}

Moreover, 
by restricting $\wt{\rho}$ to $_\CA \wt{U}_q$, 
$\wt{\rho}|_{_\CA \wt{U}_q}$ gives a surjective homomorphism 
from $_\CA \wt{U}_q$ to $_\CA \Sc_{n,r}$.  
\end{prop}

\para 
The rest of this section 
is devoted to the proof of the proposition. 
The following relations are clear 
from the definitions. 
\begin{align}
\ka_{(i,k)}\ka_{(j,l)}=\ka_{(j,l)}\ka_{(i,k)}, 
\quad 
\ka_{(i,k)}\ka_{(i,k)}^- = \ka_{(i,k)}^- \ka_{(i,k)} =1 \label{UU-2}
\end{align}

Since $\vf_{\nu,\nu}^1$ is the identity map on $M^\nu$ 
and 
$\s_{(i,k)}^\mu \in \Hom_{\He_{n,r}}(M^\mu, M^\mu)$,  
we have 
\begin{align*} 
\s_{(i,k)}^\mu \, \vf_{\nu,\nu}^1 = \vf_{\nu, \nu}^1 \s_{(i,k)}^\mu = \d_{\mu,\nu}\, \s_{(i,k)}^\mu.
\end{align*} 
This relation combined with (\ref{UU-2}) implies that 
\begin{align}
\ka_{(j,l)} \big(\ka_{(n,k)} & \ka^{-}_{(1,k+1)}(q^{-1} \s_{(n,k)} - q \s_{(1,k+1)})\big) \label{UU-5} \\
&= \big(\ka_{(n,k)}\ka^{-}_{(1,k+1)}(q^{-1} \s_{(n,k)} - q \s_{(1,k+1)})\big)  \ka_{(j,l)}. \notag
\end{align}

\para 
By the definitions of $\vf_{(i,k)}^{\pm}, \ka_{(i,k)}^{\pm}$, 
it is clear that 
$\vf_{(i,k)}^+$ (resp. $\vf_{(i,k)}^-$) for $(i,k)\in \vG'$ 
is included in 
$\Sc_{n,r}^{\geq 0}$ 
(resp. $\Sc_{n,r}^{\leq 0}$), 
and that 
$\ka_{(i,k)}^{\pm}$ for $(i,k)\in \vG$ 
is included in 
both of 
$\Sc_{n,r}^{\geq 0}$ and 
$\Sc_{n,r}^{\leq 0}$. 
Recall, 
in the case of type A,  
that 
there exists a surjective homomorphism 
$\rho : U_q \ra \Sc_{n,1}$ (Theorem \ref{Uqgl-S1}). 
Here, we extend this homomorphism to that over $\CK$. 
By using the isomorphism $\CF^{\geq 0} :  \Sc_{n,r}^{\geq 0} \ra \, _\CK \Sc_{n,1}^{\geq 0}$ 
(resp. $\CF^{\leq 0} : \Sc_{n,r}^{\leq 0} \ra\, _\CK \Sc_{n,1}^{\leq 0}$) 
in  Theorem \ref{thm-iso-Borel}, 
we have the following proposition. 

\begin{prop}\ \label{prop-surjective-Borel}

\begin{enumerate}
\item 
$\Sc_{n,r}^{\geq 0 }$ 
is generated by 
$\vf^{+}_{(i,k)}$ $((i,k)\in \vG')$ 
and 
$\ka^{\pm}_{(i,k)}$ $((i,k)\in \vG)$. 

\item 
$\Sc_{n,r}^{\leq 0}$ 
is generated by 
$\vf^{-}_{(i,k)}$ $((i,k)\in \vG')$ 
and 
$\ka^{\pm}_{(i,k)}$ $((i,k)\in \vG)$. 
\end{enumerate}
\end{prop}

\begin{proof}
We show  (\roi) only since (\roii) is shown  in a similar way. 
By the above arguments, 
$\vf_{(i,k)}^+$ and $\ka_{(i,k)}^{\pm}$ are elements of $\Sc_{n,r}^{\geq 0}$. 
On the other hand, 
by Proposition \ref{prop-rho} and Theorem \ref{thm-iso-Borel}, 
we have 
$\big((\CF^{\geq 0})^{-1} \circ \rho\big)\,(e_{p_k+i})=\vf_{(i,k)}^+$ 
and 
$\big((\CF^{\geq 0})^{-1} \circ \rho\big)\,(K_{p_k+i}^{\pm})=\ka_{(i,k)}^{\pm}$. 
Moreover, 
$_\CK \Sc_{n,1}^{\geq 0}$ 
is the image of $\CB^+$ under $\rho$  
by Theorem \ref{Uqgl-S1} (\roii),
and 
$\CB^+$ 
is generated by 
$e_i$ $(1\leq i \leq m-1)$ 
and 
$K_i^{\pm}$ $(1 \leq i \leq m)$. 
This implies (\roi). 
\end{proof}

\para 
In the proof of the above proposition, 
we have a surjection 
$(F^{\geq 0})^{-1} \circ \rho : \CB^+ \ra \Sc_{n,r}^{\geq 0}$. 
Under this surjection,  
the relations $(\ref{gl-2})$ and $(\ref{gl-5})$ 
implies the following relations $(\ref{UU-3})$ and $(\ref{UU-7})$. 
Similarly, 
the following relations $(\ref{UU-4})$ and $(\ref{UU-8})$ 
follows from the relations $(\ref{gl-3})$ and $(\ref{gl-6})$.  
\begin{align}
&\ka_{(i,k)} \vf_{(j,l)}^+ \ka_{(i,k)}^-=q^{\lan \a_{(j,l)} \,,\, h_{(i,k)} \ran} \vf_{(j,l)}^+, \label{UU-3}\\[1mm]
&\ka_{(i,k)} \vf_{(j,l)}^- \ka_{(i,k)}^-=q^{- \lan \a_{(j,l)} \,,\, h_{(i,k)} \ran} \vf_{(j,l)}^-,
	\label{UU-4}\\[1mm]
&\vf_{(i\pm 1,k)}^+ (\vf_{(i,k)}^+)^2 
	-(q+q^{-1}) \vf_{(i,k)}^+ \vf_{(i \pm 1,k)}^+ \vf_{(i,k)}^+ 
	+(\vf_{(i,k)}^{+})^2 \vf_{(i \pm 1,k)}^+ 
	=0, \label{UU-7} \\[1mm]
&\vf_{(i,k)}^+ \vf_{(j,l)}^+ = \vf_{(j,l)}^+ \vf_{(i,k)}^+ \qquad (|(p_k+i)-(p_l-j)| \geq 2), \notag \\[1mm]
&\vf_{(i\pm 1,k)}^- (\vf_{(i,k)}^-)^2 
	-(q+q^{-1}) \vf_{(i,k)}^- \vf_{(i \pm 1,k)}^- \vf_{(i,k)}^- 
	+(\vf_{(i,k)}^{-})^2 \vf_{(i \pm 1,k)}^- 
	=0, \label{UU-8} \\[1mm]
&\vf_{(i,k)}^- \vf_{(j,l)}^- = \vf_{(j,l)}^- \vf_{(i,k)}^- \qquad (|(p_k+i)-(p_l-j)| \geq 2). \notag 
\end{align}

\para 
For $i=1,\cdots,n-1$, 
let $s_i =(i,i+1) \in \FS_n$ 
be the adjacent transposition. 
For 
$\mu,\nu \in \vL_{n,r}$, 
put 
$X^\nu_\mu=\{ x \in \SD_{\ol{\mu}} \,|\, \ol{\nu}(\Ft^{\ol{\mu}}\cdot x)=\ol{\nu}(\Ft^{\ol{\mu}}) \}$.  
One can check that 
\begin{align}
& X_\mu^{\mu-\a_{(i,k)}}
	=\big\{ 1, s_{N}, (s_{N} s_{N+1}), \cdots, (s_N s_{N+1} \cdots s_{N+\mu_{i+1}^{(k)}-1}) \big\}, 
\label{X-1}\\
& X_{\mu-\a_{(i,k)}}^{\mu}
	=\big\{ 1, s_{N-1}, (s_{N-1} s_{N-2}), \cdots, (s_{N-1} s_{N-2} \cdots s_{N-\mu_{i}^{(k)}+1}) \big\}, 
\label{X-2}\\
& X_\mu^{\mu+\a_{(i,k)}}
	=\big\{ 1, s_{N}, (s_{N} s_{N-1}), \cdots, (s_{N} s_{N-1} \cdots s_{N-\mu_{i}^{(k)}+1}) \big\}, 
\label{X-3}\\
& X_{\mu+\a_{(i,k)}}^{\mu}
	=\big\{ 1, s_{N+1}, (s_{N+1} s_{N+2}), \cdots, (s_{N+1} s_{N+2} \cdots s_{N+\mu_{i+1}^{(k)}-1}) \big\}, 
\label{X-4}
\end{align}
where 
$N=\sum_{\l=1}^{k-1}|\mu^{(l)}|+\sum_{j=1}^i \mu_j^{(k)}$, 
and 
put $\mu_{n+1}^{(k)}=\mu_1^{(k+1)}$ if $i=n$. 
Then, we have the following lemma.  

\begin{lem} \label{lem-vf}
For $\mu \in \vL_{n,r}$ and $(i,k) \in \vG'$, we have the followings. 

\begin{enumerate}
\item \quad \\*[-2em]
\begin{align*}
\vf_{(i,k)}^+(m_\mu)
&= q^{-\mu_{i+1}^{(k)} +1} m_{\mu+\a_{(i,k)}} \Big( \sum_{y \in X_{\mu+\a_{(i,k)}}^{\mu}} q^{\ell(y)} T_y \Big) \\
&=q^{-\mu_{i+1}^{(k)} +1}  
	\Big(\sum_{x \in X^{\mu+\a_{(i,k)}}_\mu} q^{\ell(x)} T_x^\ast \Big) h_{+{(i,k)}}^{\mu} m_{\mu} ,
\end{align*}
where 
$h_{+(i,k)}^\mu=
	\begin{cases}
		1 &(i \not=n) \\
		L_{N+1} - Q_{k+1} &(i=n) \quad (N=|\mu^{(1)}|+\cdots +|\mu^{(k)}|). 
	\end{cases}
$\\[2mm]

\item \quad \\*[-2em]
\begin{align*}
\vf_{(i,k)}^-(m_\mu)
&= q^{-\mu_i^{(k)} +1}\Big(\sum_{y \in X_{\mu}^{\mu-\a_{(i,k)}}} q^{\ell(y)} T_y^{\ast} \Big) m_\mu \\
&=q^{-\mu_i^{(k)} +1} m_{\mu-\a_{(i,k)}} h_{-{(i,k)}}^{\mu} 
	\Big(\sum_{x \in X^{\mu}_{\mu-\a_{(i,k)}}} q^{\ell(x)} T_x \Big),
\end{align*}
where 
$h_{-(i,k)}^\mu=
	\begin{cases}
		1 &(i \not=n) \\
		L_N - Q_{k+1} &(i=n) \quad (N=|\mu^{(1)}|+\cdots +|\mu^{(k)}|). 
	\end{cases}
$
\end{enumerate}
\end{lem}
\begin{proof}
One can check them from definitions by using Lemma \ref{lem-d}.
\end{proof}

This lemma implies the following proposition. 

\begin{prop} \label{prop-ef-fe}
For $(i,k),(j,l) \in \vG'$, we have the following relations. 
\begin{enumerate}
\item 
If $(i,k)\not=(j,l)$ then we have 
\begin{align*}
\vf_{(i,k)}^+ \vf_{(j,l)}^- - \vf_{(j,l)}^- \vf_{(i,k)}^+ = 0 .
\end{align*}

\item 
If $(i,k)=(j,l)$ and $i \not=n$ then we have 
\begin{align*}
\vf_{(i,k)}^+ \vf_{(i,k)}^- - \vf_{(i,k)}^- \vf_{(i,k)}^+ 
	 =\frac{\ka_{(i,k)}\ka_{(i+1,k)}^{-} - \ka_{(i,k)}^{-} \ka_{(i+1,k)}}{q-q^{-1}} .
\end{align*}

\item 
If $(i,k)=(j,l)=(n,k)$ then we have 
\begin{align*}
&\vf_{(n,k)}^+ \vf_{(n,k)}^- - \vf_{(n,k)}^- \vf_{(n,k)}^+ \\
&\qquad = 
	-\g_{k+1}\frac{\ka_{(n,k)}\ka_{(1,k+1)}^{-} - \ka_{(n,k)}^{-} \ka_{(1,k+1)}}{q-q^{-1}} 
	+\ka_{(n,k)}\ka_{(1,k+1)}^- (q^{-1} \s_{(n,k)} - q \s_{(1, k+1)}) . 
\end{align*}
\end{enumerate}
\end{prop}

\begin{proof}
By Lemma \ref{lem-vf}, for $\mu \in \vL_{n,r}$ and $(i,k), (j,l) \in \vG'$, we have 
\begin{align*}
\vf_{(i,k)}^+ &\vf_{(j,l)}^- (m_\mu) \\*
&= \vf_{(i,k)}^+ 
	\Big( q^{-\mu_j^{(l)} +1} m_{\mu-\a_{(j,l)}} h_{-{(j,l)}}^{\mu} 
		\Big(\sum_{x \in X^{\mu}_{\mu-\a_{(j,l)}}} q^{\ell(x)} T_x \Big)\Big) \\*
&=q^{-\mu_j^{(l)}+1} q^{-(\mu -\a_{(j,l)})_{i+1}^{(k)}+1} \, 
	m_{\mu} \Big( \sum_{y \in X_{(\mu-\a_{(j,l)})+\a_{(i,k)}}^{(\mu-\a_{(j,l)})}} q^{\ell(y)} T_y \Big)
	h_{-{(j,l)}}^{\mu} \Big(\sum_{x \in X^{\mu}_{\mu-\a_{(j,l)}}} q^{\ell(x)} T_x \Big). 
\end{align*}
On the other hand, 
we have 
\begin{align*}
\vf_{(j,l)}^- &\vf_{(i,k)}^+ (m_\mu) \\
&= \vf_{(i,k)}^- 
	\Big( q^{-\mu_{i+1}^{(k)} +1} m_{\mu+\a_{(i,k)}} 
		\Big(\sum_{x \in X^{\mu}_{\mu+\a_{(i,k)}}} q^{\ell(x)} T_x \Big)\Big) \\
&=q^{-\mu_{i+1}^{(k)}+1} q^{-(\mu +\a_{(i,k)})_{j}^{(l)}+1} \, 
	m_{\mu} h_{-{(j,l)}}^{\mu+\a_{(i,k)}} 
	\Big( \sum_{y \in X_{(\mu+\a_{(i,k)})- \a_{(j,l)}}^{(\mu+\a_{(i,k)})}} q^{\ell(y)} T_y \Big)
	\Big(\sum_{x \in X^{\mu}_{\mu+\a_{(i,k)}}} q^{\ell(x)} T_x \Big). 
\end{align*}
One sees that  
$q^{-\mu_j^{(l)}+1} q^{-(\mu -\a_{(j,l)})_{i+1}^{(k)}+1}
=q^{-\mu_{i+1}^{(k)}+1} q^{-(\mu +\a_{(i,k)})_{j}^{(l)}+1} $ 
for any case.
Put 
\begin{align*}
&A= \Big( \sum_{y \in X_{(\mu-\a_{(j,l)})+\a_{(i,k)}}^{(\mu-\a_{(j,l)})}} q^{\ell(y)} T_y \Big),
&&B= \Big(\sum_{x \in X^{\mu}_{\mu-\a_{(j,l)}}} q^{\ell(x)} T_x \Big), \\
&C= \Big( \sum_{y \in X_{(\mu+\a_{(i,k)})- \a_{(j,l)}}^{(\mu+\a_{(i,k)})}} q^{\ell(y)} T_y \Big), 
&&D= \Big(\sum_{x \in X^{\mu}_{\mu+\a_{(i,k)}}} q^{\ell(x)} T_x \Big).
\end{align*}

(\roi). 
First, 
we assume that $(i,k)\not= (j,l)$. 
Then we have 
$h_{-{(j,l)}}^{\mu} = h_{-{(j,l)}}^{\mu+\a_{(i,k)}} $, 
and 
$h_{-{(j,l)}}^{\mu}$ 
commute with 
$A$. 
If $(p_j+l)- (p_k+i) \not=1$ then 
we have 
$X_{(\mu-\a_{(j,l)})+\a_{(i,k)}}^{(\mu-\a_{(j,l)})} = X_{\mu+\a_{(i,k)}}^\mu$ 
and 
$X_{(\mu+\a_{(i,k)})- \a_{(j,l)}}^{(\mu+\a_{(i,k)})}= X_{\mu -\a_{(j,l)}}^{\mu}$. 
Thus, we have
$A=D$ and $B=C$. 
Moreover, 
one sees that 
$A$  
commute with 
$B$.  
If $(p_j+1)-(p_k+i)=1$ then 
we have 
$X_{(\mu-\a_{(j,l)})+\a_{(i,k)}}^{(\mu-\a_{(j,l)})} 
	=X_{(\mu+\a_{(i,k)})- \a_{(j,l)}}^{(\mu+\a_{(i,k)})}$ 
and 
$X^{\mu}_{\mu-\a_{(j,l)}}= X^{\mu}_{\mu+\a_{(i,k)}}$. 
Hence, we have  
$A=C$ and $B=D$. 
This implies (\roi). 

(\roii). 
Next, 
we assume that 
$(i,k)=(j,l)$ and $i\not=n$. 
Then we have 
$h_{-{(j,l)}}^{\mu} = h_{-{(j,l)}}^{\mu+\a_{(i,k)}}=1 $.
Put 
$N=\sum_{l=1}^{k-1}|\mu^{(l)}| + \sum_{j=1}^i \mu_j^{(k)}$.  
Then, 
by \eqref{X-4} and \eqref{X-2}, 
we have that 
\begin{align}
&X^{(\mu-\a_{(i,k)})}_{(\mu-\a_{(i,k)})+\a_{(i,k)}} 
	= \big\{ 1,s_{N}, (s_N s_{N+1}), \cdots, (s_N s_{N+1} \cdots s_{N+\mu_{i+1}^{(k)}-1}) \big\}, \label{XX-1}\\
&X^{(\mu+\a_{(i,k)})}_{(\mu+\a_{(i,k)})-\a_{(i,k)}} 
	=\big\{ 1, s_N, (s_N s_{N-1}), \cdots, (s_N s_{N-1} \cdots s_{N-\mu_i^{(k)}+1}) \big\}. \label{XX-2}
\end{align}
Combined with $(\ref{X-2})$ and $(\ref{X-4})$, 
we have 
$AB-CD = B- D$. 
Note that  
$m_\mu T_w=q^{\ell(w)} m_\mu$ 
for $w \in \FS_\mu$,  
then we have 
\begin{align*}
(\vf_{(i,k)}^+ \vf_{(i,k)}^- - \vf_{(i,k)}^- \vf_{(i,k)}^+)(m_\mu) 
&=q^{-\mu_i^{(k)} -\mu_{i+1}^{(k)}+1} 
	\left(\left(\sum_{a=0}^{\mu_i^{(k)}-1} (q^{a})^2\right) 
		- \left(\sum_{b=0}^{\mu_{i+1}^{(k)}-1} (q^{b})^2 \right)\right) m_\mu \\
&=\frac{\ka_{(i,k)}\ka_{(i+1,k)}^{-} - \ka_{(i,k)}^{-} \ka_{(i+1,k)}}{q-q^{-1}}(m_\mu).
\end{align*}
This implies (\roii). 

(\roiii). 
Finally, 
we assume that 
$(i,k)=(j,l)=(n,k)$. 
Put 
$N=\sum_{l=1}^{k}|\mu^{(k)}|$, 
then, 
we have 
$h_{-{(n,k)}}^{\mu} =L_N - Q_{k+1}$ 
and 
$h_{-{(n,k)}}^{\mu+\a_{(n,k)}} =L_{N+1} - Q_{k+1}$. 
Hence, we have 
\begin{align}
(\vf_{(n,k)}^+ \vf_{(n,k)}^- - \vf_{(n,k)}^- \vf_{(n,k)}^+) (m_\mu) 
=\, &q^{-\mu_{n}^{(k)} - \mu_{1}^{(k)} +1} m_\mu 
		(A \cdot L_{N} \cdot B - L_{N+1}\cdot  C  \cdot D) 
\label{+- -+ iii} \\*
&\quad -Q_{k+1} \,q^{-\mu_{n}^{(k)} - \mu_{1}^{(k)} +1} m_\mu (AB - CD).\notag
\end{align}
In a similar way as in the case of (\roii), 
we have 
\begin{align}
 q^{-\mu_{n}^{(k)} - \mu_{1}^{(k)} +1} m_\mu (AB - CD) 
	=\frac{\ka_{(n,k)}\ka_{(1,k+1)}^{-} - \ka_{(n,k)}^{-} \ka_{(1,k+1)}}{q-q^{-1}} (m_\mu). 
\label{+- -+ iii kouhan}
\end{align}
By Lemma \ref{lem-L}, 
we can prove 
the following formula by induction on $c$. 
\begin{align}
L_N  (T_{N-1}T_{N-2}  \cdots T_{N-c}) 
=(&q-q^{-1}) \Big( 
	\sum_{\xi=1}^{c} T_{N-1}T_{N-2} \cdots \check{T}_{N-\xi} \cdots T_{N-c} L_{N-\xi+1} \Big)  \\*
& + T_{N-1} T_{N-2}\cdots T_{N-c} L_{N-c} \notag ,
\end{align}
where 
$\check{T}_{N-\xi}$
means 
removing 
$T_{N-\xi}$ 
from the product 
$T_{N-1} T_{N-2} \cdots T_{N-c}$. 
Combined this with 
$(\ref{X-2})$, 
we have 
\begin{align}
L_N \cdot B   
&= L_N + \sum_{c=1}^{\mu_n^{(k)}-1} \big(q^c \, L_N (T_N T_{N-1} \cdots T_{N-c})\big) \label{LB}\\
&= L_N + \sum_{c=1}^{\mu_n^{(k)}-1} \Big\{ q^c(q-q^{-1})
	\Big(\sum_{\xi=1}^c T_{N-1} T_{N-2} \cdots \check{T}_{n-\xi} \cdots T_{N-c} L_{N-\xi+1}\Big) \notag \\
	& \hspace{8em} + q^c\,T_{N-1} T_{N-2} \cdots T_{N-c} L_{N-c} \Big\} \notag \\
&= L_N + \sum_{\xi=1}^{\mu_n^{(k)}-1} 
	\Big( \sum_{c=\xi}^{\mu_n^{(k)}-1} q^c(q-q^{-1}) T_{N-1} T_{N-2} 
		\cdots \check{T}_{n-\xi} \cdots T_{N-c}\Big) L_{N-\xi+1} \notag \\
	& \qquad + \sum_{c=1}^{\mu_n^{(k)}-1} q^c T_{N-1} \cdots T_{N-c}L_{N-c} \notag
\end{align}
Similarly, we have 
\begin{align}
L_{N+1} \cdot C &=L_{N+1} + \sum_{\xi=0}^{\mu_n^{(k)}-1} 
	\Big( \sum_{c=\xi}^{\mu_n^{(k)}-1} q^{c+1}(q-q^{-1}) T_{N} T_{N-1} 
		\cdots \check{T}_{n-\xi} \cdots T_{N-c}\Big) L_{N-\xi+1}  \label{LC}\\
	& \qquad + \sum_{c=0}^{\mu_n^{(k)}-1} q^{c+1} T_{N} T_{N-1} \cdots T_{N-c}L_{N-c}, \notag
\end{align}
by using the formula. 
We also have  
\begin{align}
&L_{N+1} (T_{N+1} T_{N+2} \cdots T_{N+c}) \label{LN+1}\\
&= \Bigg( (q^{-1}-q)^c 
	+\sum_{\xi=1}^c (q^{-1} -q)^{c-\xi} 
	\Big(\sum_{(i_1,\cdots,i_\xi) \text{ s.t. } \atop 1 \leq i_1 < i_2 < \cdots <i_\xi \leq c} 
	T_{N+i_1}T_{N+i_2}\cdots T_{N+i_\xi}\Big)\Bigg) \cdot L_{N+c+1}, \notag
\end{align}
which is proved by induction on $c$ thanks to Lemma \ref{lem-L}. 
$(\ref{LB})$ and $(\ref{LC})$, 
by making use of (\ref{XX-1}),(\ref{XX-2}), 
combined with Lemma \ref{lem-L} 
implies that 
\begin{align} 
&A\cdot  L_N \cdot B - L_{N+1} \cdot C \cdot D 
\label{A LN B LN+1}\\*
&= L_N \cdot B 
	- \Big(1+ q(q-q^{-1}) 
	+ \sum_{c=1}^{\mu_n^{(k)}-1} (q^{c+1}(q-q^{-1})T_{N-1}T_{N-2}\cdots T_{N-c}\Big)\cdot L_{N+1} \cdot D. \notag 
\end{align}
Note that 
$m_\mu T_w= q^{\ell(w)} m_\mu$ for $w \in \FS_\mu$, 
and so 
$(\ref{LB})$ implies that 
\begin{align} 
m_\mu\cdot (L_N \cdot B) =m_\mu\, q^{2(\mu_n^{(k)}-1)} (L_N + L_{N-1} + \cdots + L_{N-\mu_n^{(k)}+1}). 
\label{mu LN B}
\end{align}
Similarly, $(\ref{X-4})$ and $(\ref{LN+1})$ implies that 
\begin{align}
& m_\mu \cdot \Big(1+ q(q-q^{-1}) 
	+ \sum_{c=1}^{\mu_n^{(k)}-1} (q^{c+1}(q-q^{-1})T_{N-1}T_{N-2}\cdots T_{N-c}\Big)\cdot L_{N+1} \cdot D 
\label{mu LN+1 D}\\*
&= m_\mu \,q^{2(\mu_n^{(k)})}\,(L_{N+1}+L_{N+2}+\cdots + L_{N+\mu_1^{(k+1)}}). \notag
\end{align} 
By 
$(\ref{A LN B LN+1})$, 
$(\ref{mu LN B})$ 
and 
$(\ref{mu LN+1 D})$, 
we have 
\begin{align}
&q^{-\mu_{n}^{(k)} - \mu_{1}^{(k)} +1} m_\mu 
		(A \cdot L_{N} \cdot B - L_{N+1}\cdot  C  \cdot D) 
\label{+- -+ iii zenhan}\\
&= m_\mu \, q^{\mu_n^{(k)}-\mu_{1}^{(k+1)}} 
	\Big( q^{-1} \big(L_N+L_{N-1}+\cdots +L_{N-\mu_n^{(k)}+1} \big) \notag \\
&\hspace{10em} - q\big(L_{N+1}+L_{N+2} +\cdots +L_{N+\mu_{1}^{(k+1)}} \big)\Big) \notag\\
&= \ka_{n,k}\ka_{(1,k+1)}^- \big( q^{-1} \s_{(n,k)} - q \s_{(1,k+1)}\big) (m_\mu). \notag 
\end{align}
Now 
$(\ref{+- -+ iii})$, 
$(\ref{+- -+ iii kouhan})$ 
and 
$(\ref{+- -+ iii zenhan})$ 
imply 
(\roiii). 
\end{proof}

We can now prove Proposition \ref{prop-main-1}. 
\begin{proof}[(Proof of Proposition $\ref{prop-main-1}$) ]
By the relations 
$(\ref{UU-2})$, $(\ref{UU-5})$ and 
$(\ref{UU-3})-(\ref{UU-8})$ 
together with 
Proposition \ref{prop-ef-fe}, 
one sees that 
the homomorphism 
$\wt{\rho}$ 
in Proposition \ref{prop-main-1} 
is well-defined. 
On the other hand, 
by 
Proposition \ref{prop-surjective-Borel}, 
we have 
$\wt{\rho}(\wt{\CB}^+) = \Sc_{n,r}^{\geq 0}$ 
and 
$\wt{\rho}(\wt{\CB}^-) = \Sc_{n,r}^{\leq 0}$. 
Moreover, 
we know that 
$\Sc_{n,r}=\Sc_{n,r}^{\leq 0} \Sc_{n,r}^{\geq0}$ 
by Theorem \ref{thm-iso-Borel}. 
Thus, 
we see that 
$\wt{\rho}$ is surjective. 

By Theorem \ref{Uqgl-S1} (\roiii) and (\roiv)  combined with Theorem \ref{thm-iso-Borel}, 
$\wt{\rho}|_{_\CA \wt{U}_q}$ 
gives a surjection from $_\CA \wt{U}_q$ to $_\CA \Sc_{n,r}$. 
The proposition  is now proved. 
\end{proof}



\section{Presentations of cyclotomic $q$-Schur algebras} 
Recall that $\Sc_{n,r}$ is the cyclotomic $q$-Schur algebra over $\CK=\QQ(q,\g_1,\cdots,\g_r)$ 
with parameters $q,\g_1,\cdots,\g_r$. 

\para 
For presenting  cyclotomic $q$-Schur algebras by 
generators and relations, 
we prepare some notations. 
Let 
$\CK \lan x_1,\cdots,x_{m-1}\ran$ 
be the non-commutative polynomial ring over $\CK$ 
with indeterminate elements $x_1,\cdots,x_{m-1}$. 
Note that 
$\CK \lan x_1,\cdots,x_{m-1}\ran$ 
is isomorphic to the free $\CK$-algebra generated by 
$x_1,\cdots,x_{m-1}$. 
Put $\mathbf{x}=\{x_1,\cdots,x_{m-1}\}$. 
For $(i,k) \in \vG'$, 
set $x_{(i,k)}=x_{p_k+i}$, 
where $p_k=(k-1)n$.  
Thus, 
we have 
$\mathbf{x}=\{x_{(i,k)}\,|\, (i,k) \in \vG'\}$ 
and 
$\CK \lan x_1,\cdots, x_{m-1}\ran = \CK \lan \mathbf{x} \ran = \CK \lan x_{(i,k)} \,|\, (i,k) \in \vG' \ran$. 

For 
$g(\mathbf{x}) \in \CK\lan \mathbf{x} \ran$, 
let 
$g(\vf^+)$ \big(resp. $g(\vf^-)$\big)  
be the element of $\Sc_{n,r}$ 
obtained by replacing 
$x_{(i,k)}$ with $\vf^+_{(i,k)}$ \big(resp. $\vf^-_{(i,k)}$\big) 
in $g(\mathbf{x})$. 
Then, 
we have the following lemma. 

\begin{lem} \label{def g-mu-(i,k)}
For  $\la \in \vL_{n,r}$ and $(i,k) \in \vG$, 
there exists an element 
\[g^{\la}_{(i,k)}=\sum_{j} r_j\, g_j^-(\mathbf{x}) \otimes g_j^+(\mathbf{x}) 
\in \CK\lan \mathbf{x} \ran \otimes_{\CK} \CK\lan \mathbf{x} \ran 
\quad \big(r_j \in \CK,\, g_j^-(\mathbf{x}), g_j^+(\mathbf{x}) \in \CK \lan \mathbf{x} \ran \big)\] 
such that 
$\s_{(i,k)}^\la = \sum_j r_j\, g_j^-(\vf^{-})\, g_j^+(\vf^{+})\, \vf_{\la, \la}^1$. 
\end{lem} 

\begin{proof}
By Theorem \ref{thm-iso-Borel} (\roiii), 
we have $\Sc_{n,r}=\Sc_{n,r}^{\leq 0} \cdot \Sc_{n,r}^{\geq 0}$. 
On the other hand, 
By Proposition \ref{prop-surjective-Borel}, 
$\Sc_{n,r}^{\leq 0}$ (resp. $\Sc_{n,r}^{\geq 0}$) 
is generated by $\vf_{(i,k)}^{-}$ 
(resp. $\vf_{(i,k)}^{+}$) for $(i,k) \in \vG'$ 
and 
$\ka_{(i,k)}^{\pm}$ for $(i,k) \in \vG$. 
Recall that 
$\ka_{(i,k)}^{\pm}=\sum_{\mu \in \vL_{n,r}} q^{\pm \mu_i^{(k)}} \vf_{\mu,\mu}^1$, 
and that 
$\vf_{\mu,\mu}^1$ is the identity map on $M^\mu$ and the zero map on $M^\tau$ ($\tau \not=\mu$). 
Moreover, 
$\{\vf_{\mu,\mu}^1\,|\, \mu \in \vL_{n,r}\}$ 
is a set of pairwise orthogonal idempotents. 
Combined with the relation (\ref{UU-3}) and (\ref{UU-4}), 
we obtain the lemma. 
\end{proof}

\para 
In general, 
$g_{(i,k)}^\la \in \CK \lan \mathbf{x} \ran \otimes_{\CK} \CK \lan \mathbf{x} \ran $ 
satisfying the condition in Lemma \ref{def g-mu-(i,k)} 
is not unique. 
Throughout the rest of this paper, 
for $(i,k) \in \vG'$ and  $\la \in \vL_{n,r}$, 
we fix $g_{(i,k)}^\la$'s  
once and for all. 

Let 
$\CK\lan F_1,\cdots, F_{m-1}, E_1,\cdots,E_{m-1} \ran$ 
be the non-commutative polynomial ring over $\CK$ 
with indeterminate elements 
$F_1,\cdots, F_{m-1},E_1,\cdots,E_{m-1}$. 
Put 
$F=\{F_i\,|\, 1 \leq i \leq m-1\}$ 
and 
$E=\{E_i\,|\, 1 \leq i \leq m-1\}$.  
For $(i,k) \in \vG'$, 
set 
$F_{(i,k)}=F_{p_k+i}$ 
and 
$E_{(i,k)}=E_{p_k+i}$. 
For $g(\mathbf{x}) \in \CK\lan \mathbf{x} \ran$, 
let 
$g(F)$ (resp. $g(E)$) 
be the element of $\CK\lan F \ran$ (resp. $\CK \lan E \ran$) 
obtained by replacing $x_{(i,k)}$ with $F_{(i,k)}$ (resp. $E_{(i,k)}$) in $g(\mathbf{x})$. 
For $g_{(i,k)}^\la = \sum_{j} r_j\, g_j^-(\mathbf{x}) \otimes g_j^+(\mathbf{x}) 
	\in \CK \lan \mathbf{x} \ran \otimes_{\CK} \CK \lan \mathbf{x} \ran$ 
	($(i,k) \in \vG,\, \mu \in \vL_{n,r}$) 
in Lemma \ref{def g-mu-(i,k)}, 
put  
\begin{align} 
\label{def g FE}
g_{(i,k)}^\la(F,E)=\sum_{j} r_j\, g_j^-(F) \cdot g_j^+(E) \in \CK\lan F,E \ran.
\end{align}
\para 
\label{def Snr}
Let 
$\CS_{n,r}$ 
be the associative algebra over $\QQ(q,\g_1,\cdots,\g_r)$ with $1$ 
generated by 
$E_{(i,k)}, F_{(i,k)}$ $((i,k)\in \vG')$ 
and  
$1_\la$ $(\la \in \vL_{n,r})$ 
with the following defining relations: 
\begin{align}
&1_\la 1_\mu = \d_{\la,\mu} 1_\la, \quad \sum_{\la \in \vL_{n,r}} 1_\la =1, \label{S-1}\\
&E_{(i,k)} 1_\la = 
	\begin{cases} 
		1_{\la + \a_{(i,k)}} E_{(i,k)} & \text{ if }\la + \a_{(i,k)} \in \vL_{n,r}, \label{S-2}\\
		0 & \text{otherwise},
	\end{cases}	\\
&F_{(i,k)} 1_\la = 
	\begin{cases} 
		1_{\la - \a_{(i,k)}} F_{(i,k)} & \text{ if }\la - \a_{(i,k)} \in \vL_{n,r}, \label{S-3}\\
		0 & \text{otherwise},
	\end{cases}	\\
&1_{\la} E_{(i,k)}  = 
	\begin{cases} 
		E_{(i,k)} 1_{\la - \a_{(i,k)}} & \text{ if }\la - \a_{(i,k)} \in \vL_{n,r}, \label{S-4}\\
		0 & \text{otherwise},
	\end{cases}	\\
&1_\la F_{(i,k)} = 
	\begin{cases} 
		F_{(i,k)} 1_{\la + \a_{(i,k)}} & \text{ if }\la + \a_{(i,k)} \in \vL_{n,r}, \label{S-5}\\
		0 & \text{otherwise},
	\end{cases}	\\
&E_{(i,k)}F_{(j,l)} - F_{(j,l)}E_{(i,k)} \label{S-6}
		=\d_{(i,k),(j,l)} \sum_{\la \in \vL_{n,r}} \eta_{(i,k)}^\la, \\
&E_{(i \pm 1,k)}(E_{(i,k)})^2 - (q+q^{-1}) E_{(i,k)} E_{(i \pm 1,k)} E_{(i,k)} 
	+ (E_{(i,k)})^2 E_{(i \pm 1,k)}=0 , \label{S-7}\\
& E_{(i,k)} E_{(j,l)}= E_{(j,l)} E_{(i,k)} \qquad (|(p_k+i)-(p_l+j)| \geq 2), \notag \\
&F_{(i \pm 1,k)}(F_{(i,k)})^2 - (q+q^{-1}) F_{(i,k)} F_{(i \pm 1,k)} F_{(i,k)} 
	+ (F_{(i,k)})^2 F_{(i \pm 1,k)}=0,  \label{S-8}\\
& F_{(i,k)} F_{(j,l)}= F_{(j,l)} F_{(i,k)} \qquad (|(p_k+i)-(p_l+j)| \geq 2), \notag 
\end{align}
where 
\[\eta_{(i,k)}^\la= \begin{cases} 
\Big( - \g_{k+1} [ \la_{n}^{(k)} - \la_{1}^{(k+1)}]  \\
			\hspace{2em} +q^{\la_n^{(k)}-\la_1^{(k+1)}} 
			\big( q^{-1}\, g_{(n,k)}^\la(F,E) -q \, g_{(1,k+1)}^\la(F,E) \big)\Big) 1_{\la}
			&\text{ if } i=n, \\[2mm]
[\la_{i}^{(k)} - \la_{i+1}^{(k)}] 1_\la & \text{otherwise}.			
\end{cases}
\]

\para 
It is clear that 
$\CS_{n,r}$ 
is a homomorphic image of 
$\wt{\CS}_q (\vL_{n,r})$ defined in Section \ref{Sq}. 
Thus, 
$\CS_{n,r}$ is a homomorphic image of $\wt{U}_q$. 
In fact, 
as the following lemma shows, 
$\CS_{n,r}$ is isomorphic to 
$\CS_q^{\eta_{\vL_{n,r}}}$, where 
$\eta_{\vL_{n,r}}=\{\eta_{(i,k)}^\la \,|\, (i,k) \in \vG', \, \la \in \vL_{n,r}\}$. 

\begin{lem}
For $(i,k) \in \vG'$ and $\la \in \vL_{n,r}$, 
we have 
$\eta_{(i,k)}^\la \in \wt{\CS}_q^- \wt{\CS}_q^+ 1_\la$ 
and  
$\deg (\eta_{(i,k)}^\la)=0$.
Thus, 
$\CS_{n,r}$ is isomorphic to $\CS_q^{\eta_{\vL_{n,r}}}$. 
\end{lem}

\begin{proof} 
From the definitions of 
$g_{(n,k)}^\la(F,E)$ and $g_{(1,k+1)}^\la(F,E)$, 
it is clear that 
$\eta_{(i,k)}^\la \in \wt{\CS}_q^- \wt{\CS}_q^+ 1_\la$. 
Note that 
$\s_{(i,k)}^\la \in \Hom_{\He_{n,r}}(M^\la, M^\la)$,  
Lemma \ref{def g-mu-(i,k)} 
together with 
the definitions of $
\vf_{(j,l)}^{\pm}$ 
imply that 
$\deg (g_{(i,k)}^\la(F,E))=0$. 
Thus, we have 
$\deg (\eta_{(i,k)}^\la)=0$.
\end{proof}

From now on, 
under the isomorphism 
$\CS_{n,r} \cong \CS_q^{\eta_{\vL_{n,r}}}$, 
we apply to $\CS_{n,r}$ the results in Section \ref{Sq} and \ref{specialization} 
for $\CS_q^{\eta_{\vL_{n,r}}}$. 
Recall that 
$\wt{\rho} : \wt{U}_q \ra \Sc_{n,r}$ 
and 
$\Psi : \wt{U}_q \ra \CS_{n,r}$ 
are surjective homomorphisms of algebras 
given in Proposition  \ref{prop-main-1} and the paragraph \ref{def Sq} respectively. 
We have the following proposition.

\begin{prop}
\label{prop surjection Sq Snr}
There exists a surjective homomorphism of algebras 
$\Phi : \CS_{n,r} \ra \Sc_{n,r}$ 
such that 
\begin{align}\label{def Phi}
&\Phi(E_{(i,k)})=\vf_{(i,k)}^+,\,\, 
\Phi(F_{(i,k)})=\vf_{(i,k)}^-, \,\,
\Phi(1_\la)= \vf_{\la,\la}^1.  
\end{align}
In particular,  
the surjection 
$\wt{\rho} : \wt{U}_q \ra \Sc_{n,r}$ 
factors through the algebra $\CS_{n,r}$, 
namely we have 
$\wt{\rho}=\Phi \circ \Psi$.  
Moreover, 
by restricting $\Phi$ to $_\CA \CS_{n,r}$, 
we have a surjective homomorphism 
$\Phi |_{_\CA \CS_{n,r}} : \,_\CA \CS_{n,r} \ra \,_\CA \Sc_{n,r}$. 
\end{prop}

\begin{proof}
First, 
we prove that 
$\Phi$  
gives a well-defined algebra homomorphism 
from 
$\CS_{n,r}$ to $\Sc_{n,r}$. 
One can easily check that 
the relations 
$(\ref{S-1})-(\ref{S-5})$ hold 
in the images of $\Phi$ for corresponding generators. 
By $(\ref{UU-7})$ and $(\ref{UU-8})$, 
the relations 
$(\ref{S-7})$ and $(\ref{S-8})$ 
hold in the image of  $\Phi$. 
Proposition \ref{prop-ef-fe} 
together with the definition of $\eta_{(i,k)}^\la$ 
implies that 
$(\ref{S-6})$ 
holds in the image of  $\Phi$. 
Thus, 
$\Phi$ is well-defined.  
By investigating the images of generators under each map, 
we have 
$\wt{\rho} = \Phi \circ \Psi$, 
and  
$\Phi$ is  surjective. 
The last assertion follows from 
the restriction of $\wt{\rho}=\Phi \circ \Psi$ to $_\CA \wt{U}_q$ 
together with  
Proposition \ref{prop-main-1}. 
\end{proof}

Since $\vf_{\la,\la}^1 \not=0 $ in $\Sc_{n,r}$ for $\la \in \vL_{n,r}$, 
and since 
$\Phi$ 
is surjective, 
we have  the following corollary. 

\begin{cor}
For $\la \in \vL_{n,r}$, 
$1_\la \not=0$ in $\CS_{n,r}$. 
\end{cor}

\para 
For $\la =(\la^{(1)}, \cdots, \la^{(r)}) \in \vL_{n,r}$, 
we say that 
$\la$ is an $r$-partition of size $n$ 
if all $\la^{(k)}$ ($1 \leq k \leq r$) 
are partitions, namely all $\la^{(k)}$ are weakly decreasing sequences.  
On the other hand, 
we have 
$\vL_{n,r}^+=\big\{ \la \in \vL_{n,r} \,|\, 1_\la \not \in \CS_{n,r}(>\la) \big\}$ 
by \eqref{vL+}.
Then, 
we obtain the parametrization of the isomorphism classes of 
simple $\CS_{n,r}$-modules as follows. 

\begin{lem}
\label{lem vLnr r-partition}
For $\CS_{n,r} (\cong \CS_q^{\eta_{\vL_{n,r}}})$, we have 
\[\vL_{n,r}^+
= \big\{ \la \in \vL_{n,r} \,|\, \la \,: r \text{-partition} \big\}.
\]
In particular, 
the isomorphism classes of simple $\CS_{n,r}$-modules 
are parametrized by $\vL_{n,r}^+$.  
\end{lem}

\begin{proof}
Let $(i,k) \in \vG'$ be such that $i \not= n$.  
For 
$a \in \ZZ_{>0}$ and $ \la \in \vL_{n,r}$,  
we can prove, 
by induction on $a \in \ZZ_{>0} $ together with (\ref{S-6}), that 
\begin{align}
\label{formula EF FE}
E_{(i,k)}^a F_{(i,k)}^a 1_\la 
	\equiv [a]! \Big(\prod_{j=1}^a [\la_{i}^{(k)} - \la_{i+1}^{(k)} -a +j] \Big)1_\la 
	\mod \CS_{n,r}(> \la).
\end{align}

Assume that $\la \in \vL_{n,r}$ is not an $r$-partition. 
Then, 
there exists $i,k $ 
such that 
$\la_i^{(k)}<\la_{i+1}^{(k)}$, where 
$1 \leq i \leq n-1$ and $1 \leq k \leq r$. 
Thus, by (\ref{formula EF FE}), 
we have 
\begin{align}
\label{formula EF FE 2}
E_{(i,k)}^{\la_i^{(k)}+1} F_{(i,k)}^{\la_i^{(k)}+1} 1_\la 
\equiv 
[\la_i^{(k)}+1]! \Big(\prod_{j=1}^{\la_i^{(k)}+1} [j - \la_{i+1}^{(k)} - 1] \Big)1_\la 
	\mod \CS_{n,r}(> \la).
\end{align}
Since 
$\la - (\la_i^{(k)}+1) \a_{(i,k)} \not \in \vL_{n,r}$, 
the left-hand side of (\ref{formula EF FE 2}) is equal to $0$ by (\ref{S-3}).
On the other hand,  
since $\la_i^{(k)}<\la_{i+1}^{(k)}$, 
we have 
$[\la_i^{(k)}+1]! \Big(\prod_{j=1}^{\la_i^{(k)}+1} [j - \la_{i+1}^{(k)} - 1] \Big) \not=0$. 
Thus, 
(\ref{formula EF FE 2}) 
implies that 
$1_\la \in \CS_{n,r}(>\la)$ 
if 
$\la$ is not an $r$-partition. 
By Theorem \ref{thm standard simple Sq} (\roiii), 
the isomorphism classes of simple $\CS_{n,r}$-modules 
are parametrized  
by the set  $\{\la \in \vL_{n,r} \,|\, 1_\la \not\in \CS_{n,r}(>\la)\}$. 
On the other hand, 
through the surjection $\Phi : \CS_{n,r} \ra \Sc_{n,r}$ in Proposition \ref{prop surjection Sq Snr}, 
one can regard a simple $\Sc_{n,r}$-module as a simple $\CS_{n,r}$-module. 
Moreover, 
it is known that 
the isomorphism classes of simple $\Sc_{n,r}$-modules 
are parametrized  
by the set of $r$-partitions of size $n$ 
by \cite{DJM98}. 
Thus, we obtain the lemma. 

\end{proof}

\para  \label{def Unr}
Since $\CS_{n,r}$ is a quotient algebra of $\wt{U}_q$, 
one can describe $\CS_{n,r}$ 
by generators and relations of $\wt{U}_q$ 
together with some additional relations. 
Here, 
we give such additional relations precisely. 
For 
$(i,k) \in \vG'$ and $\la \in \vL_{n,r}$, 
we define 
$g_{(i,k)}^\la(f,e) \in \wt{U}_q$ 
in a similar way 
as in (\ref{def g FE}). 
Recall 
the bijection 
from $\vL_{n,r}$ to $\vL_{n,1}$ 
such that 
$\mu \mapsto \ol{\mu}$ 
in \ref{vLnr to vLn1}. 
For $\la \in \vL_{n,r}$, 
put 
$K_\la =K_{\ol{\la}} \in \wt{U}_q$, 
where $K_{\ol{\la}}$ is defined in (\ref{def Kla}).  
For $(i,k) \in \vG'$, 
put 
\[g_{(i,k)}(f,e)=\sum_{\la \in \vL_{n,r}} \Big( g_{(i,k)}^\la(f,e) K_{\la} \Big),\]
and put 
\[\eta_{(i,k)}= \begin{cases} \displaystyle 
\Big( - \g_{k+1} \frac{K_{(n,k)}K^-_{(1,k+1)} - K_{(n,k)}^- K_{(1,k+1)}}{q-q^{-1}} \\ 
			\hspace{2em} +K_{(n,k)}K^{-1}_{(1,k+1)} 
			\big( q^{-1}\, g_{(n,k)}(f,e) -q \, g_{(1,k+1)}(f,e) \big)\Big)
			&\text{ if } i=n, \\[3mm]
\displaystyle  
\frac{K_{(i,k)}K^-_{(i+1,k)} - K_{(i,k)}^- K_{(i+1,k)}}{q-q^{-1}} & \text{otherwise}.			
\end{cases}
\]

Let 
$\wt{I}_{n,r}$ 
be the two-sided ideal of $\wt{U}_q$ 
generated by 
$\t_{p_k+i} - \eta_{(i,k)}$ ($(i,k) \in \vG'$), 
$K_1K_2 \cdots K_m - q^n$ 
and 
$(K_i-1)(K_i-q)(K_i-q^2)\cdots (K_i-q^n)$ ($1 \leq i \leq m$).  
Let 
$U_{n,r}=\wt{U}_q / \wt{I}_{n,r}$ 
be 
a quotient algebra of $\wt{U}_q$. 
One sees that 
$U_{n,r}$ 
is isomorphic to the algebra 
generated by 
$E_i, F_i$ ($1 \leq i \leq m-1$) 
and $K^{\pm}_i$ ($1 \leq i \leq m$) 
with 
defining relations 
(\ref{U-2})-(\ref{U-4}), (\ref{U-7}) and (\ref{U-8}) 
together with the following relations;  
\begin{align}
&e_{(i,k)} f_{(j,l)} - f_{(j,l)} e_{(i,k)}= \d_{(i,k),(j,l)} \eta_{(i,k)}, \label{U-a-1}\\
&K_1K_2 \cdots K_m = q^n, \label{U-a-2}\\
&(K_i-1)(K_i-q)(K_i-q^2)\cdots (K_i-q^n)=0, \label{U-a-3}
\end{align}
where we identify 
$e_{(i,k)} \leftrightarrow e_{p_k+i}$, 
$f_{(i,k)} \leftrightarrow f_{p_k+i}$ 
and 
$K^{\pm}_{(i,k)} \leftrightarrow K^{\pm}_{p_k+i}$.

\begin{prop}
\label{prop iso Unr Snr}
$\wt{I}_{n,r}$ contains the kernel of the surjection 
$\Psi : \wt{U}_q \ra \CS_{n,r}$. 
Thus, 
$\Psi$ induces the surjection 
$\Psi' : U_{n,r} \ra \CS_{n,r}$. 
Moreover, 
$\Psi'$ 
gives an isomorphism of algebras. 
\end{prop}

\begin{proof}
From the definition, 
we have 
$\Psi(\eta_{(i,k)})=\sum_{\la \in \vL_{n,r}} \eta_{(i,k)}^\la$, 
thus we have 
$\Psi(\t_{p_k+i} - \eta_{(i,k)})=0$.   
Note that 
$\Psi(K_i)=\sum_{\la \in \vL_{n,r}} q^{\ol{\la}_i} 1_\la$, 
we see easily that 
$\Psi(K_1 \cdots K_m)=q^n$ 
and 
$\Psi\big((K_i-1)(K_i-q)\cdots (K_i-q^n)\big)=0$. 
Thus, we have $\wt{I}_{n,r} \subset \Ker \Psi$, 
and 
$\Psi$ induces 
the surjection 
$\Psi' : U_{n,r} \ra \CS_{n,r}$.  

Let $U_{n,r}^0$ be the subalgebra of $U_{n,r}$ 
generated by $K_i$ ($1 \leq i \leq m$). 
In a similar way as the proof of \cite[Lemma 13.39]{DDPW},  
the restriction of $\Psi'$ to $U_{n,r}^0$ 
gives an isomorphism 
$U_{n,r}^0 \cong \CS_{n,r}^0$ 
(Note that, in the proof of \cite[Lemma 13.39]{DDPW}, 
they only use the relations of $K_i$'s which coincide with the relations in $U_{n,r}^0$). 
Through the isomorphism $U_{n,r}^0 \cong \CS_{n,r}^0$, 
we have 
\begin{align} 
\label{Unr-1}
K_\la K_\mu =\d_{\la,\mu}K_\la, 
\quad 
\sum_{\la \in \vL_{n,r}} K_\la =1
\end{align} 
in $U_{n,r}$.  
Moreover, 
for $1 \leq i \leq m$ and $\la \in \vL_{n,r}$, 
we have 
$K_i K_{\la}=q^{\ol{\la}_i} K_\la$, 
thus 
we have 
\begin{align} 
\label{gen Ki} 
K_i=K_i\big( \sum_{\la \in \vL_{n,r}} K_\la \big)=\sum_{\la \in \vL_{n,r}}q^{\ol{\la}_i} K_\la. 
\end{align} 
Let 
$\Psi^\dagger : \CS_{n,r} \ra U_{n,r}$ 
be a homomorphism of algebras 
given by 
$\Psi^\dagger (E_{(i,k)})=e_{(i,k)}$, 
$\Psi^\dagger (F_{(i,k)})=f_{(i,k)}$ 
and 
$\Psi^\dagger (1_\la)=K_\la$. 
In order to see that 
$\Psi^\dagger$ 
is well-defined, 
we may check the relations 
(\ref{S-1})-(\ref{S-8}) 
in the image of $\Psi^\dagger$ 
for corresponding generators. 
The relation (\ref{S-1}) follows from (\ref{Unr-1}). 
We can check the relations (\ref{S-2})-(\ref{S-5}) 
in a similar way as in the proof of \cite[Lemma 13.40]{DDPW}. 
The relation (\ref{S-6}) 
follows from the definition of $\eta_{(i,k)}$. 
The relation (\ref{S-7}) and (\ref{S-8}) 
are just (\ref{U-7}) and (\ref{U-8}) respectively. 
Thus, 
$\Psi^\dagger$ is well-defined. 
Moreover, 
by (\ref{gen Ki}), 
we see that 
$\Psi^\dagger$ is surjective  
and gives the inverse map of $\Phi'$, 
thus we have 
$U_{n,r} \cong \CS_{n,r}$. 
\end{proof}

\para 
Our goal is to show 
that the surjection $\Phi : \CS_{n,r} \ra \Sc_{n,r}$ in Proposition \ref{prop surjection Sq Snr} 
is actually an isomorphism. 
Let 
\[ \big\{ \vf_{ST} \bigm| S,T \in \CT(\la) \text{ for some }\la \in \vL_{n,r}^+ \big\} \] 
be a cellular basis of $\Sc_{n,r}$ constructed in \cite{DJM98}, 
where $\CT(\la)$ is the set of semi-standard tableaux of shape $\la$ (see \cite{DJM98} for the definition).     
For $\la \in \vL_{n,r}^+$, 
let 
$\Sc_{n,r}(\geq \la)$ (resp. $\Sc_{n,r}(>\la)$) 
be a subspace of $\Sc_{n,r}$ 
spanned by 
$\{ \vf_{ST} \,|\, S,T \in \CT(\mu) \text{ for some } \mu \in \vL_{n,r}^+ \text{ such that } \mu \geq \la\}$  
(resp. $\{ \vf_{ST} \,|\, S,T \in \CT(\mu) \text{ for some } \mu \in \vL_{n,r}^+ \text{ such that } \mu > \la\}$), 
then 
both of $\Sc_{n,r}(\geq \la)$ and $\Sc_{n,r}(>\la)$ 
are two-sided ideals of $\Sc_{n,r}$. 

It is known that 
$\vf_{\la,\la}^1 \in \Sc_{n,r}(\geq \la) \setminus \Sc_{n,r}(>\la)$ for $\la \in \vL_{n,r}^+$ 
($\vf_{\la,\la}^1$ is denoted by $\vf_{T^\la T^\la}$ in \cite{DJM98}). 
For $\la \in \vL_{n,r}^+$, 
a left $\Sc_{n,r}$-module $W(\la)$ of $\Sc_{n,r}$ (so called Weyl module) is defined by 
\[ W(\la) = \big(\Sc_{n,r} \cdot \vf_{\la,\la}^1 + \Sc_{n,r}(> \la) \big)\big/  \Sc_{n,r}(>\la). \]
Note that 
$W(\la)$ is an $\Sc_{n,r}$-submodule of $\Sc_{n,r}(\geq \la)/ \Sc_{n,r}(>\la)$. 
By \cite[Theorem 5.15]{DR} (and its proof), 
for $S,T \in \CT(\mu)$, 
we have 
\begin{align}
\label{cellula basis decom}
\vf_{ST}=\vf_{ST^\mu} \vf_{\mu,\mu}^1 \vf_{T^\mu T}, \text{ where }
\vf_{ST^\mu} \in \Sc_{n,r}^{\leq 0}  
\text{ and } 
\vf_{T^\mu T} \in \Sc_{n,r}^{\geq 0}.
\end{align}
One sees from this that 
\[ W(\la) \cong \Sc^{\leq 0} \cdot \vf_{\la,\la}^1 
	\big/ \big( \Sc^{\leq 0} \cdot \vf_{\la,\la}^1 \cap \Sc_{n,r}(>\la) \big) \text{ as } \CK\text{-vector spaces}.\]
It is known that 
$\{W(\la) \, | \, \la \in \vL_{n,r}^+\}$ 
gives a complete set of isomorphism classes of (left) simple $\Sc_{n,r}$-modules. 
Similarly, 
we have a complete set of isomorphism classes of (right) simple $\Sc_{n,r}$-modules 
$\{W^\sharp(\la) \,|\, \la \in \vL_{n,r}^+\}$ 
such that 
\[ W^\sharp(\la) = \vf_{\la,\la}^1 \cdot \Sc^{\geq 0} 
	\big/ \big(\vf_{\la,\la}^1 \cdot \Sc_{n,r}^{\geq 0} \cap \Sc_{n,r}(>\la) \big) 
	\text{ as } \CK \text{-vector spaces}.
\]

Recall that 
$\CS_{n,r}^{\leq 0}$ (resp. $\CS_{n,r}^{\geq 0}$)
is an subalgebra of $\CS_{n,r}$ 
defined in \ref{def Borel Sq +-}. 
Then we have the following lemma. 
\begin{lem}
\label{lem restriciton Phi to Borel}
The restriction of the surjection $\Phi$ (in Proposition \ref {prop surjection Sq Snr}) 
to $\CS_{n,r}^{\leq 0}$ (resp. $\CS_{n,r}^{\geq 0}$) 
gives an isomorphism 
$\Phi|_{\CS_{n,r}^{\leq 0}} : \CS_{n,r}^{\leq 0} \ra \Sc_{n,r}^{\leq 0}$ 
(resp. $\Phi|_{\CS_{n,r}^{\geq 0}} : \CS_{n,r}^{\geq 0} \ra \Sc_{n,r}^{\geq 0}$) 
as algebras. 
\end{lem}

\begin{proof}
By Proposition \ref{prop-surjective-Borel}, 
the restriction of $\wt{\rho}$ (in Proposition \ref{prop-main-1}) to $\wt{\CB}^-$ 
gives a surjective homomorphism  $\wt{\rho}|_{\wt{\CB}^-} : \wt{\CB}^- \ra \Sc_{n,r}^{\leq 0}$.  
Since $\Phi \circ \Psi= \wt{\rho}$ (see Proposition \ref {prop surjection Sq Snr}) 
and 
$\Psi (\wt{\CB}^-)= \Sc_{n,r}^{\leq 0}$,  
we have a surjective homomorphism 
$\Phi|_{\CS_{n,r}^{\leq 0}} : \CS_{n,r}^{\leq 0} \ra \Sc_{n,r}^{\leq 0}$. 

On the other hand, 
thanks to Theorem \ref{thm gen rel of Borel}, 
we can define the homomorphism $\Phi'^{\leq 0}$ of algebras from 
$\Sc_{n,1}^{\leq 0}$ to $U_{n,r}$ 
by sending the elements $f_i$ ($1 \leq i \leq m-1$) and $K_i^{\pm}$ ($1\leq i \leq m$) of $\Sc_{n,1}^{\leq 0}$ 
to the corresponding elements of $U_{n,r}$. 
Combining with isomorphisms $\Sc_{n,1}^{\leq 0} \cong \Sc_{n,r}^{\leq 0}$ and $U_{n,r} \cong \CS_{n,r}$, 
$\Phi'^{\leq 0}$ induces a surjective homomorphism from $\Sc_{n,r}^{\leq 0}$ to $\CS_{n,r}^{\leq 0}$. 
Thus, 
$\Phi|_{\CS_{n,r}^{\leq 0}}$ is an isomorphism. 
The case of $\CS_{n,r}^{\geq 0}$ is similar. 
\end{proof}

\begin{lem}
\label{surj Snr > la Csnr > la}
For $\la \in \vL_{n,r}^+$, 
the restriction of $\Phi$ to $\CS_{n,r}(\geq \la)$ (resp. $\CS_{n,r}(>\la)$) 
gives a surjective homomorphism of $(\CS_{n,r},\CS_{n,r})$-bimodules 
$\Phi|_{\CS_{n,r}(\geq \la)} : \CS_{n,r}(\geq \la) \ra \Sc_{n,r}(\geq \la)$ 
(resp. $\Phi|_{\CS_{n,r}(> \la)} : \CS_{n,r}(> \la) \ra \Sc_{n,r}(> \la)$).

\end{lem}

\begin{proof}
Note that $\Phi(1_\mu) = \vf_{\mu,\mu}^1$, 
and that $\vf_{\mu,\mu}^1 \in \Sc_{n,r}(\geq \la)$ if $\mu \geq \la$, 
we have 
$\Phi(\CS_{n,r}(\geq \la)) \subset \Sc_{n,r}(\geq \la)$ 
since $\Sc_{n,r}(\geq \la)$ is a two-sided ideal of $\Sc_{n,r}$. 

On the other hand, 
one sees easily that 
\[
\CS_{n,r}(\geq \la) = \sum_{\mu \in \vL_{n,r}^+ \atop \mu \geq \la} \CS_{n,r}^{\leq 0}\, 1_\mu \,\CS_{n,r}^{\geq 0}. 
\]
Combining with 
\eqref{cellula basis decom} and Lemma \ref{lem restriciton Phi to Borel},  
we have 
$\vf_{ST} \in \Phi(\CS_{n,r}(\geq \la)) $ 
for any $S, T \in \CT(\mu)$ ($\mu \in \vL_{n,r}^+$ such that $\mu \geq \la$). 
Thus, 
$\Phi|_{\CS_{n,r}(\geq \la)}$ is a surjection from $\CS_{n,r}(\geq \la)$ to $\Sc_{n,r}(\geq \la)$.
The case of $\CS_{n,r}(>\la)$ is similar. 
\end{proof}

The following theorem is our main result in this paper. 

\begin{thm}\
\label{main-thm} 

\begin{enumerate}
\item 
$\Phi : \CS_{n,r} \ra \Sc_{n,r}$ 
gives an isomorphism of algebras. 
Moreover, 
by restricting $\Phi$ to $_\CA \CS_{n,r}$, 
$\Phi|_{_\CA \CS_{n,r}} $ 
gives an isomorphism from $_\CA \CS_{n,r}$ to $_\CA \Sc_{n,r}$. 

\item
$\Sc_{n,r}$ is presented by  
generators $E_{(i,k)}, F_{(i,k)}$ $\big((i,k) \in \vG' \big)$ and 
$1_\la$ $(\la \in \vL_{n,r})$ with the defining relations 
\eqref{S-1}-\eqref{S-8}. 

\item 
$\Sc_{n,r}$ is also presented by 
generators $E_i, F_i$ $( 1 \leq i \leq m-1)$ and 
$K_i^{\pm}$ $(1 \leq i \leq m)$ with the defining relations 
\eqref{U-2}-\eqref{U-4}, \eqref{U-7}, \eqref{U-8} and 
\eqref{U-a-1}-\eqref{U-a-3}. 
\end{enumerate}
\end{thm}

\begin{proof}
Through the surjection 
$\Phi : \CS_{n,r} \ra \Sc_{n,r}$, 
we can regard a simple $\Sc_{n,r}$-module 
$W(\la)$ ($\la \in \vL_{n,r}^+$) 
as a simple $\CS_{n,r}$-module, 
and $\{W(\la)\,|\, \la \in \vL_{n,r}^+\}$ 
gives a complete set of isomorphism classes of simple $\CS_{n,r}$-modules 
by Lemma \ref{lem vLnr r-partition}. 
As $\wt{U}_q$-modules, 
both of $\D(\la)$ and $W(\la)$ are highest weight modules with a highest weight $\la$. 
Thus, by investigating the action on highest weight vectors of $\D(\la)$ and $W(\la)$, 
we have a surjective homomorphism 
\begin{align} 
\label{surj Dla to Wla}
\D(\la) \ra W(\la) \,\, \text{ as } \CS_{n,r} \text{-modules}. 
\end{align}
We claim the followings. 
\\[-3mm]

\begin{description}
\item[(claim)] 
For any $\la \in \vL_{n,r}^+$, we have
\\[-1.5em]
\begin{align*}
&\D(\la) \cong W(\la) \text{ as left } \CS_{n,r}\text{-modules},  
	\qquad 
\D^\sharp (\la) \cong W^\sharp (\la) \text{ as right } \CS_{n,r}\text{-modules},\\
&\D(\la) \otimes_{\CK} \D^\sharp(\la) \cong \CS_{n,r}(\geq \la) / \CS_{n,r}(> \la) 
	\text{ as } (\CS_{n,r},\CS_{n,r}) \text{-bimodules}. 
\end{align*}
\end{description}

If we assume the claim, 
then we have 
\begin{align*}
\dim_{\CK} \CS_{n,r} 
&=
\sum_{\la \in \vL_{n,r}^+} (\dim_{\CK} \D(\la) )^2 
\\
&=
\sum_{\la \in \vL_{n,r}^+} (\dim_{\CK} W(\la))^2 
\\
&=
\dim_{\CK} \Sc_{n,r}.
\end{align*}
This implies that $\Phi$ gives an isomorphism from $\CS_{n,r}$ to $\Sc_{n,r}$. 
Thus, it is enough to show the claim. 

We recall that 
\begin{align}
&\D(\la) \cong  \CS_{n,r}^{\leq 0} \cdot 1_\la 
	\big/ \big( \CS_{n,r}^{\leq 0} \cdot 1_\la \cap \CS_{n,r}(>\la ),  
\label{Dla as vector space}
\\
&W(\la) \cong \Sc^{\leq 0} \cdot \vf_{\la,\la}^1 
	\big/ \big( \Sc^{\leq 0} \cdot \vf_{\la,\la}^1 \cap \Sc_{n,r}(>\la ) \big)
\label{Wla as vector space}
\end{align}
as $\CK$-vector spaces. 
Lemma \ref{lem restriciton Phi to Borel} implies the following isomorphism ; 
\begin{align} 
\label{rest Phi Snr 1la}
\Phi|_{\CS_{n,r}^{\leq 0}\, 1_\la} : \CS_{n,r}^{\leq 0}\, 1_\la \cong \Sc_{n,r}^{\leq 0}\, \vf_{\la,\la}^1 \,\,
\text{ as } \CK \text{-vector spaces}. 
\end{align}
We prove the claim by backword induction on 
the partial order of $\vL_{n,r}^+$. 

First, we suppose  that $\la$ is maximal in $\vL_{n,r}^+$. 
In this case, we have 
$\CS_{n,r}(>\la)=\{0\}$ and $\Sc_{n,r}(>\la)=\{0\}$. 
Thus, 
\eqref{surj Dla to Wla}, \eqref{Dla as vector space}, \eqref{Wla as vector space} 
and  \eqref{rest Phi Snr 1la} 
implies that 
$\D(\la) \cong W(\la)$ as left $\CS_{n,r}$-modules. 
Similarly, 
we have 
$\D^\sharp(\la) \cong W^\sharp (\la)$ as right $\CS_{n,r}$-modules. 
Since $\D(\la)$ (resp. $\D^\sharp(\la)$ ) is a simple left (resp. right) $\CS_{n,r}$-module, 
the surjective homomorphism of $\CS_{n,r}$-bimodules  
$\D(\la) \otimes_{\CK} \D^\sharp(\la) \ra \CS_{n,r}(\geq \la) / \CS_{n,r}(>\la)$ 
is an isomorphism.  

Next, we suppose that $\la$ is not maximal in $\vL_{n,r}^+$. 
The induction hypothesis 
implies that the surjection 
$\Phi|_{\CS_{n,r}(> \la)} : \CS_{n,r}(>\la) \ra \Sc_{n,r}(>\la)$ 
in Lemma \ref{surj Snr > la Csnr > la} 
is an isomorphism 
by comparing dimensions. 
Combined with \eqref{surj Dla to Wla}, \eqref{Dla as vector space}, \eqref{Wla as vector space} 
and  \eqref{rest Phi Snr 1la}, 
this implies that 
$\D(\la) \cong W(\la)$ as left $\CS_{n,r}$-modules. 
Similarly, 
we have 
$\D^\sharp(\la) \cong W^\sharp (\la)$ as right $\CS_{n,r}$-modules. 
This implies that 
$\D(\la) \otimes_{\CK} \D^\sharp(\la) \cong \CS_{n,r}(\geq \la) / \CS_{n,r}(>\la)$. 
Thus, we have the claim and (\roi) follows.  
The remaining assertions (\roii) and (\roiii) 
follows from \ref{def Snr} and Proposition \ref{prop iso Unr Snr}. 
\end{proof}

\remarks\ 

(\roi)
In the case where $r=1$, 
generators and defining relations of $\CS_{n,r}$ (resp. $U_{n,r}$)  
in \ref{def Snr} (resp. \ref{def Unr}) 
coincide with 
generators and defining relations of 
$q$-Schur algebras of type A in Theorem \ref{thm-DG-2} (resp. Theorem \ref{S1-presen})
given by Doty and Giaquinto. 
 
(\roii) 
In a similar reason as in the case where $r=1$ (see Remark \ref{remark condition r=1}), 
$\CS_{n,r}$ ($\cong \Sc_{n,r}$) 
satisfies the conditions (A-1), (A-2) and (C-1).

 


\section{An algorithm for computing  decomposition numbers}
\label{section algorithm}

In this section, 
we give an algorithm for computing the decomposition numbers of $ _F \CS_{n,r} \cong \, _F \Sc_{n,r}$ 
on an arbitrary field $F$ 
and parameters $q,Q_1,\cdots,Q_r \in F$.  
Throughout this section, 
we consider the objects over a fixed field $F$,  
and so  
we will omit the subscript $F$ (e.g. $_F \CS_{n,r}$, $_F \D(\la), \cdots$) 
unless it causes some confusions.

\para 
Since 
$\CS_{n,r}$ 
satisfies the condition (C-1), 
we can define a bilinear form 
$\lan \, , \, \ran_\iota : \D(\la) \times \D(\la) \ra F$ 
by 
\[ 
	\lan \ol{y 1_\la} , \ol{x 1_\la} \ran_\iota 1_\la \equiv \iota(y 1_\la) x 1_\la 
	\mod  \CS_{n,r}(>\la) \quad \text{ for } x,y \in  \CS_{n,r}^-.
\]
Note that $\lan\, , \,\ran_\iota$ is symmetric. 
Put 
$\rad_\iota  \D(\la) =\{ \ol{x} \in \D(\la) \,|\, \lan \ol{y} , \ol{x} \ran_\iota =0 
	\text{ for any } \ol{y} \in \D(\la) \}
$. 
One sees easily that 
$\lan \ol{y} , \ol{x} \ran_\iota = \lan \ol{\iota(y)} , \ol{x} \ran $ 
for $\ol{x},\ol{y} \in \D(\la)$, 
thus we have  
$\rad_\iota  \D(\la) = \rad \D(\la)$. 
Hence, 
from now on  
we denote $\lan \, , \, \ran_\iota $ (resp. $\rad_\iota \D(\la)$) 
simply 
by $\lan \, , \, \ran$ (resp. $\rad \D(\la)$).

\para 
For an $\CS_{n,r}$-module $M$, 
we have the weight space decomposition 
\[ M = \bigoplus_{\mu \in \vL_{n,r}} M_\mu, \]
where 
$M_\mu = 1_\mu \cdot M$. 
Since $\D(\la) = \CS_{n,r}^- \cdot \ol{1_\la}$, 
we see that 
$\la \geq \mu$ 
if 
$\D(\la)_\mu \not=0$. 
It is clear that 
$\D(\la)_\mu $  
is spanned by 
\[ 
\Xi (\la-\mu) = \big\{ F_{(i_1,k_1)}^{(c_1)} F_{(i_2,k_2)}^{(c_2)} \cdots F_{(i_l, k_l)}^{(c_l)} \cdot \ol{1_\la} 
\bigm| c_1 \a_{(i_1,k_1)} + c_2 \a_{(i_2,k_2)} + \cdots + c_l \a_{(i_l, k_l)} = \la - \mu 
\big\}.
\]
Note that $\Xi (\la-\mu)$ is a finite set. 
Then 
we can pick up a homogeneous basis of 
$\D(\la)_\mu$ 
from 
$\Xi (\la-\mu)$. 
We take a homogeneous basis $\CB(\la)_\mu$ of $\D(\la)_\mu$, 
and fix it. 

For $\la \in \vL^+_{n,r},\,  \mu \in \vL_{n,r}$, 
let 
\[ M(\la)_\mu = \Big( \lan\, \ol{b'}, \ol{b}\, \ran \Big)_{\ol{b},\ol{b'} \in \CB(\la)_\mu}\]
be a Gram matrix 
of the weight space $\D(\la)_\mu$. 
Put 
$\rad \D(\la)_\mu = \rad \D(\la) \cap \D(\la)_\mu$, 
then 
we have the following lemma. 

\begin{lem}
\label{lem dim weight sp}
We have 
\[ \dim_F \rad \D(\la)_\mu = \corank M(\la)_\mu. \]
\end{lem}

\begin{proof}
For $\ol{x} \in \D(\la)_\mu$, $\ol{y} \in \D(\la)_\nu$, 
we have 
$\lan \ol{y}, \ol{x} \ran =0$ unless $\mu = \nu$ 
by \eqref{bilinear zero}. 
Thus 
$\ol{x} \in \rad \D(\la)_\mu$ 
if and only if 
$\lan \ol{b'}, \ol{x} \ran=0$ for any $\ol{b'} \in \CB(\la)_\mu$.  
This implies the lemma. 
\end{proof}
\quad \\

\textbf{(Algorithm for computing  decomposition numbers of $\CS_{n,r}$)}
\\[2mm]
\textbf{(step 1)} 
Compute the value of 
$\lan \ol{b'}, \ol{b} \ran $ for all $\ol{b}, \ol{b'} \in \CB(\la)_\mu$ ($\la \in \vL^+_{n,r}, \mu \in \vL_{n,r}$).  
\\[-3mm]

Note that by \eqref{xy constant} and the definition of the bilinear form, 
we can compute $\lan \ol{b'}, \ol{b} \ran $ 
by using  the commutative relation 
\eqref{S-6} repeatedly. 
\\[2mm]
\textbf{(step 2)} 
Compute the corank of $ M(\la)_\mu$ for all $\la \in \vL^+_{n,r}, \mu \in \vL_{n,r}$. 
\\[-3mm] 

This is an elementally calculation of the linear algebra.  
\\[3mm]
\textbf{(step 3)}
Compute 
$\dim_F (L(\la)_\mu)$ 
for all $\la \in \vL^+_{n,r}$, $\mu \in \vL_{n,r}$.  
\\[-3mm] 

Since 
$L(\la)=\D(\la)/ \rad \D(\la)$, 
we have 
\[
\dim_F (L(\la)_\mu) = \dim_F (\D(\la)_\mu) - \dim_F ( \rad \D(\la)_\mu).
\] 
Thus, 
we can compute $\dim_F (L(\la)_\mu)$ by Lemma \ref{lem dim weight sp} and (step 2). 
\\[3mm] 
\textbf{(step 4)} 
Compute the decomposition numbers $d_{\la\mu}=[\D(\la) : L(\mu)]$ for $\la,\mu \in \vL^+_{n,r}$ 
by the following inductive process. 
\\[-3mm] 

By Theorem \ref{thm standard simple Sq R}, we have $d_{\la\la}=1$ for $\la \in \vL^+_{n,r}$. 
By induction, 
we may  assume that   
$d_{\la \mu}$ 
is known
for $\mu \in \vL^+_{n,r}$ such that $\la \geq \mu > \nu$, 
and  we compute 
the decomposition number 
$d_{\la \nu}$.

Note the following four facts:

$\bullet$ 
$\rad \D(\la)$ is the unique maximal $\CS_{n,r}$-submodule of $\D(\la)$,  

$\bullet$ 
$d_{\la \mu} \not=0 $ ($\la \not= \mu$) only if $\la > \mu$. 

$\bullet$ 
$L(\mu)_{\nu} \not=0 $ only if $\mu \geq \nu$. 

$\bullet$ 
$\dim_F L(\nu)_{\nu}=1$. 
\\
These four facts imply that 
\begin{align}
\label{eq rad D la nu}
\dim_F (\rad \D(\la)_{\nu}) 
&=
\sum_{ \mu \in \vL^+_{n,r} \setminus \{\la\}} d_{\la \mu} \cdot (\dim_F L(\mu)_{\nu}) 
\\
&= \sum_{\mu \in \vL^+_{n,r} \atop \la > \mu >\nu} d_{\la \mu} \cdot(\dim_F L(\mu)_{\nu}) + d_{\la \nu}. 
\notag
\end{align}
By Lemma \ref{lem dim weight sp} and (step 2), 
we know $\dim_F (\rad \D(\la)_{\nu})$. 
By the assumption of the induction together with (step 3), 
we know $\sum_{\mu \in \vL^+_{n,r} \atop \la > \mu >\nu} d_{\la \mu} \cdot(\dim_F L(\mu)_{\nu})$.  
Thus we can compute the decomposition number $d_{\la \nu}$ from the equation \eqref{eq rad D la nu}. 
\pagebreak

\remarks\

(\roi) 
In fact, 
in order to compute the decomposition numbers,  
it is enough to consider the Gram matrix 
$M(\la)_\mu$ 
only 
for $\la,\mu \in \vL_{n,r}^+$ 
since 
we have  
\[ 
\dim_F L(\mu)_\nu = \dim_F \D(\mu)_\nu - \sum_{\t \in \vL_{n,r}^+} d_{\mu \t} \dim_F L(\t)_\nu. 
\]
In this case, 
we should skip (step 3), 
and should add the following process of another induction on $\vL_{n,r}^+$ in (step 4) :   
\\[-2em]
\begin{align*} 
&\text{
$d_{\mu \t}$ 
is known 
for 
$\mu, \t \in \vL_{n,r}^+$ 
such that 
$\la > \mu$. 
}
\\
\Leftrightarrow 
&
\text{
$\dim_F L(\mu)_\nu$ 
is known 
for 
$\mu \in \vL_{n,r}^+$, $\nu \in \vL_{n,r}$ 
such that 
$\la > \mu$.
}  
\end{align*}

(\roii) 
Thanks to Theorem \ref{thm standard induced borel R} 
and \cite[Theorem 5.16 (f)]{DR} 
(or directly by comparing the highest weights as $\wt{U}_q$-modules),  
we have $_F \D(\la) \cong \,_F W(\la)$ for $\la \in \vL_{n,r}^+$. 
In particular, 
we have 
$_F \D(\la) = F \otimes_{\CA} \,_\CA \D(\la)$ 
since 
it is known that 
$_F W(\la) = F \otimes_{\CA} \,_\CA W(\la)$. 

(\roiii) 
Our algorithm can be applied for an arbitrary field 
which is not necessarily of characteristic $0$.  


%
%


(\roiv)
There exists a surjective homomorphism 
$_\CA \wt{U}_q^- \ra \,_\CA \CS_q^-$ 
as algebras, 
and we have 
$_\CA \wt{U}_q^- \cong\, _\CA U_q^-$. 
Thus, we have a surjective homomorphism of $_\CA U_q^-$-modules: 
\[ _\CA U_q^- \ra \,_\CA \D(\la)\,\, (= \,_\CA \CS_q^- \cdot \ol{1_\la}) \text{ such that } 1 \mapsto \ol{1_\la}.\]
It maybe useful that we take a homogeneous basis of $_\CA \D(\la)$ 
from the image of a certain homogeneous basis of $_\CA U_q^-$ 
(e.g. monomial basis, PBW basis, canonical basis, $\cdots$). 

(\rov) 
We can apply this algorithm to compute the decomposition numbers of 
$_F \CS_q$ under the general setting in \S \ref{specialization}. 
Moreover, 
we can also apply to  compute the decomposition numbers of 
$_F \CS_q$ associated to any Cartan matrix of finite type, 
which includes  the generalized $q$-Schur algebra constructed in \cite{Do}. 




\appendix 
\section{A proof of Proposition \ref{prop-rho}.}
\label{proof-prop-rho}
In this section, 
we give a proof of Proposition \ref{prop-rho}.  
The author thanks T. Shoji for communicating this fact.  

\para 
Let 
$V$ 
be a vector space over 
$\QQ(q)$ 
with a basis 
$\{v_1,\cdots,v_m\}$. 
Then, 
$U_q=U_q(\Fgl_m)$ 
acts on $V$ from left by 
\begin{align*}
&e_i \cdot v_j = 
	\begin{cases}
		v_{j-1} &\text{ if }j=i+1, \\
		0 &\text{otherwise},
	\end{cases}\\ 
&f_i \cdot v_j =
	\begin{cases}
		v_{j+1} &\text{ if } j=i, \\
		0 &\text{otherwise},
	\end{cases}\\
& K_i^{\pm} \cdot v_j =
	\begin{cases}
		q^{\pm1} v_j &\text{ if } j=i, \\
		v_j &\text{otherwise}.
	\end{cases}
\end{align*}
This action is called a vector representation 
of 
$U_q$.  
We extend this action to a tensor space 
$V^{\otimes n}$ 
by using a comultiplication 
$\D$ 
of 
$U_q$ defined by 
\begin{align*}
&\D(e_i)=e_i \otimes K_i K_{i+1}^- + 1 \otimes e_i, \\
&\D(f_i)=f_i \otimes 1 + K_i^- K_{i+1} \otimes f_i, \\
&\D(K_i^{\pm})=K_i^{\pm} \otimes K_i^{\pm}.
\end{align*}
We denote this action by 
$\rho' : U_q(\Fgl_m) \ra \End(V^{\otimes n})$. 

On the other hand, 
$\He_n$ 
acts on 
$V^{\otimes n}$ 
from right as follows. 
We define 
$\wt{T} \in \End(V \otimes V)^{\text{op}}$ 
by 
\begin{align*}
(v_i \otimes v_j)\cdot \wt{T}=
	\begin{cases}
		q\, v_i \otimes v_j &\text{ if }i=j, \\
		v_j \otimes v_i & \text{ if }i<j, \\
		v_j \otimes v_i + (q-q^{-1}) v_i \otimes v_j &\text{ if }i>j, 
	\end{cases}
\end{align*}
where $\End(V\otimes V)^{\text{op}}$ means an opposite algebra of $\End(V \otimes V)$. 
For $i=1,\cdots,n-1$, 
we define 
$\wt{T}_i \in \End(V^{\otimes n})^{\text{op}}$ by 
\begin{align*}
\wt{T}_i = \id_V^{\otimes (i-1)} \otimes \wt{T} \otimes \id_V^{\otimes (n-1-i)}.
\end{align*}
Then, 
we define 
an algebra homomorphism 
$\theta : \He_n \ra \End(V^{\otimes n})^{\text{op}}$ 
by 
$\theta(T_i)=\wt{T}_i$.  
By \cite{J}, 
it is known that 
the action of $U_q$ and 
the action of $\He_n$ 
on $V^{\otimes n}$ 
commute. 
Moreover, 
we have  
\begin{align*}
\rho'(U_q)=\End_{\He_n}(V^{\otimes n}).
\end{align*}

\para 
For 
$\mu=(\mu_1,\cdots,\mu_m) \in \vL_{n,1}$, 
let 
$V_\mu^{\otimes n}$ 
be 
a subspace of 
$V^{\otimes n}$ 
spanned by 
$\big\{ v_{i_1} \otimes v_{i_2} \otimes \cdots \otimes v_{i_n} 
	\bigm| \mu_j= \sharp \{k\,|\, i_k=j \} \text{ for } j=1,\cdots,m \big\} $. 
One sees easily that 
$V^{\otimes n}_\mu$ 
is a weight space 
of 
$V^{\otimes n}$ 
with a weight $\mu$ 
as a $U_q$-module,  
and  
we have a weight space decomposition 
\begin{align*}
V^{\otimes n}=\bigoplus_{\mu \in \vL_{n,1}} V_\mu^{\otimes n}.
\end{align*}
Since 
the action of $\He_n$ 
commutes with 
the action of 
$U_q$, 
$V^{\otimes n}_\mu$ 
is invariant  
under the action of $\He_n$.  
For $\mu \in \vL_{n,1}$, 
put 
\begin{align*}  
v_\mu = \underbrace{v_1 \otimes \cdots \otimes v_1}_{\mu_1 \text{ terms }} 
	\otimes 
	\underbrace{v_2 \otimes \cdots \otimes v_2}_{\mu_2 \text{ terms }} 
	\otimes \cdots \otimes 
	\underbrace{v_m \otimes \cdots \otimes v_m}_{\mu_m \text{ terms }}.
\end{align*}
Then, 
we have 
$
V^{\otimes n}_\mu = v_\mu \cdot \He_n. 
$
Moreover, 
one can check that 
there exists an isomorphism 
$V^{\otimes n}_\mu \ra M^\mu$ 
of 
$\He_{n}$-modules 
such that 
$v_\mu \mapsto x_\mu$. 
Thus, we have the following isomorphism of algebras. 
\begin{align*}
\rho'(U_q) &= \End_{\He_n}(V^{\otimes n}) \\
&=\End_{\He_n}\Big( \bigoplus_{\mu \in \vL_{n,1}} V_\mu^{\otimes n} \Big) \\
&\cong \End_{\He_n}\Big( \bigoplus_{\mu \in \vL_{n,1}} M^\mu \Big).
\end{align*} 
This isomorphism gives the surjection 
$\rho : U_q \ra \Sc_{n,1}$ 
in Theorem \ref{Uqgl-S1}.

\para 
For 
$\mu \in \vL_{n,1}$, 
put 
\begin{align*}
&A = \underbrace{v_1 \otimes \cdots \otimes v_1}_{\mu_1 \text{ terms }} 
	\otimes 
	\underbrace{v_2 \otimes \cdots \otimes v_2}_{\mu_2 \text{ terms }} 
	\otimes \cdots \otimes 
	\underbrace{v_i \otimes \cdots \otimes v_i}_{\mu_i \text{ terms }}, \\
&B = \underbrace{v_{i+2} \otimes \cdots \otimes v_{i+2}}_{\mu_{i+2} \text{ terms }} 
	\otimes 
	\underbrace{v_{i+3} \otimes \cdots \otimes v_{i+3}}_{\mu_{i+3} \text{ terms }} 
	\otimes \cdots \otimes 
	\underbrace{v_m \otimes \cdots \otimes v_m}_{\mu_m \text{ terms }}.
\end{align*}
Then, we have 
\begin{align*}
&v_\mu = A \otimes \underbrace{v_{i+1} \otimes \cdots \otimes v_{i+1}}_{\mu_{i+1} \text{ terms }} \otimes B, \\
&v_{\mu +\a_i} 
= A \otimes v_i \otimes 
	\underbrace{v_{i+1} \otimes \cdots \otimes v_{i+1}}_{\mu_{i+1} -1 \text{ terms }} \otimes B. 
\end{align*}
By the definitions, 
one can compute that 
\begin{align*}
\rho'(e_i)(v_\mu) 
&= \sum_{j=1}^{\mu_{i+1}} q^{-(\mu_{i+1}-j)}
	A \otimes \underbrace{v_{i+1} \otimes \cdots \otimes v_{i+1}}_{j-1 \text{ terms }} 
	\otimes v_i \otimes 
	\underbrace{v_{i+1} \otimes \cdots \otimes v_{i+1}}_{\mu_{i+1}-j \text{ terms }} \otimes B \\*
&=q^{-\mu_{i+1}+1} \sum_{ x \in X^\mu_{\mu+\a_i}} q^{\ell(x)} v_{(\mu+\a_i)}\cdot  T_x. 
\end{align*}
Under the isomorphism $V^{\otimes n}_\mu \cong M^\mu$, 
this implies that 
$\rho(e_i)(m_\mu)=q^{-\mu_{i+1}+1} \psi_{\mu+\a_{i} \,,\, \mu}^1 (m_\mu)$. 
Thus, 
we have (\roi) in Proposition \ref{prop-rho}. 
For (\roii), (\roiii) in Proposition \ref{prop-rho}, 
we can prove in a similar way. 



\section{Example : Cyclotomic $q$-Schur algebra of type $G(2,1,2)$} 
\label{example G(2,1,2)}

In this appendix, 
we consider a cyclotomic $q$-Schur algebra $\Sc_{2,2}$ of type $G(2,1,2)$, 
namely associated to the complex reflection group $\FS_2 \ltimes (\ZZ/2 \ZZ)^2$. 
In this case, 
we will describe elements $\eta_{(i,k)}^\la$ explicitly, 
and compute the Gram matrices $M(\la)_\mu$ and decomposition numbers 
of $_\CC \Sc_{2,2}$. 
Throughout this appendix, 
we replace $\g_i$ with $Q_i$ ($i=1,2$), 
thus 
$\Sc_{2,2}$ is an algebra over 
$\CK=(q,Q_1,Q_2)$, 
where $q,Q_1,Q_2$ are indeterminate elements.

\para  
The cyclotomic $q$-Schur algebra 
$\Sc_{2,2}$ of type $G(2,1,2)$ 
is generated by 
the generators 
$E_{(1,1)}, E_{(2,1)}, E_{(1,2)}, F_{(1,1)}, F_{(2,1)}, F_{(1,2)}, 1_\la (\la \in \vL)$, 
where  
\[ 
\vL=
\left\{ 
\begin{matrix}
\la_{\lan 0 \ran} =\big( (2,0),(0,0) \big), &
\la_{\lan 1 \ran} =\big( (1,1),(0,0) \big), &
\la_{\lan 2 \ran} =\big( (1,0),(1,0) \big), \\
\la_{\lan 3 \ran} =\big( (1,0),(0,1) \big), &
\la_{\lan 4 \ran} =\big( (0,2),(0,0) \big), &
\la_{\lan 5 \ran} =\big( (0,1),(1,0) \big), \\
\la_{\lan 6 \ran} =\big( (0,1),(0,1) \big), &
\la_{\lan 7 \ran} =\big( (0,0),(2,0) \big), &
\la_{\lan 8 \ran} =\big( (0,0),(1,1) \big), \\
\la_{\lan 9 \ran} =\big( (0,0),(0,2) \big)
\end{matrix}
\right\}, 
\]  
with the defining relations 
\eqref{S-1} - \eqref{S-8}. 
By Lemma \ref{lem vLnr r-partition}, 
we have 
\[\vL^+=\{\la_{\lan 0 \ran} ,\la_{\lan 1 \ran} ,\la_{\lan 2 \ran} ,\la_{\lan 7 \ran} , \la_{\lan 8 \ran} \}.\] 

By Lemma \ref{def g-mu-(i,k)} and \eqref{def g FE}, 
we have 
\begin{align*}
&
g_{(2,1)}^{\la_{\lan 1 \ran} } (F,E) = Q_1\big( (q-q^{-1}) F_{(1,1)} E_{(1,1)} + q^{-2} \big),   
\\
&
g_{(2,1)}^{\la_{\lan 4 \ran} } (F,E) = Q_1 (q^2 +1), 
\\
&
g_{(2,1)}^{\la_{\lan 5 \ran} } (F,E) = Q_1,
\\
&
g_{(2,1)}^{\la_{\lan 6 \ran} } (F,E) = Q_1,
\\
& 
g_{(1,2)}^{\la_{\lan 2 \ran} } (F,E) = F_{(2,1)} E_{(2,1)} + Q_2, 
\\
&
g_{(1,2)}^{\la_{\lan 5 \ran} } (F,E) = F_{(1,1)} F_{(2,1)} E_{(2,1)} E_{(1,1)} + Q_2,
\\
&
g_{(1,2)}^{\la_{\lan 7 \ran} } (F,E) = q F_{(2,1)} E_{(2,1)} + Q_2 (1+q^2),
\\
&
g_{(1,2)}^{\la_{\lan 8 \ran} } (F,E) = F_{(2,1)} E_{(2,1)} + Q_2,
\end{align*}
and 
$g_{(2,1)}^\la(F,E)$ (resp. $g_{(1,2)}^\la(F,E)$), 
which does not appear in the above list, 
is equal to $0$. 
\\
As an example, we compute only $g_{(2,1)}^{\la_{\lan 1 \ran}} (F,E)$. 
By the definitions, 
we have 
\begin{align*}
\s_{(2,1)}^{\la_{\lan 1 \ran}} (m_{\la_{\lan 1 \ran}}) 
&= m_{\la_{\lan 1 \ran}} L_2 
\\
&=
(L_1 - Q_2)(L_2-Q_2) T_1 L_1 T_1 
\\
&= 
T_1 (L_1 - Q_2)L_1(L_2 - Q_2) T_1  
	\qquad (\because \text{ Lemma } \ref{lem-L} \, (\text{\roi}), \, (\text{\roiv}))\\
&= Q_1 T_1 (L_1 - Q_2)(L_2 - Q_2) T_1 
\\
&=
Q_1 (L_1 - Q_2)(L_2 - Q_2) \big( (q-q^{-1}) T_1 + 1 \big) 
	\qquad (\because  T_1^2= (q-q^{-1}) T_1 +1 ) 
\\
&=
Q_1 \big( (q-q^{-1}) m_{\la_{\lan 1 \ran}} T_1 + m_{\la_{\lan 1 \ran}} \big), 
\end{align*}
where the fourth equality follows from $L_1=T_0$ and $T_0^2 = (Q_1 + Q_2) T_0 - Q_1 Q_2$.
On the other hand, 
we have 
\begin{align*}
\vf_{(1,1)}^- \vf_{(1,1)}^+ (m_{\la_{\lan 1 \ran}}) 
&= 
q^{-1} m_{\la_{\lan 1 \ran}} (1 + q T_1) 
\\
&=
m_{\la_{\lan 1 \ran}} T_1 + q^{-1} m_{\la_{\lan 1 \ran}}.
\end{align*}
Thus, we have 
$
\s_{(2,1)}^{\la_{\lan 1 \ran}} = Q_1\big( (q-q^{-1}) \vf^-_{(1,1)} \vf^+_{(1,1)} + q^{-2} \big) 
	\vf_{\la_{\lan 1 \ran}, \la_{\lan 1 \ran}}^1
$. 
This implies that 
\[g_{(2,1)}^{\la_{\lan 1 \ran}} (F,E) = Q_1\big( (q-q^{-1}) F_{(1,1)} E_{(1,1)} + q^{-2} \big).\]

Since 
$\eta_{(2,1)}^\la = \Big( - Q_2 [ \la_2^{(1)} - \la_1^{(2)}] + q^{\la_2^{(1)} - \la_1^{(2)}} 
	\big( q^{-1} g _{(2,1)}^\la (F,E) - q g_{(1,2)}^\la (F,E) \big) \Big) 1_\la$, 
we have 
\begin{align*}
&
\eta_{(2,1)}^{\la_{\lan 1 \ran}} 
	=  \Big( Q_1 (q-q^{-1}) F_{(1,1)} E_{(1,1)} + (Q_1 q^{-2} - Q_2) \Big) 1_{\la_{\lan 1 \ran}},
\\
&
\eta_{(2,1)}^{\la_{\lan 2 \ran}} =  - F_{(2,1)} E_{(2,1)} 1_{\la_{\lan 2 \ran}},
\\
&
\eta_{(2,1)}^{\la_{\lan 4 \ran}} = \Big( Q_1 (q^3 + q) - Q_2(q+q^{-1})\Big)  1_{\la_{\lan 4 \ran}},
\\
&
\eta_{(2,1)}^{\la_{\lan 5 \ran}} 
	=  \Big( -q F_{(1,1)} F_{(2,1)} E_{(2,1)} E_{(1,1)} + (Q_1 q^{-1} - Q_2 q) \Big)  1_{\la_{\lan 5 \ran}},
\\
&
\eta_{(2,1)}^{\la_{\lan 6 \ran}} =  (Q_1 - Q_2 ) 1_{\la_{\lan 6 \ran}},
\\
&
\eta_{(2,1)}^{\la_{\lan 7 \ran}} =  - F_{(2,1)} E_{(2,1)}  1_{\la_{\lan 7 \ran}},
\\
&
\eta_{(2,1)}^{\la_{\lan 8 \ran}} =  - F_{(2,1)} E_{(2,1)}  1_{\la_{\lan 8 \ran}},
\\
& 
\eta_{(2,1)}^{\la_{\lan 0 \ran}} = \eta_{(2,1)}^{\la_{\lan 3 \ran}} = \eta_{(2,1)}^{\la_{\lan 9 \ran}} = 0.
\end{align*}

\para 
We can take a homogeneous basis of $_\CA \D(\la)$ for $\la \in \vL^+$ as followings. 
\\[2mm]
{\scriptsize 
\begin{tabular}{|c|r|}
\hline 
\multicolumn{2}{|c|}{} \\[-3mm]
\multicolumn{2}{|c|}{basis of $_\CA \D(\la_{\lan 0 \ran}) $}
\\[2mm] \hline
weight & basis 
\\ \hline & \\[-3mm]
$\la_{\lan 0 \ran} $ & $\ol{  1_{\la_{\lan 0 \ran}}}$
\\ \hline & \\[-3mm]
$\la_{\lan 1 \ran} $ & $\ol{ F_{(1,1)} 1_{\la_{\lan 0 \ran}}}$
\\ \hline & \\[-3mm]
$\la_{\lan 2 \ran} $ & $\ol{  F_{(2,1)} F_{(1,1)} 1_{\la_{\lan 0 \ran}}}$
\\ \hline & \\[-3mm]
$\la_{\lan 3 \ran} $ & $\ol{  F_{(1,2)} F_{(2,1)} F_{(1,1)} 1_{\la_{\lan 0 \ran}}}$
\\ \hline & \\[-3mm]
$\la_{\lan 4 \ran} $ & $\ol{  F_{(1,1)}^{(2)} 1_{\la_{\lan 0 \ran}}}$
\\ \hline & \\[-3mm]
$\la_{\lan 5 \ran} $ & $\ol{  F_{(2,1)} F_{(1,1)}^{(2)} 1_{\la_{\lan 0 \ran}}}$
\\ \hline & \\[-3mm]
$\la_{\lan 6 \ran} $ & $\ol{  F_{(1,2)} F_{(2,1)} F_{(1,1)}^{(2)} 1_{\la_{\lan 0 \ran}}}$
\\ \hline & \\[-3mm]
$\la_{\lan 7 \ran} $ & $\ol{  F_{(2,1)}^{(2)} F_{(1,1)}^{(2)} 1_{\la_{\lan 0 \ran}}}$
\\ \hline & \\[-3mm]
$\la_{\lan 8 \ran} $ & $\ol{  F_{(1,2)} F_{(2,1)}^{(2)} F_{(1,1)}^{(2)} 1_{\la_{\lan 0 \ran}}}$
\\ \hline & \\[-3mm]
$\la_{\lan 9 \ran} $ & $\ol{  F_{(1,2)}^{(2)} F_{(2,1)}^{(2)} F_{(1,1)}^{(2)} 1_{\la_{\lan 0 \ran}}}$
\\ \hline 
\end{tabular}
}
\quad 
{\small 
\begin{tabular}{|c|r|}
\hline
\multicolumn{2}{|c|}{} \\[-3mm]
\multicolumn{2}{|c|}{basis of $_\CA \D(\la_{\lan 1 \ran}) $}
\\[2mm] \hline
weight & basis 
\\ \hline & \\[-3mm]
$\la_{\lan 1 \ran} $ & $\ol{  1_{\la_{\lan 1 \ran}}}$
\\ \hline & \\[-3mm]
$\la_{\lan 2 \ran} $ & $\ol{  F_{(2,1)}  1_{\la_{\lan 1 \ran}}}$
\\ \hline & \\[-3mm]
$\la_{\lan 3 \ran} $ & $\ol{  F_{(1,2)} F_{(2,1)}  1_{\la_{\lan 1 \ran}}}$
\\ \hline & \\[-3mm]
$\la_{\lan 5 \ran} $ & $\ol{  F_{(1,1)} F_{(2,1)} 1_{\la_{\lan 1 \ran}}}$
\\ \hline & \\[-3mm]
$\la_{\lan 6 \ran} $ & $\ol{  F_{(1,2)} F_{(1,1)} F_{(2,1)}  1_{\la_{\lan 1 \ran}}}$
\\ \hline & \\[-3mm]
$\la_{\lan 8 \ran} $ & $\ol{  F_{(2,1)} F_{(1,2)} F_{(1,1)} F_{(2,1)} 1_{\la_{\lan 1 \ran}}}$
\\ \hline 
\end{tabular}}

{\small 
\begin{tabular}{|c|r|}
\hline 
\multicolumn{2}{|c|}{} \\[-3mm]
\multicolumn{2}{|c|}{basis of $_\CA \D(\la_{\lan 2 \ran}) $}
\\[2mm] \hline
weight & basis 
\\ \hline & \\[-3mm]
$\la_{\lan 2 \ran} $ & $\ol{  1_{\la_{\lan 2 \ran}}}$
\\ \hline & \\[-3mm]
$\la_{\lan 3 \ran} $ & $\ol{  F_{(1,2)}  1_{\la_{\lan 2 \ran}}}$
\\ \hline & \\[-3mm]
$\la_{\lan 5 \ran} $ & $\ol{  F_{(1,1)} 1_{\la_{\lan 2 \ran}}}$
\\ \hline & \\[-3mm]
$\la_{\lan 6 \ran} $ & $\ol{  F_{(1,2)} F_{(1,1)} 1_{\la_{\lan 2 \ran}}}$
\\ \hline & \\[-3mm]
$\la_{\lan 7 \ran} $ & $\ol{  F_{(2,1)} F_{(1,1)} 1_{\la_{\lan 2 \ran}}}$
\\ \hline & \\[-3mm]
$\la_{\lan 8 \ran} $ & $\ol{  F_{(2,1)} F_{(1,2)} F_{(1,1)} 1_{\la_{\lan 2 \ran}}}$, 
						$\ol{  F_{(1,2)} F_{(2,1)} F_{(1,1)} 1_{\la_{\lan 2 \ran}}}$
\\ \hline & \\[-3mm]
$\la_{\lan 9 \ran} $ & $\ol{  F_{(1,2)} F_{(2,1)} F_{(1,2)} F_{(1,1)} 1_{\la_{\lan 2 \ran}}}$
\\ \hline 
\end{tabular}
} 
\quad 
{\small 
\begin{tabular}{|c|r|}
\hline 
\multicolumn{2}{|c|}{} \\[-3mm]
\multicolumn{2}{|c|}{basis of $_\CA \D(\la_{\lan 7 \ran}) $}
\\[2mm] \hline
weight & basis 
\\ \hline & \\[-3mm]
$\la_{\lan 7 \ran} $ & $\ol{   1_{\la_{\lan 7 \ran}}}$
\\ \hline & \\[-3mm]
$\la_{\lan 8 \ran} $ & $\ol{  F_{(1,2)}  1_{\la_{\lan 7 \ran}}}$, 
\\ \hline & \\[-3mm]
$\la_{\lan 9 \ran} $ & $\ol{  F_{(1,2)}^{(2)} 1_{\la_{\lan 7 \ran}}}$
\\ \hline \multicolumn{2}{|c|}{} \\[-3mm] \hline
\multicolumn{2}{|c|}{} \\[-3mm]
\multicolumn{2}{|c|}{basis of $_\CA \D(\la_{\lan 8 \ran}) $}
\\[2mm] \hline
weight & basis 
\\ \hline & \\[-3mm]
$\la_{\lan 8 \ran} $ & $\ol{   1_{\la_{\lan 8 \ran}}}$
\\\hline
\end{tabular}
}

\para 
We can compute the Gram matrix of $_\CA \D(\la)_\mu$ $\la,\mu \in \vL^+$ 
with respect to the above basis. 
Here, as an example, 
we compute 
$M(\la_{\lan 0 \ran})_{\la_{\lan 2 \ran}}$. 
Note that 
$_\CA \D(\la_{\lan 0 \ran})_{ \lan 2 \ran}$ 
has a basis 
$\big\{ \ol{F_{(2,1)} F_{(1,1)} 1_{\la_{\lan 0 \ran}}} \big\}$. 
We have 
\begin{align*}
&1_{\la_{\lan 0 \ran}} E_{(1,1)} E_{(2,1)} F_{(2,1)} F_{(1,1)} 1_{\la_{\lan 0 \ran}} 
\\
&= 
E_{(1,1)}\Big( Q_1 (q-q^{-1}) F_{(1,1)} E_{(1,1)} + (Q_1 q^{-2} - Q_2)\Big) F_{(1,1)} 1_{\la_{\lan 0 \ran}} 
\\
&= 
\Big(Q_1 (q-q^{-1}) [2][2] + (Q_1 q^{-2} -Q_2) [2] \Big) 1_{\la_{\lan 0 \ran}} 
\qquad ( \because E_{(1,1)} F_{(1,1)} 1_{\la_{\lan 0 \ran}}  = [2] 1_{\la_{\lan 0 \ran}} )
\\
&= 
[2](Q_1 q^2 - Q_2) 1_{\la_{\lan 0 \ran}}.
\end{align*}
This implies that 
$\big\lan \ol{F_{(2,1)} F_{(1,1)} 1_{\la_{\lan 0 \ran}}}\, , \,
	\ol{F_{(2,1)} F_{(1,1)} 1_{\la_{\lan 0 \ran}}} \big\ran 
= [2](Q_1 q^2 - Q_2)$. 
Thus, we have 
$ M(\la_{\lan 0 \ran})_{\la_{\lan 2 \ran}} = \Big( [2](Q_1 q^2 - Q_2) \Big)$. 

In a similar way, 
we can compute the Gram matrix $M(\la)_\mu$ for $\la,\mu \in \vL_{n,r}^+$, 
and we have 

$\D(\la_{\lan 0 \ran})$ ; 
\\[-3em]
\begin{align*}
&M(\la_{\lan 0 \ran})_{\la_{\lan 1 \ran }} = \Big( [2] \Big)
\\
&M(\la_{\lan 0 \ran})_{\la_{\lan 2 \ran }} = \Big( [2] (q^2 Q_1 - Q_2 ) \Big)
\\
&M(\la_{\lan 0 \ran})_{\la_{\lan 7 \ran }} = \Big( (Q_1 - Q_2) (q^2 Q_1 - Q_2) \Big)
\\
&M(\la_{\lan 0 \ran})_{\la_{\lan 8 \ran }} = \Big( [2] (Q_1 - Q_2)(q^2 Q_1 -Q_2) \Big)
\end{align*}

$\D(\la_{\lan 1 \ran})$ ; 
\\[-3em] 
\begin{align*}
&M(\la_{\lan 1 \ran})_{\la_{\lan 2 \ran}} = \Big( (q^{-2} Q_1 - Q_2) \Big)
\\
&M(\la_{\lan 1 \ran})_{\la_{\lan 8 \ran}} = \Big( (Q_1 - Q_2)(q^{-2} Q_1 - Q_2) \Big)
\end{align*}
\\[3em] 

$\D(\la_{\lan 2 \ran})$ ; 
\\[-3em]
\begin{align*}
&M(\la_{\lan 2 \ran})_{\la_{\lan 7 \ran}} = \Big( q(q^{-2} Q_1 - Q_2) \Big)
\\
&M(\la_{\lan 2 \ran})_{\la_{\lan 8 \ran}} = \left(   \begin{matrix}  
													(Q_1 - Q_2) & q(q^{-2} Q_1 - Q_2) \\
													q(q^{-2} Q_1 - Q_2 ) & [2] q(q^{-2} Q_1 - Q_2) 				
													\end{matrix}\right)
\\
&\quad 
\Big( \det  M(\la_{\lan 2 \ran})_{\la_{\lan 8 \ran}} =(q^{-2} Q_1 - Q_2) (q^2 Q_1 - Q_2) \Big)
\end{align*}

$\D(\la_{\lan 7 \ran})$ ; 
\\[-3em]
\begin{align*}
M(\la_{\lan 7 \ran})_{\la_{\lan 8 \ran}} = \Big( [2] \Big) 
\end{align*}
\\[-3em]

\para 
Let 
$\CA \ra \CC$ 
be a ring homomorphism, 
and we express 
the image of $q, Q_1,Q_2$ 
in 
$\CC$ 
by the same symbol. 
We can compute the decomposition numbers of $_\CC \Sc_{2,2} = \CC \otimes_{\CA} \, _\CA \Sc_{2,2}$ 
by using the algorithm in \S \ref{section algorithm}, 
and we have the following decomposition matrix of $_\CC \Sc_{2,2}$.
\\

{\scriptsize
\begin{tabular}{c|ccccc}
\multicolumn{6}{l}{
($q^{2}\not= \pm 1, 0$,\,\, $Q_1=Q_2 \not=0$)
}
\\ \hline 
\\[-3mm] 
$_{\D(\la)} \backslash^{ L(\mu)}$ & $\la_8$ & $\la_7$ & $\la_2$ & $\la_1$ & $\la_0$
\\ \hline 
$\la_8$ & 1 &  &  &  &  
\\ 
$\la_7$ & 0 & 1 &  &  &  
\\
$\la_2$ & 0 & 0 & 1 &  &  
\\
$\la_1$ & 1 & 0 & 0 & 1 &  
\\
$\la_0$ & 0 & 1 & 0 & 0 & 1 
\end{tabular}
\quad 
\begin{tabular}{c|ccccc}
\multicolumn{6}{l}{
($q^{2}\not= \pm 1, 0$,\,\, $q^{-2} Q_1=Q_2 \not=0$)
}
\\ \hline 
\\[-3mm] 
$_{\D(\la)} \backslash^{ L(\mu)}$ & $\la_8$ & $\la_7$ & $\la_2$ & $\la_1$ & $\la_0$
\\ \hline 
$\la_8$ & 1 &  &  &  &  
\\ 
$\la_7$ & 0 & 1 &  &  &  
\\
$\la_2$ & 0 & 1 & 1 &  &  
\\
$\la_1$ & 0 & 0 & 1 & 1 &  
\\
$\la_0$ & 0 & 0 & 0 & 0 & 1 
\end{tabular}
\\[3mm]

\begin{tabular}{c|ccccc}
\multicolumn{6}{l}{
($q^{2}\not= \pm 1, 0$,\,\, $q^{2} Q_1=Q_2 \not=0$)
}
\\ \hline 
\\[-3mm] 
$_{\D(\la)} \backslash^{ L(\mu)}$ & $\la_8$ & $\la_7$ & $\la_2$ & $\la_1$ & $\la_0$
\\ \hline 
$\la_8$ & 1 &  &  &  &  
\\ 
$\la_7$ & 0 & 1 &  &  &  
\\
$\la_2$ & 1 & 0 & 1 &  &  
\\
$\la_1$ & 0 & 0 & 0 & 1 &  
\\
$\la_0$ & 0 & 0 & 1 & 0 & 1 
\end{tabular}
\quad
\begin{tabular}{c|ccccc}
\multicolumn{6}{l}{
($q^{2}= -1 $,\,\, $\pm Q_1 \not= Q_2 $)
}
\\ \hline 
\\[-3mm] 
$_{\D(\la)} \backslash^{ L(\mu)}$ & $\la_8$ & $\la_7$ & $\la_2$ & $\la_1$ & $\la_0$
\\ \hline 
$\la_8$ & 1 &  &  &  &  
\\ 
$\la_7$ & 1 & 1 &  &  &  
\\
$\la_2$ & 0 & 0 & 1 &  &  
\\
$\la_1$ & 0 & 0 & 0 & 1 &  
\\
$\la_0$ & 0 & 0 & 0 & 1 & 1 
\end{tabular}
\\[3mm]

\begin{tabular}{c|ccccc}
\multicolumn{6}{l}{
($q^{2}= -1 $,\,\, $ Q_1 = Q_2 \not=0 $)
}
\\ \hline 
\\[-3mm] 
$_{\D(\la)} \backslash^{ L(\mu)}$ & $\la_8$ & $\la_7$ & $\la_2$ & $\la_1$ & $\la_0$
\\ \hline 
$\la_8$ & 1 &  &  &  &  
\\ 
$\la_7$ & 1 & 1 &  &  &  
\\
$\la_2$ & 0 & 0 & 1 &  &  
\\
$\la_1$ & 1 & 0 & 0 & 1 &  
\\
$\la_0$ & 1 & 1 & 0 & 1 & 1 
\end{tabular}
\quad
\begin{tabular}{c|ccccc}
\multicolumn{6}{l}{
($q^{2}= -1 $,\,\, $ -Q_1 = Q_2 \not=0 $)
}
\\ \hline 
\\[-3mm] 
$_{\D(\la)} \backslash^{ L(\mu)}$ & $\la_8$ & $\la_7$ & $\la_2$ & $\la_1$ & $\la_0$
\\ \hline 
$\la_8$ & 1 &  &  &  &  
\\ 
$\la_7$ & 1 & 1 &  &  &  
\\
$\la_2$ & 0 & 1 & 1 &  &  
\\
$\la_1$ & 0 & 0 & 1 & 1 &  
\\
$\la_0$ & 0 & 1 & 1 & 1 & 1 
\end{tabular}
\\[3mm]

\begin{tabular}{c|ccccc}
\multicolumn{6}{l}{
($q^{2}= 1 $,\,\, $ Q_1 = Q_2 =0 $)
}
\\ \hline 
\\[-3mm] 
$_{\D(\la)} \backslash^{ L(\mu)}$ & $\la_8$ & $\la_7$ & $\la_2$ & $\la_1$ & $\la_0$
\\ \hline 
$\la_8$ & 1 &  &  &  &  
\\ 
$\la_7$ & 0 & 1 &  &  &  
\\
$\la_2$ & 1 & 1 & 1 &  &  
\\
$\la_1$ & 1 & 0 & 1 & 1 &  
\\
$\la_0$ & 0 & 1 & 1 & 0 & 1 
\end{tabular}
\quad
\begin{tabular}{c|ccccc}
\multicolumn{6}{l}{
($q^{2} \not= -1,0  $,\,\, $ Q_1 = Q_2 =0 $)
}
\\ \hline 
\\[-3mm] 
$_{\D(\la)} \backslash^{ L(\mu)}$ & $\la_8$ & $\la_7$ & $\la_2$ & $\la_1$ & $\la_0$
\\ \hline 
$\la_8$ & 1 &  &  &  &  
\\ 
$\la_7$ & 0 & 1 &  &  &  
\\
$\la_2$ & 1 & 1 & 1 &  &  
\\
$\la_1$ & 1 & 0 & 1 & 1 &  
\\
$\la_0$ & 0 & 1 & 1 & 0 & 1 
\end{tabular}
\\[3mm]

\begin{tabular}{c|ccccc}
\multicolumn{6}{l}{
($q^{2}= - 1 $,\,\, $ Q_1 = Q_2 =0 $)
}
\\ \hline 
\\[-3mm] 
$_{\D(\la)} \backslash^{ L(\mu)}$ & $\la_8$ & $\la_7$ & $\la_2$ & $\la_1$ & $\la_0$
\\ \hline 
$\la_8$ & 1 &  &  &  &  
\\ 
$\la_7$ & 1 & 1 &  &  &  
\\
$\la_2$ & 2 & 1 & 1 &  &  
\\
$\la_1$ & 1 & 0 & 1 & 1 &  
\\
$\la_0$ & 1 & 1 & 1 & 1 & 1 
\end{tabular}
}




\section{Example : The case of $\eta_i^\la=0$}
\label{example eta=0}
In this appendix, 
we give an extreme example of $\CS_q$ which is not a cyclotomic $q$-Schur algebra. 
\para 
We take 
$\CK=\QQ(q)$. 
Put 
$\vL=\{ \la=(\la_1, \cdots, \la_m) \in \ZZ^m_{\geq 0} \,|\, \la_1+\cdots + \la_m=n \}$, 
and 
$\eta_i^\la =0$ for any $i=1,\cdots,m-1$ and $\la \in \vL$. 
Then, 
$\CS_q= \CS_q^{\eta_{\vL}}$ is the algebra generated by 
$E_i, F_i$ ($1 \leq i \leq m-1$) and $1_\la $ ($\la \in \vL$) 
with the defining relations 
\eqref{Sq-1}-\eqref{Sq-6}, \eqref{Sq-8}, \eqref{Sq-9} 
together with the relation 
\begin{align*}
\tag{2.1.7'}
E_i F_j - F_j E_i=0.
\end{align*} 

In this case, 
one sees easily that 
$\vL=\vL^+$.  
We denote 
a monomial of $F_i$ (resp. $E_i$) for $i=1,\cdots,m-1$ by 
$X(F)$ (resp. $Y(E)$). 
Then, 
one sees that 
\[ X(F) 1_\la \not\in \CS_q(>\la), \qquad \text{\big(resp. } 1_\la Y(E) \not\in \CS_q(>\la) \text{\big)} \] 
if $\la + \deg(X(F))  \in \vL$ (resp. $  \la - \deg(Y(E))  \in \vL$). 
On the other hand, we have 
\begin{align*}
X(F) \, 1_\la \, Y(E)
&= X(F) Y(E) 1_{\la - \deg(Y(E))} 
\\
&= 
Y(E) X(F) 1_{\la - \deg(Y(E))} 
\\
&= 
Y(E) 1_{\la - \deg(Y(E)) + \deg(X(F))} X(F).
\end{align*}
Thus, 
we have 
$X(F) \, 1_\la \, Y(E) =0 $ 
if $ \la - \deg(Y(E)) + \deg(X(F))  \not\in \vL$. 
It happens that 
$\la + \deg(X(F)) \in \vL$, $\la - \deg(Y(E)) \in \vL$ 
and 
$\la - \deg(Y(E)) + \deg(X(F)) \not\in \vL$. 
This shows that 
the natural surjection 
$\D(\la) \otimes_{\CK} \D^\sharp(\la) \ra \CS_q( \geq \la) / \CS_q(>\la)$ 
is not an isomorphism in general. (Note that (C-2) $\Leftrightarrow $ (C'-2).)  

For $\la,\mu \in \vL^+ (=\vL)$, 
one sees that 
\[ M(\la)_\mu =0 \quad \text{ unless } \la=\mu,\]
where $0$ means the zero-matrix. 
This implies that 
$\dim_\CK L(\la)_\mu =0$ unless $\la =\mu$, 
and that 
\[ [\D(\la) : L(\mu)]= \dim_\CK \D(\la)_\mu.\] 



\begin{thebibliography}{DJM10}
\bibitem[AK]{AK94}
S.~Ariki and K.~Koike,  
\newblock A Hecke algebra of $(\ZZ/r\ZZ)\wr \FS_n$ and construction of its irreducible representations,  
\newblock {\em Adv. Math.} {\bf 106} (1994), 216-243.

\bibitem[DDPW]{DDPW}
B.~Deng, J.~Du, B.~Parshall and J.~Wang,  
\newblock {\em Finite Dimensional Algebras and Quantum Groups}, 
{\em Mathematical Surveys and Monographs} Vol.{\bf 150}, 
\newblock Amer. Math. Soc. 2008.

\bibitem[DJM]{DJM98}
R.~Dipper, G.~James, and A.~Mathas, 
\newblock Cyclotomic {$q$}-{S}chur algebras, 
\newblock {\em Math. Z.} {\bf 229} (1998), 385-416.

\bibitem[Do]{Do}
S.~Doty, 
\newblock Presenting generalized $q$-Schur algebras, 
\newblock {\em Representation theory} {\bf 7} (2003), 196-213. 

\bibitem[Du]{Du}
J.~Du, 
\newblock A note on quantized Weyl reciprocity at root of unity, 
\newblock{\em Algebra Colloq. } {\bf 2} (1995), 363-372. 


\bibitem[DG]{DG}
S.~Doty and A.~Giaquinto,  
\newblock Presenting Schur Algebras, 
\newblock{\em International Mathematical Research Notices} {\bf 36} (2002) 1907-1944.


\bibitem[DP]{DP}
J.~Du and B.~Parshall, 
\newblock Monomial bases for $q$-Schur algebras, 
\newblock{\em Trans. Amer. Math. Soc.} {\bf 355} (2003), 1593-1620. 

\bibitem[DR1]{DR98}
J.~Du and H.~Rui,  
\newblock Based algebras and standard bases for quasi-hereditary algebras, 
\newblock {\em Trans. Amer. Math. Soc.} {\bf350} (1998), 3207-3235.

\bibitem[DR2]{DR}
J.~Du and H.~Rui,  
\newblock  Borel type subalgebras of the $q$-Schur$^m$ algebra,  
\newblock {\em J. Algebra} {\bf 213} (1999), 567-595. 

\bibitem[GL]{GL96}
J.~J. Graham and G.~I. Lehrer, 
\newblock Cellular algebras, 
\newblock {\em Invent. Math.} {\bf 123} (1996), 1-34.

\bibitem[J]{J} 
M.~Jimbo, 
\newblock A q-analogue of $U(\mathfrak{gl}(N+1))$, Hecke algebra and the Yang-Baxter equation, 
\newblock {\em Lett. Math. Phys.}  {\bf 11} (1986), 247-252.  

\bibitem[KX]{KX98}
S.~K\"{o}nig and C.C.~ Xi,   
\newblock On the structure of cellular algebras,  
\newblock {\em Canadian Math. Soc. Conference Proceedings.} {\bf24} (1998), 365-386. 

\bibitem[M]{M-book}
A.~Mathas, 
\newblock {\em Iwahori-{H}ecke algebras and {S}chur algebras of the symmetric
  group}, {\em University Lecture Series} Vol.{\bf 15}, 
\newblock Amer. Math. Soc. 1999.

\bibitem[PW]{PW}
B.~Parshall and J.-P. Wang, 
\newblock{\em Quantum linear groups}, 
{\em Mem. Amer. Math. Soc. } vol.{\bf 89}, No. 439, (1991).

\bibitem[SakS]{SakS}
M.~Sakamoto and T.~Shoji, 
\newblock{Schur-Weyl reciprocity for Ariki-Koike algebras}, 
\newblock{\em J. Algebra} {\bf 221} (1999), 293-314. 

\bibitem[Saw]{Saw}
N.~Sawada. 
\newblock On decomposition numbers of the cyclotomic {$q$}-{S}chur algebras, 
\newblock {\em J. Algebra} {\bf311} (2007), 147--177. 


\bibitem[SawS]{SawS}
N.~Sawada and T.~Shoji, 
\newblock{Modified Ariki-Koike algebras and cyclotomic $q$-Schur algebras}, 
\newblock{\em Math. Z.} {\bf 249} (2005), 829-867.

\bibitem[S1]{S1}
T.~Shoji 
\newblock{A Frobenius formula for the characters of Ariki-Koike algebras}, 
\newblock{\em J. Algebra} {\bf 226} (2000), 818-856.

\bibitem[S2]{S2}
T.~Shoji,  
\newblock { private communication.} 

\bibitem[SW]{SW}
T.~Shoji and K.~Wada, 
\newblock{Cyclotomic q-Schur algebras associated to the Ariki-Koike algebra}, 
preprint, arXiv:0707.1733.

\end{thebibliography}
\end{document}